\documentclass[reqno,11pt,a4paper]{amsart}
\usepackage{amssymb,amsmath,amsfonts,amssymb}%numbysec}
\usepackage{hyperref}
\usepackage{enumerate}
\usepackage{dsfont}
\usepackage{verbatim}
-\marginparwidth \oddsidemargin -9mm \evensidemargin \oddsidemargin
\usepackage{latexsym}
\advance\hoffset by 5mm

\def\@abssec#1{\vspace{.05in}\footnotesize \parindent .2in
{\bf #1. }\ignorespaces}

\newtheorem{theorem}{Theorem}[section]

\newtheorem{lemma}[theorem]{Lemma}
\newtheorem{proposition}[theorem]{Proposition}

\newtheorem{definition}[theorem]{Definition}
\newtheorem{remark}[theorem]{Remark}

\DeclareMathOperator{\divg}{div}

\newcommand{\R}{\ensuremath{\mathbb{R}}}

\newcommand{\WW}{\ensuremath{\mathcal{W}}}
\newcommand{\BB}{\ensuremath{\widetilde{\mathcal{B}}}}
\newcommand{\bb}{\ensuremath{\mathcal{B}}}
\newcommand{\cc}{\ensuremath{\mathcal{C}}}

\newcommand{\Z}{\ensuremath{\mathbb{Z}}}
\newcommand{\N}{\ensuremath{\mathbb{N}}}

\newcommand{\dd}{\mathrm{d}}
\newcommand{\DD}{\ensuremath{\mathbb{R}}^3}
\newcommand{\PPr}{\mathrm{Pr}}
\allowdisplaybreaks \numberwithin{equation}{section}

\begin{document}
\title[Temperature patches for 3D Boussinesq system with large Prandtl number]
{Global regularity of temperature patches for the 3D non-diffusive Boussinesq system with large Prandtl number}
\author{Qianyun Miao}
\address{School of Mathematics and Statistics, Beijing Institute of Technology, Beijing 100081, P. R. China}
\email{qianyunm@bit.edu.cn}
\author{Jiakun Yang}
\address{School of Mathematical Sciences, Laboratory of Mathematics and Complex Systems (MOE), 
Beijing Normal University, Beijing 100875, P.R. China}
\email{yangjk@mail.bnu.edu.cn}

\date{\today}
\thanks{Q. Miao is supported by National Key R$\&$D Program of China (No. 2024YFA1015300) 
and National Natural Science Foundation of China (Nos. 12371199, 12171031). 
J. Yang was supported by a scholarship (No. 202506040134) of China Scholarship Council. 
Part of the work was done when J. Yang was a visiting Ph.D. student in University of Warsaw, 
and she would like to express her deep gratitude for the hospitality by University of Warsaw.}


\begin{abstract}
So far the global well-posedness of strong solutions for the 3D non-diffusive Boussinesq system with large initial data remains 
a remarkable open problem. In this paper, we solve this problem in the regime of large Prandtl number.
%We establish global regularity and study the infinite Prandtl number limit of temperature patches for the 3D non-diffusive Boussinesq system. 
More precisely, we prove the global existence and uniqueness of strong solution for this 3D Boussinesq system associated with initial data 
$(u_0,\theta_0)\in H^{\frac{1}{2}}(\mathbb{R}^3) \times (L^1\cap L^s(\mathbb{R}^3))$ with $s>3$, 
provided that the Prandtl number is sufficiently large (the threshold depends only on a scale-invariant norm of $(u_0,\theta_0)$); 
moreover, for the non-constant temperature patch initial data, 
we establish the global persistence of $C^{1,\gamma}$, $W^{2,\infty}$, 
and $C^{2,\gamma}$ ($0<\gamma<1$) boundary regularity of the evolved temperature patch, 
with corresponding estimates uniform in the large Prandtl number regime. 
Furthermore, we rigorously justify the limit as the Prandtl number tends to infinity and show that the patch solution of the 3D Boussinesq system 
converges to the unique patch solution of the 3D Stokes-transport system, 
and that the patch boundary regularity in $C^{1,\gamma}$, $W^{2,\infty}$, and $C^{2,\gamma}$ is preserved globally in time. 
In particular, our result for the 3D Stokes-transport system can be viewed as the 3D analogue of the main result in Grayer II [ARMA 2023] 
concerning 2D Stokes-transport system.
\end{abstract}

\maketitle
\tableofcontents

\section{Introduction}
In this paper, we study the Cauchy problem of the $d$-dimensional ($d$D) non-diffusive Boussinesq system
\begin{equation}\tag{B}\label{eq:BousEq}
\left\{\begin{aligned}
   \partial_t \theta + u\cdot \nabla \theta & = 0, \\
  \tfrac{1}{\mathrm{Pr}} \big(\partial_t u + u\cdot\nabla u \big) - \nu\Delta u 
  + \nabla p & = \theta e_d, \\
   \mathrm{div}\,u & = 0, \\
  (\theta, u)|_{t=0}(x) & = (\theta_0,u_0)(x),
\end{aligned}\right.
\end{equation}
where $d=2,3$, $(t,x)\in \R_+\times\mathbb{R}^d$, $e_d=(0,\cdots,0,1)^t$ 
and $\nu > 0$ is the kinematic viscosity coefficient, 
$\mathrm{Pr}>0$ is the non-dimensional Prandtl number (which is the quotient of $\nu$ over the thermal diffusivity $\kappa$).
The unknowns are the scalar temperature $\theta$, the velocity field $u=(u^1,u^2,u^3)^t$ and the scalar pressure $p$. 
The Boussinesq system \eqref{eq:BousEq} plays an important role in modeling natural convection phenomena in oceanic and atmospheric dynamics~\cite{Maj03,JP1987}, 
and serves as a key model for understanding Rayleigh-B\'enard convection \cite{Cha61,CD1999}. 

From the mathematical perspective, the Boussinesq system \eqref{eq:BousEq} 
is closely related to several fundamental models of incompressible fluid mechanics. 
When $\theta\equiv 0$, it reduces to the classical incompressible Navier-Stokes/Euler equations.
%so that one cannot expect to have a better theory for the Boussinesq system than for the Navier-Stokes equations.  
The coupling of velocity $u$ and temperature $\theta$ via the forcing term $\theta e_d$ 
in the momentum equation of \eqref{eq:BousEq} 
introduces an additional vortex-stretching mechanism due to temperature variation,
which may cause significant analytical difficulty even in the 2D case.
%which shares many of the analytical difficulties of the 3D Navier-Stokes and Euler equations.
As noted in \cite{MB02}, the inviscid (i.e., $\nu=0$) 2D Boussinesq system \eqref{eq:BousEq} 
shares deep analogy with the 3D axi-symmetric Euler equations with swirl (away from the symmetry axis), 
and to date, the global regularity of smooth solutions in this 2D inviscid whole-space case remains an outstanding open problem 
(see \cite{ChH21,ChH22} for recent advances in the spatial domain with boundary). 

For the viscous (i.e., $\nu>0$) Boussinesq system \eqref{eq:BousEq} with $\PPr=1$, 
the global existence and uniqueness of regular solutions associated with large initial data was well established in the 2D case. 
This was proved independently by Chae~\cite{DC06} and by Hou and Li~\cite{HL05} 
for the global well-posedness of smooth solution with 
$(u_0,\theta_0)\in H^s\times H^s(\mathbb{R}^2)$, $s>2$. 
%Later, Abidi and Hmidi~\cite{AH07} improved this result by showing that global well-posedness still holds for significantly less regular data, 
%allowing $\theta_0\in B^{0}_{2,1}$ and $u_0\in L^{2}\cap B^{-1}_{\infty,1}$. 
Subsequently, Hmidi and Keraani~\cite{HK2007} constructed the global weak solution
for initial data $\theta_0\in L^2(\mathbb{R}^2)$ and $u_0\in H^s(\mathbb{R}^2)$ with $s\in[0,2)$; later, 
Danchin and Paicu~\cite{DanP08b} obtained the uniqueness of this weak solution using the paradifferential calculus.  
See also~\cite{AH07,HKZ15} for related global regularity results.
%proved that these weak solutions possess global-in-time persistence of regularity in a wide range of Sobolev spaces.

For the 3D viscous Boussinesq system \eqref{eq:BousEq} with $\PPr=1$, 
the global well-posedness of the smooth solutions remains a remarkable open problem, 
akin to the millennium prize problem of 3D Navier-Stokes equations, 
and only a few partial results on global solutions are available. %on this model.  
Danchin and Pa\"icu~\cite{DanP08b} constructed global Leray-Hopf type weak solution of finite energy for the 3D Boussinesq system \eqref{eq:BousEq} 
with initial data 
$(u_0,\theta_0)\in L^2 \times L^2(\mathbb{R}^3)$; %And then addressing the uniqueness issue for such solutions only in 2D case. 
and they also established a Fujita-Kato type well-posedness result: 
for arbitrary data in the scale-invariant Lorentz and Besov spaces, 
a unique local-in-time strong solution exists, 
and under the additional assumption that $\|u_0\|_{L^{3,\infty}}+ \|\theta_0\|_{L^1}\leq c\, \nu$
with sufficiently small $c>0$, a unique global strong solution is obtained. 
%This result may be viewed as an extension of the classical Fujita-Kato theory for the incompressible three-dimensional Navier-Stokes equations~\cite{FujK64,Kato84}. 
In a subsequent paper~\cite{DanP08}, 
Danchin and Pa\"icu further refined the analysis to show the global well-poseness 
under the same smallness assumption, but with the initial data only in Lorentz spaces. 
%thus weakening the Besov regularity requirement. 
%and the smallness condition on the initial data.
Dong and Sun \cite{DS23} proved global asymptotic stability of a stratified stationary solution for the 3D Boussinesq system \eqref{eq:BousEq}
in the domain $\mathbb{R}^2\times (0,1)$ under slip boundary conditions.
See also Hmidi et al. \cite{AHK11,HR10} for global well-posedness results of the 3D Boussinesq system 
\eqref{eq:BousEq} with axisymmetric initial data without swirl.

\begin{comment}
In the modeling of large-scale atmospheric and oceanic flows, 
it is natural to allow the viscosity and diffusion coefficients in the Boussinesq approximation to be anisotropic, 
reflecting different horizontal and vertical scales and possibly vanishing dissipation in some coordinate directions. 
One can see \cite{CaoW13,LLT13,LiT16} 
and the references therein for various global well-posedness results of the Boussinesq system in the anisotropic or partially dissipative setting.
\end{comment}
\vskip1mm

Recently, there has been considerable interest on the so-called \textit{Boussinesq temperature patch problem} 
for the non-diffusive Boussinesq system \eqref{eq:BousEq}. 
%This problem involves a free boundary problem with 
Consider temperature initial data of the form, e.g. $\theta_0 = \mathbf{1}_{D_0}$,
which is a characteristic function of a bounded simply connected domain $D_0 \subset \mathbb{R}^d$.
Since $\theta$ satisfies a transport equation and the velocity $u$ is assumed to be regular enough,
it follows (at least formally) that the temperature patch structure is preserved, i.e., $\theta(x, t)=\mathbf{1}_{D(t)}(x)$,
where $D(t)=X_t\left(D_0\right)$ and $X_t:\mathbb{R}^d\mapsto\mathbb{R}^d$ 
is the flow map generated by the velocity $u$ satisfying
\begin{align}\label{eq:flowmap}
  \frac{\dd X_t(y)}{\dd t}=u\left(X_t(y), t\right),\left.\quad X_t(y)\right|_{t=0}=y.
\end{align}
This naturally raises the question of whether the initial regularity of the patch boundary is globally preserved along the evolution. 
More precisely, the typical question is as follows:
\begin{align*}
  \textit{suppose $\partial D_0 \in C^{k,\gamma} \,(\text{or}\,\,W^{k,\infty})$, $k \in \mathbb{Z}^{+}, \gamma \in(0,1)$,
  does $\partial D(t) \in C^{k,\gamma}\,(\text{or}\,\,W^{k,\infty})$ hold for all time?}
  \end{align*} 
Here, for a bounded simply connected domain $D_0\subset \R^3$, $\partial D_0\in C^{k,\gamma}$ 
means that there exists a finite number of charts $\{V_\beta, \varphi_\beta\}_{1\leq \beta\leq m}$ with $\varphi_\beta\in C^{k,\gamma}$ 
(or $W^{k,\infty}$) covering the two-dimensional compact submanifold $\partial D_0$ such that
\begin{equation}\label{def:varphi-beta}
\begin{split}
  \varphi_\beta:\quad  U_\beta&\longmapsto V_\beta \subset \R^3, \\
  (s_1,s_2)&\longmapsto \varphi_\beta(s_1,s_2)= (\varphi_\beta^1,\varphi_\beta^2,\varphi_\beta^3)(s_1,s_2)\in V_\beta,
\end{split}
\end{equation}
where $U_\beta$ is an open set of $\R^2$, $V_\beta$ is an open set of $\R^3$ 
near the neighbourhood of $\partial D_0$.
We say that $\partial D(t)\in C^{k,\gamma}$ (or $W^{k,\infty}$) if $X_t(\varphi_\beta) \in C^{k,\gamma}$ (or $W^{k,\infty}$)
for every $\beta=1,\cdots,m$.

The study of patch solutions dates back to the classical \textit{vorticity patch problem} 
for the 2D incompressible Euler equations raised in 1980s. 
%with initial vorticity $\omega_0 = \mathbf{1}_{D_0}$. 
Although some numerical experiments of the contour dynamics equation describing vortex patches were controversial, 
Chemin \cite{Chemin93} first gave an affirmative solution to this problem in the Eulerian framework.
By originally developing estimates of striated regularity, 
he established the global persistence of the patch boundary in the class $C^{k,\gamma}$ 
with $(k,\gamma)\in\mathbb{Z}_+\times(0,1)$. 
Later, Bertozzi and Constantin \cite{BerC} 
provided an alternative proof of the same persistence result based on a geometric cancellation lemma combined with harmonic analysis estimates; 
see also Serfati \cite{Serf94} and Radu~\cite{Ra22} for other proofs. 
%More recently, Radu~\cite{Ra22}
%revisited the vortex patch problem in the level-set formulation and obtained a self–contained proof of the global persistence of
%$C^{k,\gamma}$-regularity for the patch boundary, for all $(k,\gamma)\in\mathbb Z_+\times(0,1)$.
Very recently, Kiselev and Luo~\cite{KL23} proved the strong ill-posedness of boundary regularity for vortex patches in $C^2$ (or $W^{2,\infty}$).
%by constructing initial patch data in this class whose boundary curvature becomes instantaneously unbounded.
The analogue of the vortex patch problem for the 3D Euler equations was addressed by Gamblin and Saint-Raymond 
\cite{GSR95}, who proved the local well-posedness of vorticity patches with 
$C^{1,\gamma}$-boundaries as well as the global well-posedness result in the axisymmetric (without swirl) case; 
see also Danchin \cite{Dan99} for the related regularity persistence result in higher dimensions. 

The temperature patch problem of the viscous Boussinesq system \eqref{eq:BousEq} with $\PPr=1$ was initiated by
Danchin and Zhang \cite{DanZ17}. Using paradifferential calculus and striated estimates,
they proved that the $C^{1,\gamma}$-regularity of the patch boundary is globally preserved in the 2D case, 
and also in the higher–dimensional case under an additional smallness condition: 
$\|u_0\|_{L^{d,\infty}} + \nu^{-1} \|\theta_0\|_{L^{d/3}} \leq c\,\nu$. 
Subsequently, Gancedo and Garc\'ia-Ju\'arez \cite{GGJ17} provided an alternative proof of this $C^{1,\gamma}$ persistence result and, 
moreover, showed the global propagation of $W^{2,\infty}$ and $C^{2,\gamma}$ regularity of the temperature patch boundary. 
Their approach is based on new cancellations in time-dependent Calder\'on–Zygmund operators and on the tangential derivative along the patch boundary.
The same result was extended to the 3D non-constant temperature patches (called \textit{temperature fronts} in the literature \cite{Maj03}) 
by the same authors \cite{GGJ20} under a smallness condition of initial data.
Furthermore, by developing higher-order striated estimates, 
the first author together with Chae and Xue \cite{CMX22} established the global 
$C^{k,\gamma}$-regularity persistence of the patch boundary in the 2D case for all
$k\in\mathbb Z_+$ and $\gamma\in (0,1)$; 
see also \cite{LLX24} for the analogous persistence result for the 3D Boussinesq system \eqref{eq:BousEq}. 
We also mention that Park \cite{Park24} 
constructed an explicit patch initial data for the 2D Boussinesq system \eqref{eq:BousEq} 
showing the infinite-in-time (algebraic) growth of the curvature and perimeter of the evolved temperature patch boundary.

There are physical situations, such as in some highly-pressurized gases, where the Prandtl number $\PPr$ of the fluid is very large. 
When $\PPr$ tends to infinity,
the material derivative of $u$ in the momentum equation in the Boussinesq system \eqref{eq:BousEq} formally vanishes and
the system \eqref{eq:BousEq} reduces to the $d$-dimensional Stokes-transport system %(by assuming $\nu=1$):
\begin{equation}\tag{ST}\label{eq:ST}
\left\{\begin{aligned}
  \partial_t \theta + u \cdot \nabla \theta & =0,\\
  -\nu \Delta u + \nabla p & = \theta e_d, \\
  \nabla \cdot u & =0, \\
  \left.\theta \right|_{t=0} & = \theta_0.
%  \left.\theta\right|_{t=0} & =\theta_0 .
  \end{aligned}\right.
\end{equation}
Physically, the Stokes-transport system \eqref{eq:ST} 
arises as a macroscopic limit model for the sedimentation of spherical rigid particles in a viscous incompressible fluid 
in the regime where both fluid and particle inertia are negligible or very small \cite{Hofer18,Mec19,HS21,HS25}. 
H\"ofer~\cite{Hofer18} proved the global existence and uniqueness result for the 3D Stokes-transport system \eqref{eq:ST} 
with regular initial data. 
This result was later extended by Mecherbet~\cite{Mec21} to initial data in $L^1\cap L^\infty(\R^3)$ with the finite first moment, 
and independently, a similar global well-posedness result was obtained in~\cite{HS21} without the moment assumption. 
More recently, Mecherbet and Sueur~\cite{MS22} showed the global existence and uniqueness of weak solutions 
associated with the initial data $\theta_0\in L^1\cap L^p(\R^3)$, $p\geq 3$, 
and also proved analyticity of trajectories and exact controllability.
Inversi \cite{Inv23} obtained global existence (without uniqueness) of Lagrangian solutions for $L^1$-initial data.
In addition, Leblond~\cite{Leblond} established global existence and uniqueness for the Stokes-transport system \eqref{eq:ST} 
with bounded initial data in bounded domains of $\R^d$ ($d=2,3$) 
as well as in the infinite strip $\Omega=(0,1)\times\R$ with a flux condition.

The formal convergence from the Boussinesq system \eqref{eq:BousEq} to the Stokes-transport system \eqref{eq:ST} 
was first noted by \cite{Gray23} in the 2D case, and Grayer II \cite{Gray23} established global well-posedness of the 2D system \eqref{eq:ST} 
for Yudovich-type initial data in $L^1\cap L^\infty$, 
together with the global persistence of $C^{1,\gamma}$, $W^{2,\infty}$ and $C^{2,\gamma}$ ($0<\gamma<1$) boundary regularity of the patch solution.
The infinite Prandtl number limit of the 2D Boussinesq system \eqref{eq:BousEq} %to the 2D system \eqref{eq:ST} 
was later rigorously justified for regular solutions (including patch solution) on the torus $\mathbb{T}^2$ 
by the second author and Lazar, Xue ~\cite{LXY25}; they also proved the global persistence of 
$C^{k,\gamma}$ boundary regularity with any $k\ge1$ for the patch solution of the 2D system \eqref{eq:ST}. 
%In particular, they rigorously justified the infinite Prandtl number limit in $\mathbb{T}^2$, i.e.  as $\PPr\to\infty$, 
%patch solutions of the Boussinesq system converge to the unique patch solutions of the  Stokes–transport system, 
%while preserving the $C^{k,\gamma}$-regularity of the patch boundary and extending the result by Grayer II \cite{Gray23} 
%for the Stokes-transport system to all $k\ge1$.  
For the 3D Stokes-transport system \eqref{eq:ST}, 
Mecherbet \cite{Mec21} discovered an interesting traveling-wave patch solution: 
for $\theta_0 = \mathbf{1}_{B_0}$ (with $B_0\subset \mathbb{R}^3$ the unit ball), 
the system admits a unique solution $\theta = \mathbf{1}_{B(t)}$ 
that preserves the spherical patch structure for all the time. More precisely, 
there exists a constant $v^*\in\mathbb{R}^3$ such that
\begin{align*}
  (u,\theta)(t,x) = \big(u_0(x-v^* t), \theta_0(x-v^* t)\big),
  \quad\; \big(u - v^*\big)\cdot \mathbf{n} =0\;\;\; \textrm{on}\;\; \partial B(t),
\end{align*}
where $\mathbf{n}$ is the unit outer normal vector (see \cite{BM25} for related instability analysis).
For other type of patch solutions,
Gancedo et al. \cite{GGS25a,GGS25b,GGHSY} (see \cite{AMY00} for previous work) 
studied the dynamics of the sharp interface for the 2D Stokes-transport system \eqref{eq:ST} 
in the unbounded strip $\mathbf{D}= \mathbb{T}\times \mathbb{R}$. Using the contour dynamics formulation, 
they proved global well-posedness and global $C^{k,\gamma}$ ($k\geq 1$, $0<\gamma<1$) regularity persistence of the free interface,
and also investigated the long-time behavior around stable/unstable stratified states.
We refer to~\cite{DGL23,Park25,SZZ26a,SZZ26b} for recent studies on the long-time dynamics of the Stokes-transport system \eqref{eq:ST}
around stratified states or shear flows. 
%More recently, Cobb~\cite{Cobb23} obtained several well-posedness results in critical functional spaces for the fractional Stokes-transport system in all spatial dimensions $d\ge 2$.

We also remark that replacing the transport equation in \eqref{eq:ST} by the following transport-diffusion equation
$\partial_t \theta + u\cdot \nabla \theta - \kappa \Delta \theta = 0$ ($\kappa >0$)
yields a related Rayleigh-B\'enard model derived from the Boussinesq system with diffusion in the regime of infinite Prandtl number \cite{Cha61}. 
In recent decades, much research has been devoted to obtaining the (optimal) upper bound of the Nusselt number in terms of the Rayleigh number 
for this three-dimensional limit system in spatial domains with boundary condition; see \cite{CD1999,OS11,ChM25,Seis25} and references therein. 
Here, the Nusselt number is a physical quantity that measures the average vertical heat flux.
See also \cite{Wang04,Wang08,Wang08b,CNO16} for analogous work on the 3D Boussinesq system at finite (large) Prandtl number.
\vskip1mm

It is worth noting that, so far the global well-posedness result for the 3D Boussinesq system \eqref{eq:BousEq} 
typically requires smallness assumptions on the initial data, whereas its formal infinite Prandtl number limit, 
the Stokes-transport system \eqref{eq:ST}, is globally well-posed without any smallness condition. 
This contrast suggests that a large Prandtl number regime can be viewed as an alternative to the usual smallness assumption. 

In this paper, motivated by the above observation, 
we first establish global well-posedness for the 3D Boussinesq system \eqref{eq:BousEq} 
with a large Prandtl number and initial data $u_0\in H^{\frac{1}{2}}(\mathbb{R}^3)$, $\theta_0\in L^1\cap L^s(\mathbb{R}^3)$, $s>3$, 
and show the global uniform-in-$\PPr$ persistence of $C^{1,\gamma}$, $W^{2,\infty}$, 
and $C^{2,\gamma}$ boundary regularity of the non-constant temperature patch solution in this regime.
We then rigorously justify the convergence of the 3D Boussinesq system \eqref{eq:BousEq} 
in the limit of infinite Prandtl number and obtain the global regularity propagation of the patch solution for the 3D Stokes-transport system 
\eqref{eq:ST}. Note that our result for the Stokes-transport system is a 3D analogue of the main result in Grayer II \cite{Gray23}.
For simplicity, we set $\nu=1$ throughout the sequel. 

Our first result concerns the 3D Boussinesq system \eqref{eq:BousEq} with large Prandtl number.
\begin{theorem}\label{thm:exi-reg}
%Let $\mathrm{Pr}>0$. 
Let $u_0\in H^{\frac{1}{2}}(\R^3)$ be a divergence-free vector field and assume that  
$\theta_0 \in L^1\cap L^s(\R^3)$, $s>3$.
Then, there exists a universal small constant $c_*>0$ 
(chosen from \eqref{eq:c*} and \eqref{es.v.Hfr12} below) such that if
\begin{equation}\label{eq:data-cond}
  \mathrm{Pr}\geq \mathrm{Pr}_* \triangleq \frac{ \|u_0\|_{L^{3,\infty}(\R^3)} + \|\theta_0\|_{L^{1}(\R^3)}}{c_*},
\end{equation}
then the 3D Boussinesq system \eqref{eq:BousEq}
has a unique global solution $(u,\theta)$ which satisfies that, for any $T>0$ and $q\in(3,6)$,
\begin{equation}\label{eq:v-the-es}
  u \in C([0,T]; H^{\frac{1}{2}}(\R^3)) \cap L^2([0,T]; H^{\frac{3}{2}}\cap L^\infty)
  \cap \widetilde{L}^1([0,T];\dot{B}^{1+\frac{3}{q}}_{q,\infty}),
  \quad \theta \in L^\infty([0,T];L^1\cap L^s(\R^3)).
\end{equation}

Moreover, consider the temperature patch initial data 
\begin{align}\label{eq:tem-patch-data}
  \theta_0(x) = \overline{\theta}_0(x) \,\mathbf{1}_{D_0}(x), \quad \textrm{with}\;\;
  D_0\subset\DD\;\; \textrm{a bounded simply connected domain},
\end{align}
then under the condition \eqref{eq:data-cond}, the following regularity persistence results hold. 
\begin{enumerate}[(1)]
\item 
If $\partial D_0 \in C^{1,\gamma}$, $\gamma\in(0,1)$, $\overline{\theta}_0\in L^\infty(\overline{D_0})$, and $u_0$ satisfies
\begin{align*}
   u_0\in 
   \begin{cases}
   H^1(\mathbb{R}^3), 
      \quad & \textrm{if}\;\gamma\in(0,\frac{1}{2}), \\
     H^1\cap W^{1,p}(\mathbb{R}^3), \, \textrm{with some}\,\,\frac{3}{2-\gamma}<p<\infty,
     \quad & \textrm{if}\,\,\gamma\in[\frac{1}{2},1),
   \end{cases}
 \end{align*}
 then we have
\begin{align}\label{eq:the-patch-exp}
  \theta(x,t) = \overline{\theta}_0(X^{-1}_t(x)) \mathbf{1}_{D(t)}(x),\quad D(t)=X_t(D_0),
\end{align}
and
\begin{equation*}
  \partial D(t) \in L^\infty([0,T]; C^{1,\gamma}),\quad 
  \textrm{uniformly in \;$\PPr\in [\PPr_*,\infty)$},
\end{equation*}
where $X_t$ is the flow map solving \eqref{eq:flowmap} and $X_t^{-1}$ is its inverse.

\item 
If additionally $\partial D_0 \in W^{2,\infty}$, $\overline{\theta}_0\in C^{\mu}(\overline{D_0})$,
$\mu\in (0,1)$, and $u_0\in H^1\cap W^{1,p}(\R^3)$ for some $3<p<\infty$,
then we have
\begin{equation*}
  \partial D(t) \in L^\infty([0,T]; W^{2,\infty})\quad 
  \textrm{uniformly in \;$\PPr\in [\PPr_*,\infty)$}.
\end{equation*}

\item 
If additionally $ \partial D_0 \in C^{2,\gamma}$, $\gamma\in(0, 1)$, 
$\overline{\theta}_0\in C^\gamma(\overline{D_0})$, $u_0 \in H^1\cap W^{2,p}(\R^3)$ for some $\frac{3}{2-\gamma}<p<\infty$,
then we have
\begin{equation}
  \partial D(t) \in L^\infty([0,T]; C^{2,\gamma})\quad 
  \textrm{uniformly in \;$\PPr\in [\PPr_*,\infty)$}.
\end{equation}
\end{enumerate}
\end{theorem}

\begin{comment}
In the above, the patch boundary $\partial D(t)\in L^\infty_T(X)$ 
($X$ is some Banach space, e.g., $C^{k,\gamma}$)
means that $X_t(\varphi_\beta)\in L^\infty_T(X)$ for every $\beta\in\{1,\cdots,m\}$,
where $\{V_\beta,\varphi_\beta\}_{1\leq \beta\leq m}\in X$ 
is the charts of the initial boundary $\partial D_0$ given by \eqref{def:varphi-beta}.
One can see Lemma \ref{lem:sr-cond} for the characterization of 
$\partial D(t)\in L^\infty_T(X)$ in terms of the regularity of the admissible system 
$\mathcal{W}(t,x)=\{W^i(t,x)\}_{1\leq i\leq 5}$ 
composed of divergence-free tangential vectors.
\end{comment}

\begin{remark}[Persistence of higher $C^{k,\gamma}$ and $W^{k,\infty}$ regularity]\label{rem:high-reg}
  Under the assumptions of Theorem \ref{thm:exi-reg} and by adding suitable hypotheses,
one can prove the uniform-in-$\PPr$ global $C^{k,\gamma}$ and 
$W^{k,\infty}$ boundary regularity of temperature patches with any $k\geq 3$, $0<\gamma<1$ 
and $\mathrm{Pr}\in[\mathrm{Pr}_*,\infty)$. 

Indeed, with necessary modification, 
the proof is analogous to that for the 2D case presented in \cite{LXY26}. 
%where the uniform-in-$\PPr$ global $C^{k,\gamma}$- and 
%$W^{k,\infty}$-boundary regularity of temperature patches is obtained for any 
%$k\in\mathbb{Z}_+$, $0<\gamma<1$ and $\PPr\geq 1$. 
We sketch here only the persistence result of the $W^{3,\infty}$ boundary regularity. 
If additionally $\partial D_0 \in W^{3,\infty}$, $\overline{\theta}_0\in C^{1,\mu}(\overline{D_0})$ with
$\mu\in (0,1)$, and $u_0\in H^1\cap W^{2,p}(\R^3)$ for some $p\in(3,\infty)$,
then arguing as in Section~\ref{sec:W2} (considering the equation of $\partial_{\mathcal{W}}(\nabla \mathcal{W})$ 
instead of equation \eqref{eq:nabW}), and taking advantage of \eqref{es:prod-es6}-\eqref{es:prod-es7}
in Lemma~\ref{lem:str-es1}, Lemma \ref{lem:str-reg} and \eqref{eq:W-C2+gam-2}, 
we can show that 
$\partial_\WW\WW\in L^\infty(0,T; W^{1,\infty})$ uniformly in $\PPr\in [\PPr_*,\infty)$,
which yields $\partial_\WW^2\WW\in L^\infty(0,T; L^{\infty})$ uniformly in $\PPr$ 
and the global uniform persistence of the $W^{3,\infty}$ patch boundary regularity.
\end{remark}

\begin{remark}[Two-phase temperature patches]
The temperature patch initial data can also be treated in the form
$\theta_0=\bar\theta_1\,\mathbf{1}_{D_0}+\bar\theta_2\,\mathbf{1}_{D_0^c},$ where $\bar\theta_1$ and 
$\bar\theta_2$ are functions defined in $\overline{D_0}$ and $\overline{D_0^c}$, respectively; see, e.g., the 3D patch setting in~\cite{LLX24}.
\end{remark}

Next, our second main result rigorously investigates the temperature patch solution in the regime of infinite Prandtl number,
which connects both systems \eqref{eq:BousEq} and \eqref{eq:ST} in the 3D whole space $\mathbb{R}^3$.
%we rigorously investigate the limit when $\PPr\to\infty$ and show the persistence of  
%$C^{k,\gamma}$- and $W^{k,\infty}$-boundary regularity of the 3D Stokes–transport system \eqref{eq:ST}. 
%Our second main result can be formulated as follows.
\begin{theorem}\label{thm:limit}
Let $T>0$ be any given. %and $\PPr_*>0$ be the constant given by \eqref{eq:data-cond}.
For every $\mathrm{Pr} \in [\mathrm{Pr}_*,\infty)$ with $\PPr_*>0$ large enough ($\PPr_*$ can be chosen universally due to 
\eqref{eq:data-cond} and the strong convergence of $u_0^{\PPr}$), suppose that $(u^{\mathrm{Pr}}, \theta^{\mathrm{Pr}})$ 
is the unique global regular solution to the 3D Boussinesq system \eqref{eq:BousEq} associated with $(u_0^{\PPr},\theta_0)$, 
constructed in Theorem \ref{thm:exi-reg},
where $\theta_0(x) = \overline{\theta}_0(x) \,\mathbf{1}_{D_0}(x)$ 
is the temperature patch initial data \eqref{eq:tem-patch-data}, $\partial D_0 \in C^{1,\gamma}(\R^3)$ with $\gamma\in(0,1)$, $\overline{\theta}_0\in L^\infty(\overline{D_0})$, and $u_0^{\PPr}$ satisfies $\nabla\cdot u_0^{\PPr} = 0$ and
\begin{align*}
   u_0^{\PPr}\in 
   \begin{cases}
   H^1(\mathbb{R}^3), 
      \quad & \textrm{if}\;\; \gamma \in (0,\frac{1}{2}), \\
     H^1\cap W^{1,p}(\mathbb{R}^3), \, \textrm{with some}\,\,\frac{3}{2-\gamma}<p<\infty,
     \quad & \textrm{if}\;\;\gamma\in[\frac{1}{2},1).
   \end{cases}
\end{align*}
In addition, by sending $\PPr\rightarrow \infty$, assume that $u^{\mathrm{Pr}}_0$ strongly converges to $u_0\in H^1(\mathbb{R}^3)$, 
which satisfies the compatibility condition $-\Delta u_0 + \nabla p_0 = \theta_0 e_3$ with some function $p_0$.

Then, as $\mathrm{Pr}\rightarrow \infty$, up to extraction of a subsequence, 
\begin{align*}
  \textrm{$(u^{\mathrm{Pr}},\theta^{\mathrm{Pr}})$\;\; converges (weakly) to\;
  $(u, \theta)$}, 
\end{align*}
which is the unique global solution of the 3D Stokes-transport system
\eqref{eq:ST} with initial data $\theta_0$, and $\theta$ satisfies
\begin{equation}\label{eq:limit-targ-1}
  \theta(x, t) = \overline{\theta}_0(X_t^{-1}(x)) \mathbf{1}_{D(t)}(x), \quad
  \textrm{with}\;\; \partial D(t) \in L^{\infty}\big(0, T ; C^{1,\gamma}\big),
\end{equation}
where $D(t)=X_t(D_0)$,  $X_t:\mathbb{R}^3\rightarrow \mathbb{R}^3$ solving \eqref{eq:flowmap} 
is the flow map generated by the velocity $u$ and $X^{-1}_t$ is its inverse. 
Moreover, the following statements hold.
\begin{enumerate}[(1)]
\item 
If additionally $\partial D_0 \in W^{2,\infty}$, $\overline{\theta}_0\in C^{\mu}(\overline{D_0})$ with
$\mu\in (0,1)$, and $u_0^{\PPr} \in H^1\cap W^{1,p}(\R^3)$ for some $3<p<\infty$.
Then the temperature patch solution \eqref{eq:limit-targ-1} satisfies
\begin{equation}\label{eq:limit-targ-2}
  \partial D(t) \in L^\infty([0,T]; W^{2,\infty}).
\end{equation}

\item 
If additionally $ \partial D_0 \in C^{2,\gamma}$ with $\gamma\in(0, 1)$, 
$\overline{\theta}_0\in C^\gamma(\overline{D_0})$, $u_0^{\PPr}\in H^1\cap W^{2,p}(\DD)$ for some $\frac{3}{2-\gamma}<p<\infty$. 
Then the temperature patch solution \eqref{eq:limit-targ-1} satisfies
\begin{equation}\label{eq:limit-targ-3}
  \partial D(t) \in L^\infty([0,T]; C^{2,\gamma}).
\end{equation}
\end{enumerate}
\end{theorem}

\begin{remark}
It should be noted that, in contrast to Theorem \ref{thm:limit} and Proposition \ref{prop:weak-limit} concerning the 3D Boussinesq system \eqref{eq:BousEq} 
in $\mathbb{R}^3$ (which also naturally extend to $\mathbb{T}^3$), 
it remains an open problem to show the analogous convergence result as $\PPr\rightarrow \infty$ for the 2D Boussinesq system \eqref{eq:BousEq}
in whole space $\mathbb{R}^2$ (recalling that \cite{LXY25} only treats the torus case $\mathbb{T}^2$).
The main technical reason is that we have the classical embedding $\dot H^1(\mathbb{R}^3)\hookrightarrow L^6(\mathbb{R}^3)$, 
but $\dot H^1(\mathbb{R}^2)$ does not embed continuously in any Lebesgue space $L^p(\mathbb{R}^2)$ with $p\in[1,\infty]$.
\end{remark}

\begin{remark}
In view of Remark~\ref{rem:high-reg}, the conclusion of Theorem~\ref{thm:limit} can be further extended.  
In particular, the convergence result and the persistence of boundary regularity remain valid for general patch initial data whose boundary 
$\partial D_0$ possesses $W^{k,\infty}$ or $C^{k,\gamma}$ regularity with any $k \geq 3$ and $\gamma \in (0,1)$.  
\end{remark}

In the proof of the existence and uniqueness part of Theorem~\ref{thm:exi-reg}, 
the crucial new ingredient is the \textit{a priori} global-in-time estimate 
$u\in L^\infty(0,T;H^{\frac{1}{2}})\cap L^2(0,T;H^{\frac{3}{2}})$ 
for large Prandtl number $\PPr$ without imposing the smallness assumption on the initial data.
This estimate is based on a classical energy-type argument and the global $L^\infty_T(L^{3,\infty})$-estimate of $u$ 
(analogous to the corresponding estimate in \cite{DanP08b} with small initial data). Once this estimate is established, 
the desired existence and uniqueness result follows essentially by repeating the argument in \cite{LLX24}. 

The remaining part of Theorem \ref{thm:exi-reg} deals with the regularity persistence results for non-constant temperature patches. 
Motivated by the recent work \cite{LLX24} (which originated from the ideas of \cite{HKR10,HKR11}),
we introduce the following auxiliary quantity 
\begin{align}\label{def:Gamma}
  \Gamma: = \Omega - (\mathcal{R}_{-1,2}\theta, - \mathcal{R}_{-1,1}\theta,0)^t,
\end{align} 
where $\Omega = \nabla \wedge u$ denotes the vorticity and $\mathcal{R}_{-1,j} : =\partial_j (-\Delta)^{-1}$ for $j=1,2$.
The good unknown $\Gamma$ satisfies the equation \eqref{eq.Gamm} (compare with the vorticity equation \eqref{eq.Omeg}) 
and plays an important role in the analysis, especially in deriving uniform estimates with respect to the Prandtl number $\PPr$ 
(note that the constitutive relation of the 3D Stokes-transport system \eqref{eq:ST} can be written as 
$\Omega = (\mathcal{R}_{-1,2}\theta, - \mathcal{R}_{-1,1}\theta,0)^t$, i.e., $\Gamma\equiv 0$ in \eqref{def:Gamma}).

In order to prove the propagation of $C^{1,\gamma}$ regularity, since $\partial D(t)=X_t(\partial D_0)$, 
and $\partial D_0$ is determined by the charts $\{V_\beta,\varphi_\beta\}_{1\leq \beta \leq m}$ satisfying \eqref{def:varphi-beta},
it suffices to prove that the velocity field $\nabla u$ belongs to $L^1_T(C^{\gamma})$;
according to Lemma \ref{lem:flow}, this directly implies $\nabla X_t\in L^\infty(0,T;C^{\gamma})$, 
where $X_t:\mathbb{R}^3\mapsto \mathbb{R}^3$ is the flow map satisfying the ODE \eqref{eq:flowmap}. 
In view of the Biot-Savart law, we decompose $\nabla u$ as follows: 
\begin{align}\label{eq.v.GamThe0}
  \nabla u = (-\Delta)^{-1}\nabla\nabla \wedge \Omega 
  = \Lambda^{-2}\nabla\nabla\wedge \Gamma 
  + \nabla^2\partial_3  \Lambda^{-4} \theta  + \Lambda^{-2}\nabla \theta\otimes e_3,
\end{align}
with $\Lambda: = (-\Delta)^{1/2}$.
Using the equation \eqref{eq.Gamm} of good unknown $\Gamma$, commutator estimates, and the regularity of $u$ in the existence result, 
we can obtain the desired estimate for $\Gamma$ that is uniform in $\PPr$. 
Since $\theta$ belongs to $L^1\cap L^\infty$ uniformly in time,
it is easy to see that $\nabla^2\partial_3  \Lambda^{-4} \theta$ and $\Lambda^{-2} \nabla\theta $
belong to $L^1_T(C^\gamma)$ for all $\gamma\in (0, 1)$; therefore, the desired bound $\nabla u \in L^1_T(C^\gamma)$ follows immediately;
see Propositions \ref{prop:ap-es1} and \ref{prop:ap-es2} for more details.
 
For the propagation of $W^{2,\infty}$ or $C^{2,\gamma}$ boundary regularity, thanks to Lemma \ref{lem:sr-cond}, 
it suffices to prove that 
\begin{align*}
  \textrm{$\WW\in L^\infty(0,T;W^{1,\infty})$\; or \;$\WW\in L^\infty(0,T;C^{1,\gamma})$\; uniformly in \;$\PPr$}, 
\end{align*}
where $\mathcal{W}(t) =\{W^i(t)\}_{1\leq i\leq 5}$ satisfying equation \eqref{eq:Wi} 
is the system of tangential divergence-free vector fields related to the patch boundary $\partial D(t)$.
Note that compared to \cite{LLX24}, we adopt a different approach to show the global persistence of $W^{2,\infty}$-boundary regularity,
which can be similarly extended to the proof of higher $W^{k,\infty}$-regularity persistence with $k\geq 3$ (see Remark \ref{rem:high-reg}).
%Next our target is to prove the regularity of $\WW$. 
%The maximum principle of \eqref{eq:Wi} and the fact that $\nabla u \in L^1_T(C^\gamma)$ 
%yields the $L^\infty(0,T;L^\infty)$ norm of $\WW$. 
To estimate the $L^\infty(0,T;L^\infty)$-norm of $\nabla \WW$, 
we apply $\nabla$ to equation \eqref{eq:Wi} and use the maximum principle; the main issue is then to bound 
$\|\partial_\WW \nabla u\|_{L^1(0,T;L^\infty)}$. 
In view of the decomposition \eqref{eq.v.GamThe0} and the regularity of $\Gamma$, it only needs to prove that
$\partial_\WW \nabla^3 \Lambda^{-4}\theta,\,\partial_\WW \nabla\Lambda^{-2} \theta \in L^1(0,T;L^\infty)$.
Take the term $\partial_\WW \nabla^3 \Lambda^{-4}\theta$ as an example. We split it as
\begin{align}\label{eq:fact-commutator}
  \partial_\WW\nabla^3 \Lambda^{-4}\theta
  = \nabla^3 \Lambda^{-4}\partial_\WW\theta
  - [\nabla^3 \Lambda^{-4},\WW\cdot\nabla]\theta,
\end{align}
where $[A,B]=AB-BA$ denotes the usual commutator.
To handle the commutator term, we establish a key commutator estimate \eqref{es:prod-es6}. 
As for the term involving $\partial_\WW\theta$, we use Lemma \ref{lem:str-reg} together with the fact that 
$\partial_\WW\theta$ satisfies a transport equation to obtain $\partial_\WW\theta \in L^\infty(0,T;C^{\mu-1})$, 
which by a direct computation yields $\nabla^3 \Lambda^{-4}\partial_\WW\theta \in L^\infty(0,T;C^\mu)$. 
Combining the above bounds, we derive an estimate showing that $\|\nabla \WW(t)\|_{L^\infty}$ is controlled 
by the time integral of $\|\nabla \WW\|_{L^\infty}$ weighted by $(\|\nabla u\|_{L^\infty}+1)$. 
The desired global bound then follows from Gr\"onwall's inequality.

To estimate the $L^\infty(0,T; C^{1,\gamma})$-norm of $\WW$, 
in light of the equation of $\nabla \mathcal{W}$ and the product estimate in $C^\gamma$, 
it suffices to control $\|\partial_\WW \nabla u\|_{L^1(0,T;C^\gamma)}$. 
Owing to the decomposition \eqref{eq.v.GamThe0}, we treat $\partial_\WW\Gamma$ and $\partial_\WW\theta$ separately. 
For the term $\partial_\WW\Gamma$, in order to impose an optimal assumption on $u_0$ 
(note that we have relaxed the scope of $p$ from $3<p<\infty$ in \cite{LLX24} to $\frac{3}{2-\gamma}<p<\infty$), 
we decompose $\partial_\WW\Gamma$ into two parts: one solves equation \eqref{eq:F1} with zero initial data, 
and the other solves equation \eqref{eq:F2} with zero forcing.
We then apply the smoothing estimates \eqref{TD-sm-es} and \eqref{eq:TD-sm5} to these two equations. 
By exploiting the crucial product estimate \eqref{eq:prodBes-endp} and the commutator estimate \eqref{eq:R-1cm-es1}, 
we obtain a uniform-in-$\PPr$ bound for $\partial_\WW\Gamma$ in $L^1_T(B^\gamma_{\infty,1})$, 
which, together with the striated estimate \eqref{eq:prod-es4}, 
yields a satisfactory control of $\partial_\WW (\nabla^2\Lambda^{-2}\Gamma)$ in $L^1_T(C^\gamma)$. 
As for the term $\partial_\WW\theta$, we combine Lemma \ref{lem:str-reg} with the striated estimate \eqref{eq:prod-es4}, 
to obtain a uniform-in-$\PPr$ upper bound for both $\partial_\WW\nabla^3 \Lambda^{-4}\theta$ and 
$\partial_\WW\nabla\Lambda^{-2}\theta$ in $L^1_T(C^\gamma)$. 
Collecting all these estimates and using Gr\"onwall's inequality, we obtain the desired bound of $\|\WW\|_{L^\infty(0,T:C^{1+\gamma})}$ 
that is uniform with respect to $\PPr$,
thereby completing the proof of Theorem \ref{thm:exi-reg}.

The proof of Theorem \ref{thm:limit} relies primarily on the uniform-in-$\PPr$ estimates established in Theorem \ref{thm:exi-reg}, 
together with standard compactness arguments. Since $\nabla u^{\PPr}$ is uniformly bounded in 
$L^\infty(0,T;L^2(\R^3)) \cap L^1(0,T;L^\infty(\R^3))$, 
the Sobolev embedding $\dot{H}^1(\R^3)\hookrightarrow L^6(\R^3)$ and interpolation imply that 
$u^{\PPr}$ is uniformly bounded with respect to $\PPr$ in the space 
$L^{\frac{3}{2}}(0,T;W^{1,6}(\R^3))$, which leads to the weak convergence \eqref{eq:u-weak-conv}. 
From the equation governing $\theta^{\PPr}$ and its uniform boundedness, 
we apply the Aubin-Lions lemma to show the strong convergence of $\theta^{\PPr}$, 
as stated in \eqref{eq:the-str-conv}. 
Consequently, passing to the limit $\PPr \to +\infty$, we deduce that the sequence $(u^{\PPr},\theta^{\PPr})$ 
converges to a limit pair $(u,\theta)$ that solves the 3D Stokes--transport system \eqref{eq:ST} in the distributional sense. 
A more detailed formulation of this convergence result for initial data $(u_0^{\PPr},\theta_0^{\PPr})\in H^1\cap (L^1\cap L^6)$ 
is provided in Proposition \ref{prop:weak-limit}. 
Furthermore, by analyzing the convergence properties of the flow maps $X^{\PPr,\pm1}_t$ and the associated vector fields 
$\WW^{\PPr}$ in suitable functional frameworks, we are able to propagate the geometric structure of patch solutions in the limit $\PPr\rightarrow \infty$. 
As a result, the limit function $\theta$ globally preserves the patch structure, as well as the 
$C^{1,\gamma}$, $W^{2,\infty}$ and $C^{2,\gamma}$ regularity of the patch boundary $\partial D(t)$, 
as stated in Theorem \ref{thm:limit}.
\vskip1mm

The paper is organized as follows. In the next section,
we introduce the notion of an admissible conormal vector system adapted to the temperature patch, together with its useful properties.
We also compile various estimates---including product estimates, commutator estimates, striated estimates, 
and smoothing estimates---in the framework of Besov spaces, and record several auxiliary lemmas.
In Section \ref{sec:exi-uni}, we present the proof for the global existence and uniqueness part of Theorem \ref{thm:exi-reg} 
in the regime of large Prandtl number.
Section \ref{sec:C1gam} is dedicated to establishing the regularity persistence results for patch solutions stated in Theorem \ref{thm:exi-reg}; 
namely, we prove that the boundary regularity of the non-constant temperature patch (also called temperature front) 
in the spaces $C^{1,\gamma}$, $W^{2,\infty}$, $C^{2,\gamma}$ is globally persisted along the evolution of 3D Boussinesq system \eqref{eq:BousEq}. 
Finally, in Section~\ref{sec:limit} we prove the convergence result when the Prandtl number tends to infinity: 
the constructed global strong solution of the 3D Boussinesq system \eqref{eq:BousEq} 
converges to the unique global solution of the 3D Stokes-Transport system \eqref{eq:ST} (Proposition \ref{prop:weak-limit}), 
and moreover, the limit patch solution of \eqref{eq:ST} globally propagates the $C^{1,\gamma}$, $W^{2,\infty}$, 
$C^{2,\gamma}$ regularity of the patch boundary (Theorem \ref{thm:limit}).

\section{Preliminaries}\label{sec:pril}
\subsection{The admissible conormal vector system}
The conormal (i.e. striated) vector fields play an important role in obtaining the boundary regularity persistence of patches. 
Let us first present the definition of an admissible system of conormal vectors (see \cite{GSR95}).
\begin{definition}\label{definition-W}
A system $\mathcal W = (W^1 , W^2 ,\cdots, W^N )$ of $N$ continuous
vector fields is said to be admissible if the function
\begin{align*}
  [\mathcal W]^{-1}\stackrel{\mathrm{def}}{=}\Big(\frac 2{N(N-1)}
  \sum_{\mu<\nu}|W^\mu\wedge W^\nu|^2\Big)^{-\frac {1}{4}} < \infty,
\end{align*}
where the wedge product $X\wedge Y$ is defined as
$X \wedge Y= (X_2Y_3-X_3Y_2,X_3Y_1-X_1Y_3,X_1Y_2-X_2Y_1)^t $ for any two vector fields $X = (X_1 ,X_2 ,X_3 )^t$ and
$Y = (Y_1 ,Y_2 ,Y_3)^t$.
\end{definition}

Let $D_0 \subset \mathbb{R}^3$ be a bounded simple-connected smooth domain with boundary 
$\partial D_0 \in C^{k,\gamma}$, $k\in \mathbb{Z}^+$, $\gamma\in (0,1)$.
Referring to \cite{GSR95}, we have the following result about the existence of an admissible system 
of divergence-free tangential vector fields for a two-dimensional submanifold $\partial D_0$ of $\mathbb{R}^3$.
\begin{proposition}\label{prop:5vec}
  For any compact two-dimensional $C^{k,\gamma}$ submanifold $\partial D_0$ of $\R^3$,
we can find an admissible system $\mathcal{W}_0$ consisting of five divergence-free $C^{k-1,\gamma}$-vector fields tangent to $\partial D_0$.
\end{proposition}

Consider a non-constant temperature patch $\theta_0(x) = \overline{\theta}_0(x) \,\mathbf{1}_{D_0}(x)$, with $D_0\subset \mathbb{R}^3$ 
as above and the patch boundary $\partial D_0$ a two-dimensional compact $C^{k,\gamma}$-submanifold, 
according to Proposition \ref{prop:5vec}, there exists an admissible system $\mathcal W_0 = \{W^1_0, \cdots, W^5_0\}$ such that
\begin{equation}\label{W-i-0}
\begin{split}
  \divg W^i_0=0 \;\;\mathrm{and}\;\; W^i_0\in C^{k-1,\gamma}(\R^3)\,\;\mathrm{are \,\,tangent\,\, to \,\,}\partial D_0,  \quad i = 1,\cdots,5.
\end{split}
\end{equation}

As in the vorticity patch problem for the incompressible Euler equations 
(see e.g., \cite{Chemin93,GSR95}),
we consider the evolution of the vectors $\mathcal{W}(t) = \{W^1(t), \cdots, W^5(t)\}$ 
where $W^i(t)$ is a solution of the following equation
\begin{equation}\label{eq:Wi}
  \partial_t W^i + u \cdot\nabla W^i = W^i\cdot\nabla u = \partial_{W^i}u, 
  \quad\quad W^i|_{t=0}(x)=W^i_0(x),
\end{equation}
for all $i=1,\cdots,5$ and all initial data $W^i_0$ satisfying \eqref{W-i-0}, where the velocity field $u$ is divergence-free and satisfies 
$\eqref{eq:BousEq}_2$.
Since $\divg W^i$ satisfies the transport equation
\begin{equation}\label{eq.divX}
  \partial_t(\divg W^i) + u\cdot\nabla (\divg W^i)=0,\qquad \divg W^i|_{t=0}=\divg W^i_0 =0,
\end{equation}
we see that $W^i(t,x)$ is still divergence-free.
According to \cite[Lemma 1.4]{MB02}, we also have
\begin{align}\label{exp:Wit}
  W^i(t,x) = (\partial_{W^i_0}X_t)\big(X_t^{-1}(x)\big),\quad i=1,\cdots,5,
\end{align}
where $X_t:\R^3\rightarrow \R^3$ is the flow map that satisfies the ODE \eqref{eq:flowmap} with its inverse $X_t^{-1}$.

Referring to \cite[Lemma 2.3]{LLX24} or \cite{GSR95}, 
we have the following key lemma showing the deep relationship between the (higher-order) boundary regularity of $\partial D(t)$
and the striated regularity of the system $\mathcal{W}(t)=\{W^i(t)\}_{1\leq i\leq 5}$. 

\begin{lemma}\label{lem:sr-cond}
 Let $T>0$,  $k\ge 2$ be an integer and $\gamma\in(0,1)$.
Let $X_t(\cdot): \R^3\rightarrow \R^3$ defined by \eqref{eq:flowmap} be the measure-preserving bi-Lipschitzian flow map on $[0,T]$.
Then, the temperature patch boundary $\partial D(t)=X_t( \partial D_0)$ preserves its 
$C^{k,\gamma} (\text{or}\,\,W^{k, \infty})$-regularity on the time interval $[0,T]$, provided that
\begin{align}\label{eq:targ-sr}
  \textrm{$\partial^\ell_{\mathcal{W}} \mathcal{W} \in L^\infty([0,T]; C^\gamma(\R^3))\; (\text{or}\,\,L^\infty([0,T]; L^\infty(\R^3)))$ 
  \quad for all \quad  $0\leq \ell\leq k-1$.}
\end{align}
\end{lemma}
\noindent In the above, we have used the vector-valued notation $\partial_{\mathcal{W}}=\mathcal{W} \cdot \nabla 
= \left\{W^i \cdot \nabla\right\}_{1 \leq i \leq N}$,
and $\partial_{\mathcal{W}}^\lambda=\big\{\partial_{W^1}^{\lambda_1} \cdots \partial_{W^N}^{\lambda_N}: \lambda_1+
\cdots+\lambda_N=\lambda, \lambda_i \in \mathbb{N}\big\}$
for all $\lambda \in \mathbb{N}$.

The following lemma presents the striated estimate of the initial temperature patch, 
which can be proved as \cite[Lemma 2.4]{LLX24} or \cite[Lemma 2.6]{CMX22} 
(note that it requires that $\partial D_0\in C^{k,\gamma}$ in \cite{CMX22,LLX24}, 
but the condition $\partial D_0\in W^{k-1,\infty}$ 
suffices to get the conclusion by checking the proof).
\begin{lemma}\label{lem:str-reg}
Let $k\geq 2$ and $\gamma\in(0,1)$. 
Assume that $D_0 \subset \DD$ is a bounded simply connected domain with boundary
$\partial D_0\in W^{k-1,\infty}$ and $\theta_0(x) = \overline{\theta}_0(x) \mathbf{1}_{D_0}(x)$ with 
$\overline{\theta}_0 \in C^{k-2,\mu}(\overline{D_0})$.
%\footnote{Here for a Lipschitz domain $\Omega\subset \R^d$, the Besov space $B^{s,n}_{p,r}(\Omega)$
%consists of those distributions $f\in \mathcal{D}'(\Omega)$ which have extensions beyond $\Omega$ belonging to $B^{s,n}_{p,r}(\R^d)$,
%and it is endowed with the quotient-space quasi-norms (e.g. see \cite{Ryc99}).}
Let $\mathcal{W}_0 = \{W^i_0\}_{1\leq i\leq 5}$ be defined as in \eqref{W-i-0}. 
Then we have 
%\begin{align}\label{pWj-theta}
%  \partial_{\WW_0}^{j} \theta_0 \in L^1\cap L^{\infty}(\R^3),\quad 
%  \forall j\in \{0,1,2,\cdots,k-2\},
%\end{align}
%and 
\begin{align}\label{pWk-theta}
  \partial_{\WW_0}^{k-1} \theta_0 \in C^{-1,\mu}(\R^3) = B^{\mu-1}_{\infty,\infty}(\mathbb{R}^3).
\end{align}
\end{lemma}

\subsection{Besov spaces and related estimates}
One can choose two nonnegative radial functions $\chi, \varphi \in C_c^{\infty}\left(\mathbb{R}^d\right)$
be supported respectively in the ball $\left\{\xi \in \mathbb{R}^d:|\xi| \leq \frac{4}{3}\right\}$
and the annulus $\big\{\xi \in \mathbb{R}^d: \frac{3}{4}\leq|\xi| \leq \frac{8}{3}\big\}$
such that (see \cite{BCD11})
\begin{equation*}
  \chi(\xi)+\sum_{j \geq 0} \varphi\left(2^{-j} \xi\right)=1 \quad \forall\xi \in \mathbb{R}^d, 
  \quad \textrm{and} \quad  
  \sum_{j \in \mathbb{Z}} \varphi\left(2^{-j} \xi\right)=1 \quad \forall \xi \in \mathbb{R}^d\backslash\{0\}.
\end{equation*}
For all tempered distribution $f$, the dyadic block operators $\Delta_j$, $S_j$ and 
$\dot \Delta_j$ are defined by
\begin{equation}\label{eq:Del-Sj}
\begin{aligned}
  & \Delta_{-1} f=\chi(D) f, \qquad \Delta_j f=\varphi\left(2^{-j} D\right) f, 
  \quad \forall j \in \mathbb{N}, \\
  & S_j f=\chi\left(2^{-j} D\right) f = \sum_{-1 \leq l \leq j-1} \Delta_l f, 
  \quad \forall j \in \mathbb{N},\quad
  \dot{\Delta}_j f=\varphi\left(2^{-j} D\right) f, \quad \forall j \in \mathbb{Z}.
\end{aligned}
\end{equation}
%where $h_j(\cdot) := 2^{jd} h(2^j \cdot)$, $h:= \mathcal{F}^{-1} \varphi\in \mathcal{S}(\mathbb{R}^d)$,
%$\overline{h}_j(\cdot):= 2^{j d}
%\overline{h}(2^j \cdot)$, $\overline{h}:=\mathcal{F}^{-1} \chi\in \mathcal{S}(\mathbb{R}^d)$.
%and $\mathcal{F}^{-1}$ is the Fourier inverse transform.

For all $f, g \in \mathcal{S}^{\prime}\left(\mathbb{R}^d\right)$ (the space of tempered distributions), 
we have Bony's decomposition:
\begin{equation*}
  f g = T_f g+T_g f+R(f, g),
\end{equation*}
with
\begin{equation*}
  T_f g:= \sum_{q \in \mathbb{N}} S_{q-1} f \Delta_q g,
  \quad R(f, g):= \sum_{q \geq-1} \Delta_q f \widetilde{\Delta}_q g,
  \quad \widetilde{\Delta}_q:= \Delta_{q-1}+\Delta_q+\Delta_{q+1}.
\end{equation*}
%In what follows, for a vector field $W: \mathbb{R}^d \rightarrow \mathbb{R}^d$,
%we also use the notation $T_{W \cdot \nabla}$ to denote the operator 
%$\sum_{q \in \mathbb{N}} S_{q-1} W \cdot \nabla \Delta_q$. \\

Now, we introduce the nonhomogeneous/homogeneous Besov spaces $B_{p, r}^s\left(\mathbb{R}^d\right)$ and $\dot{B}_{p, r}^s\left(\mathbb{R}^d\right)$.
\begin{definition}\label{def:Besov-str}
Let $s\in \R$, $(p,r)\in [1,\infty]^2$. Denote by $B^s_{p,r}= B^s_{p,r}(\R^d)$ the space of tempered distributions 
$f\in \mathcal{S}'(\R^d)$ such that
\begin{align*}
  \|f\|_{B^s_{p,r}} := \big\| \big\{2^{qs}  \|\Delta_q f\|_{L^p}\big\}_{q\geq -1}  \big\|_{\ell^r}  < \infty.
\end{align*}
Denote by $\dot{B}^s_{p,r}= \dot{B}^s_{p,r}(\R^d)$ the space of tempered distributions 
$f\in \mathcal{S}'(\R^d)/\mathcal{P}(\R^d)$ such that
\begin{align*}
  \|f\|_{\dot{B}^s_{p,r}} := \big\| \big\{2^{qs}  \|\dot{\Delta}_q f\|_{L^p}\big\}_{q\in\mathbb{Z}}  \big\|_{\ell^r}  
  < \infty.
\end{align*}
\end{definition}

We also use the Chemin-Lerner mixed space-time Besov space
$\widetilde{L}^\rho\left([0, T] , B_{p, r}^s\right)$. which is the set of tempered distribution $g$ such that
$\|g\|_{\widetilde{L}_T^\rho\left(B_{p, r}^s\right)} := \big\|\big(2^{q s}\|\Delta_q g\|_{L_T^\rho(L^p)}\big)_{q \geq-1}\big\|_{\ell^r}<\infty$. 
Compared with the usual space-time norm
$\|g\|_{L^\rho_T (B^s_{p,r})} % =  \left\|  \left\| u \right\|_{(B^s_{p,r})_x}  \right\|_{L^\rho_t} 
  = \big\| \big\| \big(2^{sq}\left\| \Delta_q g \right\|_{L^p_x}\big)_{q\geq -1}\big\|_{\ell^r_q}  \big\|_{L^\rho_T}$,
from Minkowski's inequality we have
\begin{align*}
  L_{T}^{\rho} (B_{p, r}^s) \hookrightarrow \widetilde{L}_{T}^\rho (B_{p, r}^s),
  \quad \text { if }\; r \ge\rho,  \quad\quad
  L_T^\rho (B_{p, r}^s) \hookrightarrow \widetilde{L}_T^\rho (B_{p, r}^s),
  \quad \text { if }\; \rho \ge r.
\end{align*} 
In particular, $\widetilde{L}_T^r (B_{p, r}^s) = L_T^r (B_{p, r}^s)$. \vskip1mm

The following product estimates, commutator estimates and first-order striated estimate play an important role in the main proof.
\begin{lemma}\label{lem:str-es1}
Assume that $u$ is a smooth divergence-free vector field of $\mathbb{R}^d~(d \geq 2)$
and~$\mathcal{W}=\left\{W_i\right\}_{1 \leq i \leq N}$ ($N \in \mathbb{Z}^{+}$) is a set of smooth divergence-free vector fields.
Let $\phi: \mathbb{R}^d \rightarrow \mathbb{R}$ be a smooth function.
Let $m(D) = \Lambda^\sigma m_0(D)$, $\sigma\in \mathbb{R}$ and $m_0(D)$ be a zero-order pseudo-differential operator with
$m_0(\xi) \in C^{\infty}\left(\mathbb{R}^d \backslash\{0\}\right)$.
Then the following statements are true.
\begin{enumerate}[(1)]
\item For all $\epsilon \in(0,1)$ and $(p, r) \in[1, \infty]^2$, there exists a constant $C=C(d, \epsilon)>0$ such that
\begin{equation}\label{eq:prod-es}
  \|u \cdot \nabla \phi\|_{B_{p, r}^{-\epsilon}}
  \leq C \min \left\{\|u\|_{B_{p, r}^{-\epsilon}}\|\nabla \phi\|_{L^{\infty}},\|u\|_{L^{\infty}}\|\nabla \phi\|_{B_{p, r}^{-\epsilon}}\right\},
\end{equation}
and
\begin{equation}\label{eq:prodBes-endp}
  \|u \cdot \nabla \phi\|_{B_{p, r}^{-1}}
  \leq C \min \left\{\|u\|_{B_{p, r}^{-1}} \|\nabla \phi\|_{B^0_{\infty,1}}, 
  \|u\|_{B^0_{\infty,1}} \|\nabla \phi\|_{B_{p, r}^{-1}}\right\}.
\end{equation}
\item For all $s \in(-1,1)$, $\sigma>-1$, $-1<\sigma+ s <1$ and $(p, r) \in[1, \infty]^2$, there exists a constant $C=C(d, s, \sigma)>0$ so that
\begin{equation}\label{eq:prod-es5}
  \|[m(D), u \cdot \nabla] \phi\|_{B_{p, r}^s} \leq C\|u\|_{W^{1, \infty}}\|\phi\|_{B_{p, r}^{s + \sigma}}.
\end{equation}
\item For all $-d < \sigma <0$, there exists a constant $C = C(d,\sigma)>0$ such that
\begin{equation}\label{es:prod-es6}
  \left\|[ m(D), u \cdot \nabla] \phi\right\|_{L^\infty}
  \leq C \|\nabla u\|_{L^\infty}\|\phi\|_{L^1\cap L^\infty},
\end{equation}
and
\begin{equation}\label{es:prod-es7}
\begin{aligned}
  \|\partial_{\WW}([m(D),u\cdot\nabla]\phi)\|_{L^\infty} 
  & \leq C\bigl( \|\nabla\WW\|_{L^\infty}\|\nabla u\|_{L^\infty} 
  + \|\nabla(\partial_{\WW}u)\|_{L^\infty} \bigr) \|\phi\|_{L^1\cap L^\infty} \\
  & \quad + C\|\nabla u\|_{L^\infty}\|\partial_{\WW}\phi\|_{L^1\cap L^\infty}.
\end{aligned}
\end{equation}
\item For all $s \in(-1,1)$, $\sigma>-1$, $-1<\sigma+ s <1$ and $(p, r) \in[1, \infty]^2$, there exists a constant $C=C(d, s, \sigma)>0$ so that
\begin{equation}\label{eq:prod-es4}
  \left\|\partial_{\mathcal{W}}(m(D) \phi)\right\|_{B_{p,r}^s}
  \leq C\left\|\partial_{\mathcal{W}} \phi\right\|_{B_{p, r}^{s +\sigma}}+C\|\mathcal{W}\|_{W^{1, \infty}}\|\phi\|_{B_{p, r}^{s+\sigma}}.
\end{equation}
\end{enumerate}
\end{lemma}

\begin{proof}[Proof of Lemma \ref{lem:str-es1}]
One can see \cite[Lemma 2.5]{CMX22} for the proof of \eqref{eq:prod-es} and 
\cite[Lemma 2.3]{LXY25} for the proof of \eqref{eq:prod-es5} and \eqref{eq:prod-es4}.
The proof of \eqref{eq:prodBes-endp} 
can be done using Bony's decomposition and arguing similarly as that of \eqref{eq:prod-es}, 
and we here omit the details (see e.g. \cite{LXY26} for the proof).

For the proof of \eqref{es:prod-es6}, noting that
\begin{align*}
  m(D) \phi(x) =  K_\sigma\ast \phi(x)=\int_{\mathbb{R}^d} K_\sigma(x-y) \phi(y) \dd y, 
  \quad \textrm{with}\;\; K_\sigma(x)=\frac{\Omega(x/|x|)}{|x|^{d+\sigma}},
\end{align*}
and $\Omega \in C^\infty(\mathbb{S}^{d-1})$,
we indeed have (see \cite{LXY26} for the detailed proof)
\begin{equation}\label{eq:comm-u-phi}
  [m(D), u \cdot \nabla] \phi 
  = \int_{\R^d} (u(y)-u(x))\cdot(\nabla K_\sigma)(x-y)\phi(y)\,\dd y,
\end{equation}
%and
%\begin{align*}%\label{eq:pWcomm-u-phi}
%  \partial_{\mathcal{W}}\big([m(D), u \cdot \nabla] \phi\big)
%  = &\, \int_{\R^d} (u(y)-u(x))\cdot (\nabla^2 K_\sigma)(x-y)\cdot (\mathcal{W}(x) - \mathcal{W}(y))\phi(y)\,\dd y\\
%  & + \int_{\R^d} (\partial_{\mathcal{W}}u(y) - \partial_{\mathcal{W}}u(x))\cdot (\nabla K_\sigma)(x-y)\phi(y)\,\dd y\\
%  & + \int_{\R^d} (u(y) - u(x))\cdot (\nabla K_\sigma)(x-y)\partial_{\mathcal{W}}\phi(y)\,\dd y.
%\end{align*}
which immediately leads to the desired estimates that for $-d <\sigma <0$,
\begin{align*}
  \|[m(D),u\cdot\nabla]\phi\|_{L^\infty} \leq C \|\nabla u\|_{L^\infty} 
  \Big\|\int_{\mathbb{R}^d} \frac{1}{|x-y|^{d+\sigma}} \phi(y) \dd y\Big\|_{L^\infty_x}
  \leq C \|\nabla u\|_{L^\infty} \|\phi\|_{L^1\cap L^\infty}.
\end{align*}
%and
%\begin{align*}
%  &\left\|\partial_{\mathcal{W}}[m(D), u \cdot \nabla] \phi\right\|_{L^\infty}\\
%  &\leq C \big(\|\nabla\mathcal{W}\|_{L^\infty}\|\nabla u\|_{L^\infty} 
%  + \left\|\nabla \partial_{\mathcal{W}} u\right\|_{L^\infty} \big) 
%  \left\|\int_{\R^d}\frac{|\phi(y)|}{|x-y|^{d+\sigma}}\dd y\right\|_{L^\infty_x}
%  +C \|\nabla u\|_{L^\infty}\left\|\int_{\R^d} 
%  \frac{|\partial_{\mathcal{W}}\phi(y)|}{|x-y|^{d+\sigma}} \dd y\right\|_{L^\infty_x}\\
%  &\leq C \left(\|\nabla\mathcal{W}\|_{L^\infty}\|\nabla u\|_{L^\infty}
%  + \left\|\nabla \partial_{\mathcal{W}} u\right\|_{L^\infty} \right)\|\phi\|_{L^1\cap L^\infty} 
%  + C \|\nabla u\|_{L^\infty}\|\partial_{\mathcal{W}}\phi\|_{L^1\cap L^\infty}.
%\end{align*}

The proof of \eqref{es:prod-es7} is in a similar manner as that of \eqref{es:prod-es6}: 
by deriving a striated version of the expression formula \eqref{eq:comm-u-phi}
for $\partial_{\WW}([m(D),u\cdot\nabla]\phi)$, 
the wanted estimate \eqref{es:prod-es7} directly follows. 
One can see \cite{LXY26} for more details.
\end{proof}

We have the following useful commutator estimates involving the operator $\mathcal{R}_{-1}$. %One can refer to \cite[Lemma 2.6]{LXY25} for the proof.
\begin{lemma}\label{lem:Rbeta-cm}
Let $d\geq 2$.
Let $(p, r) \in[2, \infty] \times[1, \infty]$, $\mathcal{R}_{-1} := \Lambda^{-1} m(D)$, 
and $m(D)$ be a zero-order pseudo-differential operator with its Fourier symbol 
$m(\xi)\in C^\infty(\mathbb{R}^d\setminus \{0\})$.
Assume that $u=(u_1,\cdots,u_d)$ is a smooth divergence-free vector field on $\mathbb{R}^d$
and $\phi$ is a smooth scalar function in $\mathbb{R}^d$.
%If $\beta\in (1,2)$, we have that for all $s \in (\beta-2, 1)$,
%\begin{align}\label{eq:Rbeta-cm-es}
%  \left\|\left[\mathcal{R}_{1-\beta}, u \cdot \nabla\right] \phi\right\|_{B_{p, r}^s}
%  \leq C_{s, \beta,d}\,\|\nabla u\|_{L^p}\left(\|\phi\|_{B_{\infty, r}^{s+1-\beta}}+\|\phi\|_{L^2}\right) ,
%\end{align}
We have that for all $s \in (0,1)$ and $q\in (1,\infty)$,
\begin{equation}\label{eq:R-1cm-es1}
  \|[\mathcal{R}_{-1}, u\cdot\nabla] \phi\|_{B^s_{p,r}}
  \leq C_{s,q,d} \,\Big( \|\nabla u\|_{L^p} \big( \|\phi\|_{B^{s - 1}_{\infty,r}} + \|\phi\|_{L^2} \big)
  + \|u\|_{L^q} \|\phi\|_{L^{\frac{q}{q-1}}} \Big).
\end{equation}
%Besides, if $p=\infty$, we also have
%\begin{align*}
%  \left\|\left[\mathcal{R}_\beta, u \cdot \nabla\right] \phi\right\|_{B_{\infty, r}^s}
%  \le C _{s, \beta}\left(\|\omega\|_{L^{\infty}}
%  + \|u\|_{L^2}\right)\|\phi\|_{B_{\infty, r}^{s+\frac{(1-\beta)}{2} }}+\|u\|_{L^2}\|\phi\|_{L^2}.
%\end{align*}
\end{lemma}

\begin{proof}[Proof of Lemma \ref{lem:Rbeta-cm}]
%We can refer to \cite[Lemma 2.6]{LXY25} for the proof of \eqref{eq:Rbeta-cm-es} (although in \cite{LXY25} it treated the special case 
%$\mathcal{R}_{1-\beta} = \partial_1 \Lambda^\beta $, the proof directly extends to this general case). 
The proof of \eqref{eq:R-1cm-es1} follows a similar approach as \cite[Lemma 0.1]{CMX22E} with small modification. 
Taking advantages of Bony's decomposition, we write
\begin{align*}
  {\left[\mathcal{R}_{-1}, u \cdot \nabla\right] \phi} 
  & =\sum_{q \in \mathbb{N}}\left[\mathcal{R}_{-1}, S_{q-1} u \cdot \nabla\right] \Delta_q \phi
  + \sum_{q \in \mathbb{N}}\left[\mathcal{R}_{-1}, \Delta_q u \cdot \nabla\right] S_{q-1} \phi 
  + \sum_{q \geq-1}\left[\mathcal{R}_{-1}, \Delta_q u \cdot \nabla\right] \widetilde{\Delta}_q \phi \\
  & =: \mathrm{I}+\mathrm{II}+\mathrm{III}.
\end{align*}
Since $\text{div}~u=0$, we also have
\begin{align*}
  \mathrm{III} & =\sum_{q \geq 3} \mathcal{R}_{-1} \nabla \cdot \big(\Delta_q u \widetilde{\Delta}_q \phi\big)
  -\sum_{q \geq 3} \Delta_q u \cdot \nabla \mathcal{R}_{-1} \widetilde{\Delta}_q \phi 
  +\sum_{-1 \leq q \leq 2}\left[\mathcal{R}_{-1} \nabla \cdot, \Delta_q u\right] \widetilde{\Delta}_q \phi \\
  & =: \mathrm{III}_1+\mathrm{III}_2+\mathrm{III}_3 .
\end{align*}
The estimates for $\mathrm{I}$, $\mathrm{II}$, $\mathrm{III}_1$ and $\mathrm{III}_2$ 
are the same as those in \cite[Lemma 0.1]{CMX22E}, which give that
\begin{align*}
  \left\|\mathrm{I}\right\|_{B_{p, r}^s}+\left\|\mathrm{II}\right\|_{B_{p, r}^s} 
  +\left\|\mathrm{III}_1\right\|_{B_{p, r}^s}+\left\|\mathrm{III}_2\right\|_{B_{p, r}^s} 
  \leq C_s\|\nabla u\|_{L^p} \big(\|\phi\|_{B_{\infty,r}^{s-1}}+\|\phi\|_{L^2}\big).
\end{align*}
As for the remaining term $\mathrm{III}_3$, using the Calder\'on-Zygmund theorem 
for the singular integral operator and H\"older's inequality, we derive that
\begin{align*}
  \left\|\mathrm{III}_3\right\|_{B_{p, r}^s} 
  & \leq C_s \sum_{-1 \leq j \leq 6} \sum_{-1 \leq q \leq 2} 
  \left(\big\|\Delta_j \mathcal{R}_{-1} \nabla \cdot
  \big(\Delta_q u \,\widetilde{\Delta}_q \phi\big) \big\|_{L^p}
  +\big\|\Delta_q u \cdot \nabla \mathcal{R}_{-1} \widetilde{\Delta}_q \phi\big\|_{L^p}\right) \\
  & \leq C_s \sum_{-1 \leq q \leq 2}\left(\big\|\Delta_q u \, \widetilde{\Delta}_q \phi\big\|_{L^1}
  +\big\|\Delta_q u \cdot \nabla \mathcal{R}_{-1} \widetilde{\Delta}_q \phi\big\|_{L^1}\right) \\
  & \leq C_{s,q} \sum_{-1 \leq q \leq 2}\big\|\Delta_q u\big\|_{L^q} 
  \big\|\widetilde{\Delta}_q \phi\big\|_{L^{\frac{q}{q-1}}} \leq C_{s,q}\,\|u\|_{L^q}\|\phi\|_{L^{\frac{q}{q-1}}}.
\end{align*}
Collecting the above estimates leads to \eqref{eq:R-1cm-es1}, as desired.
\end{proof}

We have the following smoothing estimates for the transport and transport-diffusion equations.
\begin{lemma}\label{lem:TD-sm2}
  Assume $(\rho, r,p)\in [1,\infty]^3,\,\alpha\in (0,1]$ and $-1< s < 1$. 
Let $u$ be a smooth divergence-free vector field and $\phi$ be a smooth function solving the following equation
\begin{equation*}%\label{eq:TD-eq2}
  \partial_t \phi + u\cdot\nabla \phi + \nu \Lambda^{2\alpha} \phi = f,\quad \phi|_{t=0}(x)=\phi_0(x),\quad x\in \R^d.
\end{equation*}
The following statements hold.
\begin{enumerate}[(1)]
\item
If $\nu=0$, then there exists a constant $C= C(d,s)>0$ such that for any $t>0$, 
\begin{equation}\label{eq:T-sm2}
  \|\phi\|_{L_t^{\infty}\left(B_{p, r}^s\right)} \leq C\left(\left\|\phi_0\right\|_{B_{p, r}^s}
  +\|f\|_{\widetilde{L}_t^1\left(B_{p, r}^s\right)} 
  +\int_0^t\|\nabla u(\tau)\|_{L^{\infty}}\|\phi(\tau)\|_{B_{p, r}^s} \mathrm{~d} \tau\right),
\end{equation}
and
\begin{equation}\label{eq:T-sm3}
  \|\phi\|_{L^\infty_t (B^{s}_{p,r})} \leq C e^{C \int_0^t \|\nabla u(\tau)\|_{L^\infty} \dd \tau} 
  \Big( \|\phi_0\|_{B^{s}_{p,r}} + \|f\|_{\widetilde{L}^1_t (B^{s}_{p,r})} \Big).
\end{equation}
%\item 
%If $\nu>0$, then there exists a constant $C = C(d,s)>0$ such that for any $t>0$,
%\begin{equation}\label{es.sm1}
%  \nu^{\frac{1}{\rho}}  \|\phi\|_{\widetilde{L}^\rho_t (B^{s+ \frac{2\alpha}{\rho}}_{p,r})} 
%  \leq C (1+\nu t)^{\frac{1}{\rho}}
%  e^{C\int_0^t\|\nabla u(\tau)\|_{L^\infty}\dd\tau}\big( \|\phi_0\|_{B^{s}_{p,r}} 
%  + \|f\|_{L^1_t (B^{s}_{p,r})} \big).
%\end{equation}
\item 
If $\nu>0$, then there exists a constant $C=C(d,s,\alpha)>0$ such that for any $t>0$,
\begin{equation}\label{TD-sm-es}
  \nu^{\frac{1}{\rho}} \Big\|\Big(2^{q(s+\frac{2\alpha}{\rho})} 
  \|\Delta_q\phi\|_{L_t^\rho(L^p)}\Big)_{q\in \mathbb{N}}\Big\|_{\ell^r}
  \leq C e^{C \int_0^t\|\nabla u(\tau)\|_{L^{\infty}} \dd \tau}
  \left(\|\phi_0\|_{B^s_{p,r}}+\|f\|_{L_t^1(B^s_{p,r})}\right).
\end{equation}
In particular, if $s= -1$, $r=\infty$, there exists $C=C(d,\alpha)>0$ independent of $\nu$ such that for all $t>0$, 
\begin{align}\label{eq:TD-sm5}
  \nu^{\frac{1}{\rho}} \sup_{q\in\mathbb{N}} 2^{q(-1+\frac{2\alpha}{\rho})}\|\Delta_q\phi\|_{L_t^\rho(L^p)}
  \leq C e^{C \int_0^t\|u(\tau)\|_{B^1_{\infty,1}} \dd \tau}
  \left(\|\phi_0\|_{B^{-1}_{p,\infty}} + \|f\|_{L_t^1(B^{-1}_{p,\infty})} \right).
\end{align}
\end{enumerate}
\end{lemma}

\begin{proof}[Proof of Lemma \ref{lem:TD-sm2}]
One can see \cite[Lemma 2.8, Lemma 2.9]{LXY25} for the proof of 
\eqref{eq:T-sm2}-\eqref{eq:T-sm3} and \eqref{TD-sm-es}.
For the proof of \eqref{eq:TD-sm5}, it can be treated in a similar manner as showing \eqref{TD-sm-es}, 
and we only need to replace the estimate of $\| (2^{qs} \|[\Delta_q, u\cdot \nabla] \phi\|_{L^p})_{q\geq -1}\|_{\ell^r}$ (as in \cite[Lemma 2.100]{BCD11})
with the following commutator estimate
\begin{equation}\label{eq:comm-es-end}
  \sup_{q\ge -1} 2^{-q}\,\bigl\|[\Delta_q,\; u\cdot\nabla]\phi\bigr\|_{L^p}
  \leq C \|u\|_{B^{1}_{\infty,1}}\;\|\phi\|_{B^{-1}_{p,\infty}};
\end{equation}
the proof of \eqref{eq:comm-es-end} can be done using the standard Littlewood-Paley theory and one can see \cite{LXY26} for the details.
\end{proof}

\subsection{Auxiliary lemmas}
Some basic facts on the flow map are collected below (see e.g. \cite[Proposition 3.10]{BCD11}).
\begin{lemma}\label{lem:flow}
Assume that $u$ is a divergence-free velocity field belonging to $L^1([0,T]; \dot W^{1,\infty}(\R^d))$. Let $X_t(x)$ be the flow map generated by velocity $u$ which solves that
\begin{equation}\label{eq:flow1}
    \frac{\partial X_t(x)}{\partial t} = u ( X_t(x),t),\quad X_t(x)|_{t=0}=x,
\end{equation}
that is
\begin{equation}\label{eq:flow2}
  X_t(x) = x + \int_0^t u(X_\tau(x),\tau) \dd \tau.
\end{equation}
Then the system \eqref{eq:flow1} has a unique solution $X_t(\cdot):\R^d\mapsto \R^d$ on $[0,T]$ 
which is a volume-preserving bi-Lipschitzian homeomorphism and satisfies that
$\nabla X_t$ and its inverse $\nabla X^{-1}_t$
belong to $L^\infty([0,T]\times \R^d)$ with
\begin{equation}\label{DXest}
  \|\nabla X^{\pm1}_t\|_{L^\infty(\R^d)} \leq C e^{\int_0^t \|\nabla u(\tau)\|_{L^\infty}\dd \tau}.
\end{equation}
In addition, the following statements hold true.
\begin{enumerate}[(1)]
\item If  $u\in L^1([0,T]; C^{1,\gamma}(\R^d))$, then $\nabla X^{\pm1}_t \in L^\infty([0,T]; C^{\gamma}(\R^d))$ with
\begin{equation}\label{X-C1gam-es}
  \|\nabla X^{\pm 1}_t\|_{C^\gamma} \leq C  e^{(2+\gamma)\int_0^t \|\nabla u\|_{L^\infty} \dd \tau} 
  \Big( 1 + \int_0^t \|\nabla u(\tau)\|_{C^\gamma} \dd \tau \Big).
\end{equation}
\item If  $u\in L^1([0,T]; W^{2,\infty}(\R^d))$, then $\nabla X^{\pm1}_t \in L^\infty([0,T];  W^{1,\infty}(\R^d))$ with
\begin{equation}\label{X-W2inf-es}
  \|\nabla^2 X^{\pm 1}_t\|_{L^\infty} \leq C  e^{3\int_0^t \|\nabla u\|_{L^\infty} \dd \tau} 
  \int_0^t \|\nabla^2 u(\tau)\|_{L^\infty} \dd \tau.
\end{equation}
\end{enumerate}
\end{lemma}

We refer to Lemma 6.10 of \cite{HKR10} for this useful commutator estimate as follows.
\begin{lemma}\label{lem:HKR}
Let $d\geq 2$. Let $u$ be a smooth divergence-free vector field of $\R^d$ and $f$ be a smooth scalar function. 
Then, for all $p\in [1,\infty]$ and $j\geq -1$,
\begin{align*}
  \|[\Delta_j, u\cdot\nabla]f\|_{L^p}\leq C \|\nabla u\|_{L^p}\|f\|_{B^{0}_{\infty,\infty}}.
\end{align*}
\end{lemma}

The following compactness lemma is useful in the proof of the limiting result of infinite Prandtl number.
\begin{lemma}[Aubin-Lions's lemma \cite{L1969}]\label{lem:Aubin-Lions}
Assume that $X_0\hookrightarrow X\hookrightarrow X_1$ are Banach spaces and that 
$X_0$ is compactly embedded in $X$ and
$X $ is continuously embedded in $X_1$. For all $1\leq p,q\leq \infty$, let
\begin{align*}
  V := \left\{u\in L_T^p(X_0)\;:\; \partial_tu\in L_T^q(X_1) \right\}.
\end{align*}
Then we have
\begin{enumerate}
\item if $p<\infty$, then $V$ is compactly embedded in $L_T^p(X)$;
\item if $p=\infty$ and $q>1$, then $V$ is compactly embedded in $C([0,T];X)$.
\end{enumerate}
\end{lemma}

\section{Existence and uniqueness result of 3D Boussinesq system with large Prandtl number}\label{sec:exi-uni}

%This section is closely related to \cite[Section 3]{LLX24} 
%where the author proved the existence and uniqueness result of 3D Boussinesq system 
%\eqref{eq:BousEq} with $\mathrm{Pr}=1$. 
%Based on this, we only need to prove the global existence of solutions under the condition \eqref{eq:data-cond}. 
%However, to prepare for the subsequent sections, we recompute some \textit{a priori} 
%estimates to get the dependence on $\mathrm{Pr}$.

We first focus on establishing crucial \textit{a priori} estimates assuming that 
$(u,\theta)$ is a smooth solution to the 3D Boussinesq system \eqref{eq:BousEq}.

From the transport equation $\eqref{eq:BousEq}_1$, we have that for every $s\in (3,\infty]$,
\begin{equation}\label{es0.thet}
  \Vert \theta(t) \Vert_{L^1\cap L^s } \leq \Vert\theta_0\Vert_{L^1\cap L^s},\quad \forall t\geq 0.
\end{equation}
Now we pay attention to the estimates for velocity $u$. 
The basic $L^2$-energy estimate for the velocity equation $\eqref{eq:BousEq}_2$ yields
\begin{align*}
  \frac{1}{2 \mathrm{Pr}} \frac{\dd }{\dd t} \|u\|_{L^2}^2 + \|\nabla u\|_{L^2}^2  
  \leq \Big|\int_{\R^3} u^3\, \theta \dd x \Big| \leq \|u\|_{L^2} \|\theta\|_{L^2} 
  \leq \|u\|_{L^2} \|\theta_0\|_{L^2},
\end{align*}
which implies that 
\begin{align*} 
  \|u\|_{L^\infty_T(L^2)} \leq  \|u_0\|_{L^2} + T \,\mathrm{Pr} \|\theta_0\|_{L^2},
\end{align*}
and
\begin{equation}\label{es.v.L2}
 \|u\|_{L^\infty_T(L^2)}^2 + 2 \mathrm{Pr} \|\nabla u\|_{L^2_T(L^2)}^2 
 \leq \|u_0\|_{L^2}^2 + 2T \mathrm{Pr} \|\theta_0\|_{L^2} \big(\|u_0\|_{L^2} 
 +  T \mathrm{Pr} \|\theta_0\|_{L^2}\big).
\end{equation}

%Using the idea of Danchin and Pa\"icu \cite{DanP08b} (see ), 
%we know that under the condition \eqref{eq:data-cond}, 
%there exists an absolute constant $C_0>0$ such that for any $T>0$,
%\begin{equation}
%  \|u\|_{L^\infty_T(L^{3,\infty})} \leq C_0 \big( \|u_0\|_{L^{3,\infty}} + \|\theta_0\|_{L^{1}} \big) 
%  \leq C_0 \frac{ c_*}{\varepsilon} .
%\end{equation}
The velocity equation $\eqref{eq:BousEq}_2$ writes 
\begin{align*}
  \partial_t u  - \mathrm{Pr}\, \Delta u =  \mathbb{P}\, \mathrm{div}\,(u \otimes u) + \mathrm{Pr} \,\mathbb{P}(\theta e_3),
  \quad u|_{t=0} = u_0,
\end{align*}
where $\mathbb{P} := \mathrm{Id} - \nabla \Delta^{-1} \mathrm{div}\,$ is the Leray projection operator,
thus we have
\begin{align*}
  u(t)=e^{t \mathrm{Pr} \Delta} u_0 
  + \int_0^t e^{(t-\tau) \mathrm{Pr} \,\Delta}\mathbb{P}\mathrm{div}(u\otimes u) \dd \tau 
  + \mathrm{Pr} \int_0^t e^{(t-\tau) \mathrm{Pr}\, \Delta} \mathbb{P}(\theta e_3) \dd \tau.
\end{align*}
Arguing as \cite[Section 5.2.2]{DanP08b}, we find
\begin{equation}\label{eq:C}
\begin{aligned}
  \|u\|_{L_t^\infty(L^{3,\infty})} &\leq C_0 \big(\|u_0\|_{L^{3,\infty}} 
  + \tfrac{1}{\mathrm{Pr}}\|u\|^2_{L_t^\infty(L^{3,\infty})}+\|\theta\|_{L_t^\infty(L^{1})} \big) \\
  & \leq C_0 \big(\|u_0\|_{L^{3,\infty}} + \tfrac{1}{\mathrm{Pr}}\|u\|^2_{L_t^\infty(L^{3,\infty})} 
  +\|\theta_0\|_{L^{1}} \big),
\end{aligned}
\end{equation}
with $C_0>0$ a universal constant.
Hence, by using the bootstrap argument, 
provided that
\begin{equation}\label{eq:c*}
  \|u_0\|_{L^{3,\infty}(\mathbb{R}^3)} + \|\theta_0\|_{L^1(\mathbb{R}^3)}  \leq  c_* \mathrm{Pr}, 
\end{equation}  
where $c_*>0$ is a universal small constant (e.g. $c_* \leq \frac{1}{6 C_0^2}$, $C_0$ is the constant in \eqref{eq:C}), 
we obtain that for any $T>0$,
\begin{align}\label{eq:vL3inf-es}
  \|u\|_{L^\infty_T(L^{3,\infty})}\leq 2C(\|u_0\|_{L^{3,\infty}} + \|\theta_0\|_{L^{1}}) \leq  2 C_0 c_* \mathrm{Pr}.
\end{align}

\begin{comment}
  Let $f(x)=C\varepsilon x^2-x+C(\|u_0\|_{L^{3,\infty}}+\|\theta_0\|_{L^{1}})=ax^2+bx+c$ where $C>1$, 
we have $f(\|u_0\|_{L^{3,\infty}})>0,$ so $\|u\|_{L_T^\infty (L^{3,\infty})}$ is bounded provided that 
$b^2-4ac>0, \|u_0\|_{L^{3,\infty}}<-\frac{b}{2a}$ by using the continuity method, 
so we need 
$1-4C^2\varepsilon(\|u_0\|_{L^{3,\infty}}+\|\theta_0\|_{L^{1}})>0, \|u_0\|_{L^{3,\infty}}<\frac{1}{2C\varepsilon}$, 
it is equal to 
$\|u_0\|_{L^{3,\infty}}+\|\theta_0\|_{L^{1}}<\frac{1}{4C^2\varepsilon}, \|u_0\|_{L^{3,\infty}}<\frac{1}{2C\varepsilon}$. 
So we can choose $c^*<\min\{\frac{1}{4C^2}, \frac{1}{2C}\}$, then we can get
\begin{align*}
  \|u\|_{L_T^\infty (L^{3,\infty})}&\leq\frac{-b-\sqrt{b^2-4ac}}{2a} \\
  & = \frac{1-\sqrt{1-4C^2\varepsilon(\|u_0\|_{L^{3,\infty}}+\|\theta_0\|_{L^{1}})}}{2C\varepsilon}\\
  & \leq \frac{1}{2C\varepsilon}\frac{4C^2\varepsilon(\|u_0\|_{L^{3,\infty}}
  + \|\theta_0\|_{L^{1}})}{1+\sqrt{1-4C^2\varepsilon(\|u_0\|_{L^{3,\infty}}+\|\theta_0\|_{L^{1}})}} \\
  & \leq 2C(\|u_0\|_{L^{3,\infty}}+\|\theta_0\|_{L^{1}}).
\end{align*}
\end{comment}

Taking the inner product of $\Lambda u$ with both sides of the velocity equation $\eqref{eq:BousEq}_2$, 
we get
\begin{align*}
  \frac{1}{2}\frac{\dd}{\dd t} \|u\|_{\dot H^{\frac12}}^2 + \mathrm{Pr} \| u\|_{\dot H^{\frac32}}^2
  & \leq \| u\|_{\dot H^{\frac{3}{2}}} \|u\cdot\nabla u\|_{\dot H^{-\frac{1}{2}}} +
  \mathrm{Pr} \| u\|_{\dot H^{\frac{3}{2}}}\|\theta\|_{\dot H^{-\frac{1}{2}}}  \\
  &  \leq \| u\|_{\dot H^{\frac{3}{2}}} \|u\otimes u\|_{\dot H^{\frac{1}{2}}} +
  C \mathrm{Pr} \| u\|_{\dot H^{\frac{3}{2}}} \|\theta\|_{L^{\frac{3}{2}}} .
\end{align*}
In light of the para-differential calculus and using the fact  
$L^{3,\infty}(\R^3)\hookrightarrow \dot B^{-1}_{\infty,\infty}(\R^3)$, we deduce that
\begin{align*}
  \|u\otimes u\|_{\dot H^{\frac{1}{2}}} & \leq C \|T_u u\|_{\dot H^{\frac{1}{2}}} + C \|R(u, u)\|_{\dot H^{\frac{1}{2}}} \\
  & \leq C  \| u\|_{\dot B^{-1}_{\infty,\infty}} \| u\|_{\dot H^{\frac{3}{2}}} 
  \leq C \|u\|_{L^{3,\infty}} \|u\|_{\dot H^{\frac{3}{2}}}.
\end{align*}
Consequently, taking advantage of \eqref{eq:vL3inf-es} gives
\begin{align}\label{es.v.Hfr12}
  \frac{1}{2}\frac{\dd}{\dd t} \|u\|_{\dot H^{\frac12}}^2 +  \mathrm{Pr} \|u\|_{\dot H^{\frac32}}^2
  & \leq C\| u\|_{\dot H^{\frac32}}^2 \| u\|_{L^{3,\infty}} 
  + C \mathrm{Pr} \| u\|_{\dot H^{\frac32}}\|\theta_0\|_{L^{\frac32}}\nonumber\\
  & \leq \Big( 2 C C_0 c_* + \frac{1}{4} \Big) \mathrm{Pr} \| u\|_{\dot H^{\frac32}}^2
  + C \mathrm{Pr} \|\theta_0\|^2_{L^{\frac{3}{2}}}.
\end{align}
By choosing  $ c_* \leq \frac{1}{8C C_0}$ in \eqref{eq:c*}, one can obtain
\begin{align}\label{es.v.Hfr12b}
  \|u\|_{L^\infty_T(\dot H^{\frac12})}^2 + \mathrm{Pr} \| u\|_{L^2_T(\dot H^{\frac32})}^2 
  \leq \| u_0\|^2_{\dot H^{\frac12}} + C T \mathrm{Pr} \|\theta_0\|^2_{L^{\frac{3}{2}}}.
\end{align}

Furthermore, according to estimate \eqref{es.v.Hfr12b} and \cite[Lemma 3.1]{LLX24}, we infer that
\begin{align*}
 \|u\|_{L^2_T(L^\infty)}
  \leq C \Big(1+ \| u_0\|^2_{\dot H^{\frac12}} + C T \mathrm{Pr} (\|\theta_0\|^2_{L^1\cap L^s} + \|\theta_0\|_{L^1\cap L^s}) \Big)^\frac{3}{2},
\end{align*}
and for every $q>3$,
\begin{align}\label{eq:u-es1}
  \|u\|_{L^\infty_T(\dot B^{-1+\frac{3}{q}}_{q,\infty})} + \|u\|_{\widetilde{L}^1_T(\dot B^{1+\frac{3}{q}}_{q,\infty})}
  \leq C \Big(1+ \| u_0\|^2_{\dot H^{\frac12}} + C T \mathrm{Pr} (\|\theta_0\|^2_{L^1\cap L^s} + \|\theta_0\|_{L^1\cap L^s}) \Big)^2 .
\end{align}
As a result, we can establish the \textit{a priori} estimates as in \eqref{eq:v-the-es}. 

Since $\theta\in L^\infty_T(L^1\cap L^s)\hookrightarrow L^\infty_T(\dot B^{-1+\frac{3}{q}}_{q,\infty})$ with $q>3$
and $u$ satisfies \eqref{eq:u-es1}, we can exactly adapt the argument in \cite[Lemma 3.2]{LLX24} or \cite{DanP08}
to conclude the uniqueness part of Theorem \ref{thm:exi-reg}.

Hence, by using a standard approximation process (e.g. see \cite{LLX24}), 
we complete the proof of the global existence and uniqueness of strong solutions for the 3D Boussinesq system~\eqref{eq:BousEq} 
with large Prandtl number.

\begin{comment}
The \textit{a priori} estimates \eqref{es0.thet}, \eqref{es.v.L2} and \eqref{es.v.Hfr12b}
are enough to show the existence (e.g. by using a standard approximation process).
First solve the Cauchy problem \eqref{eq:BousEq} with frequency localized initial data
$(\theta_{0,n},v_{0,n}) := (S_n \theta_0, S_n v_0)$, $n\in \N$,
where $S_n$ is defined by \eqref{eq:Del-Sj}. It is clear that $(\theta_{0,n},v_{0,n})$ belongs to $H^s(\R^3)$ for all $s$. Then, it follows from  \cite{DanP08b}
that we obtain a local (unique) smooth solution $(\theta_n, v_n, \nabla p_n)$ to the system \eqref{eq:BousEq} associated with
$(\theta_{0,n}, v_{0,n})$. Moreover, the \textit{a priori} estimates below ensure that $(\theta_n,v_n)$ satisfies
\eqref{eq:v-the-es} on $[0,T]$ uniformly in $n$. Hence by using Rellich compactness theorem for instance (see \cite{PGLR1}), one can pass to the limit $n\rightarrow \infty$ (up to a subsequence) to show that there exist functions $(\theta,v,\nabla p)$ satisfying \eqref{eq:v-the-es}
which are solutions to the 3D Boussinesq system \eqref{eq:BousEq} in the sense of distribution. \\
\end{comment}

\section{Boundary regularity persistence of temperature patches for the 3D Boussinesq system}\label{sec:C1gam}
In this section, we prove the global persistence of $C^{1,\gamma}$, $W^{2,\infty}$ and $C^{2,\gamma}$
boundary regularity of the temperature patches \eqref{eq:the-patch-exp} along the evolution. 
The feature is that these results are uniform in the Prandtl number $\mathrm{Pr}$.
\vskip1mm

Let $\Omega=(\Omega^1,\Omega^2,\Omega^3) = \nabla\wedge u$ be the vorticity of the fluid, 
where the notation $\wedge $ stands for the wedge operation, that is,
$\Omega=\nabla\wedge u=\big(\partial_2u^3-\partial_3u^2, \partial_3u^1-\partial_1u^3, \partial_1u^2-\partial_2u^1\big)^t$.
Applying the operator $\nabla\wedge $ to the velocity equation $\eqref{eq:BousEq}_2$ gives the vorticity equation:
\begin{equation}\label{eq.Omeg}
  \frac{1}{\mathrm{Pr}} \partial_t \Omega + \frac{1}{\mathrm{Pr}} u\cdot \nabla \Omega -\Delta \Omega 
  = \frac{1}{\mathrm{Pr}} \Omega\cdot \nabla u + (\partial_2 \theta, -\partial_1\theta, 0)^t.
\end{equation}
Note that
\begin{equation*}
  \frac{1}{\mathrm{Pr}} \partial_t \Omega + \frac{1}{\mathrm{Pr}} u\cdot \nabla \Omega -\Delta 
  \big(\Omega - \Lambda^{-2} (\partial_2 \theta, -\partial_1\theta,0)^t \big) 
  = \frac{1}{\mathrm{Pr}} \Omega\cdot\nabla u,
\end{equation*}
where $\Lambda := (-\Delta)^{1/2}$.
By setting
\begin{align}\label{eq:op-R-1}
  \mathcal{R}_{-1,j} := \partial_j \Lambda^{-2},\;\;j=1,2,\quad  \textrm{and}\quad
  \mathcal{R}_{-1} := (\mathcal{R}_{-1,2}, - \mathcal{R}_{-1,1},0)^t,
\end{align}
then $ \mathcal{R}_{-1} \theta = \Lambda^{-2} (\partial_2 \theta, -\partial_1\theta,0)^t$ 
and this vector-valued quantity satisfies
\begin{align*}
  \partial_t \mathcal{R}_{-1}\theta + u\cdot \nabla \mathcal{R}_{-1}\theta = - [\mathcal{R}_{-1},u\cdot\nabla]\theta,
\end{align*}
with
\begin{align}\label{def:R-1comm}
  [\mathcal{R}_{-1}, u\cdot\nabla]\theta : = \big([\mathcal{R}_{-1,2},u\cdot\nabla]\theta,  -[\mathcal{R}_{-1,1},u\cdot\nabla]\theta,0\big)^t .
\end{align}
By introducing a new unknown $\Gamma = (\Gamma^1,\Gamma^2,\Gamma^3)$ defined as
\begin{equation}\label{Gamma}
  \Gamma:= \Omega - \mathcal{R}_{-1}\theta, \end{equation}
we infer that $\Gamma$ satisfies
\begin{equation}\label{eq.Gamm}
  \frac{1}{\mathrm{Pr}}\partial_t \Gamma + \frac{1}{\mathrm{Pr}} u\cdot\nabla \Gamma -\Delta\Gamma 
  = \frac{1}{\mathrm{Pr}}\Omega\cdot \nabla u + \frac{1}{\mathrm{Pr}} [\mathcal{R}_{-1},u\cdot\nabla]\theta. 
\end{equation}
The good unknown $\Gamma$ plays a significant role in the proof of the persistence of boundary regularity of temperature patches
in our paper.

\subsection{Persistence of the \texorpdfstring{$C^{1,\gamma}$}{C(1,gamma)}-boundary regularity}
In light of \eqref{X-C1gam-es},
one can get that $\nabla u\in L_T^1(C^\gamma(\R^3))$ gives rise to $\nabla X_t^{\pm 1}(x)\in L^\infty(0,T; C^{\gamma}(\R^3))$,
which clearly yields
\begin{equation*}
  \partial D_0\in C^{1,\gamma}\quad \Longrightarrow \quad \partial D(t)= X_t(\partial D_0)\in L^\infty(0,T;C^{1,\gamma}).
\end{equation*}
So we only need to prove the regularity of the velocity $u$.
More precisely, we are going to prove the following propositions.
\begin{proposition}\label{prop:ap-es1}
Suppose that $(u,\theta)$ is a smooth solution of 3D Boussinesq system \eqref{eq:BousEq} with
\begin{itemize}
\item $u_0\in H^1(\DD)$,
\item $\nabla\cdot u_0=0$,
\item $\theta_0\in L^1\cap L^6(\DD)$.
\end{itemize}
Let $\PPr \in [\Pr_*,\infty)$, with $\Pr_*>0$ given by \eqref{eq:data-cond}.
Then, there exists a constant $C>0$ depending only on the norms of $(u_0,\theta_0)$ 
but independent of $\mathrm{Pr}$ such that for every $T>0$ and $r\in[1,\frac43)$, we have
%for $\alpha =1$,
%\begin{align}\label{eq:nab-u-es-alp1}
%  \varepsilon \|\nabla u\|_{L^\infty_T(L^2)}^2 + \| u\|_{L^2_T(\dot H^2)}^2
%  \leq \varepsilon \|\nabla u_0\|_{L^2}^2 + \|\theta_0\|_{L^2}^2 T,
%\end{align}
%\begin{align}\label{eq:uLip-es}
%  \|\nabla u\|_{L^\infty_T(L^p)} + \|\nabla u\|_{L^1_T(B^{\frac{2}{p}(2\alpha-1)}_{\infty,1})} + \|\Gamma\|_{L^\infty_T(L^p)} +
%  \left\|\Gamma\right\|_{L_T^1 (B_{p, 1}^{\frac{4\alpha}{p}})} \leq C e^{C T} ,
%\end{align}
%for $\alpha \in(\frac{1}{2},1)$,
\begin{align}\label{eq:u-Gam-es2}
  \|(\nabla u,\Gamma)\|_{L^\infty_T(L^2)} + \|\Gamma\|_{L^1_T(B^{\frac{3}{2}}_{2,1})}
  + \left\|\Gamma\right\|_{\widetilde{L}_T^r (B_{2, \infty}^{\frac{2}{r}})}
  + \|\nabla u\|_{\widetilde{L}^1_T(C^{\frac{1}{2}})} + \|\nabla u\|_{L^r_T(L^\infty)}
  \leq C e^{C T} .
\end{align}

%\item If $\mathbf{D} = \mathbb{R}^2$, $\alpha =1$, we have that for all $\gamma \in (0,1)$,
%\begin{align}\label{eq:u-Gam-es2-2}
 % \|(\nabla u,\Gamma)\|_{L^\infty_T(L^2)} + \|\Gamma\|_{L^1_T(B^1_{2,1})}
 % + \left\|\Gamma\right\|_{\widetilde{L}_T^1 (B_{2, \infty}^2)}
 % + \|\nabla u\|_{L^1_T(C^\gamma)}
 % \leq C e^{C T + C E_\varepsilon(T)} ,
%\end{align}
%where $E_\varepsilon(t)$ defined by \eqref{eq:u-L2-es} is the energy bound.
\end{proposition}

\begin{proof}[Proof of Proposition \ref{prop:ap-es1}]
Taking the $L^2$ inner product of the $\Gamma$-equation \eqref{eq.Gamm} 
with $\Gamma$, and using the fact that $\mathrm{div}\,u=0$, we obtain
\begin{align*}
  \frac{1}{2} \frac{\dd }{\dd t} \|\Gamma(t)\|_{L^2}^2 + \mathrm{Pr} \|\nabla \Gamma(t)\|_{L^2}^2 
  &=\int_{\R^3} \Gamma\otimes\Omega:\nabla u \dd x + \int_{\R^3}[\mathcal{R}_{-1}, u\cdot \nabla]\theta\cdot\Gamma \dd x\\
  & \leq  C\|\Omega\|_{L^2}\|\nabla u\|_{L^{3}}\|\Gamma\|_{L^{6}} +  C\|[\mathcal{R}_{-1}, u\cdot \nabla]\theta\|_{L^{\frac65}} \|\Gamma\|_{L^6}.
\end{align*}
Noting that $[\mathcal{R}_{-1}, u\cdot \nabla]\theta=\mathcal{R}_{-1}\text{div}\,( u\theta)-u\cdot \nabla\mathcal{R}_{-1}\theta$ (using  $\mathrm{div}\,u=0$ again), and applying the Calderón–Zygmund theorem and the Sobolev embedding $\dot{H}^1(\R^3)\hookrightarrow L^6(\R^3)$, we deduce that 
\begin{align*}
  \|[\mathcal{R}_{-1}, u\cdot \nabla]\theta(t)\|_{L^{\frac{6}{5}}}
  & \leq C \| \mathcal{R}_{-1} \mathrm{div}\,(u\,\theta) (t)\|_{L^{\frac{6}{5}}} 
  + C \|u\cdot \nabla \mathcal{R}_{-1}\theta (t)\|_{L^{\frac{6}{5}}} \\
  & \leq C \|u(t)\|_{L^6}\|\theta(t)\|_{L^\frac{3}{2}} \\
  & \leq C \|\nabla u(t)\|_{L^2}\|\theta_0\|_{L^{\frac{3}{2}}} 
  \leq C \|\Omega(t)\|_{L^2}\|\theta_0\|_{L^1\cap L^6}.
\end{align*}
Taking advantage of the Sobolev embedding $\dot{H}^{\frac{3}{2}}(\mathbb{R}^3) 
\hookrightarrow \dot{W}^{1,3}(\mathbb{R}^3)$ 
and the relation $\Omega = \Gamma + \mathcal{R}_{-1}\theta$, we arrive at
\begin{align}\label{eq:Gam-Ene-es1-2}
  \frac{1}{2} \frac{\dd }{\dd t} &\|\Gamma(t)\|_{L^2}^2 + \mathrm{Pr} \|\nabla \Gamma(t)\|_{L^2}^2 \nonumber \\
  & \leq C\|\Omega(t)\|_{L^2}\|\theta_0\|_{L^1\cap L^6}\|\nabla\Gamma(t)\|_{L^2} 
  + C\|\Omega(t)\|_{L^2}\| u(t)\|_{\dot{H}^{\frac32}}\|\nabla \Gamma(t)\|_{L^{2}} \nonumber \\
  &\leq C\frac{1}{\mathrm{Pr}}\|\Omega(t)\|_{L^2}^2 
  \big(\|u(t)\|_{\dot{H}^{\frac32}}^2 
  +\|\theta_0\|_{L^1\cap L^6}^2\big)+\tfrac{1}{2}\mathrm{Pr}\|\nabla \Gamma(t)\|_{L^{2}}^2 \nonumber \\
  & \leq C\frac{1}{\mathrm{Pr}}\big(\|\Gamma(t)\|_{L^2}^2 +  \|\mathcal{R}_{-1}\theta(t)\|_{L^2}^2\big) \big(\|u(t)\|_{\dot{H}^{\frac{3}{2}}}^2 
  + \|\theta_0\|_{L^1\cap L^6}^2\big) 
  + \frac{1}{2}\mathrm{Pr}\|\nabla \Gamma(t)\|_{L^{2}}^2 \nonumber \\
  &\leq C \frac{1}{\mathrm{Pr}} \big(\|\Gamma(t)\|_{L^2}^2 
  +\|\theta_0\|_{L^1\cap L^6}^2\big) \big(\|u(t)\|_{\dot{H}^{\frac32}}^2 
  + \|\theta_0\|_{L^1\cap L^6}^2\big) + \frac{1}{2}\mathrm{Pr}\|\nabla \Gamma(t)\|_{L^2}^2,
\end{align}
where in the last line we have used the Hardy-Littlewood-Sobolev inequality
\begin{align}\label{eq:Rbet-L2es}
  \|\mathcal{R}_{-1} \theta(t)\|_{L^2} \leq C \|\theta(t)\|_{L^{\frac{6}{5}}} \leq C \|\theta_0\|_{L^1\cap L^6}.
\end{align}
Applying Grönwall’s inequality together with \eqref{es.v.Hfr12b}, we obtain
\begin{equation}\label{es:GamL2-2}
\begin{split}
  \|\Gamma\|_{L^\infty_T(L^2)}^2 + \mathrm{Pr} \|\nabla \Gamma\|_{L^2_T(L^2)}^2
  & \leq C \big( \|\Gamma_0\|_{L^2}^2+ \|\theta_0\|_{L^1\cap L^6}^2 \big)
  \exp\big\{C \tfrac{1}{\mathrm{Pr}}(\|u\|_{L_T^2\dot{H}^{\frac32}}^2+\|\theta_0\|_{L^1\cap L^6}^2T)\big\} \\
  & \leq C e^{CT},
\end{split}
\end{equation}
where we have used the fact that $\Gamma_0:= \Omega_0-\mathcal{R}_{-1} \theta_0$ and 
\begin{align}\label{es.Gamma_0}
  \|\Gamma_0\|_{L^2}\leq \|\Omega_0\|_{L^2} + \|\mathcal{R}_{-1} \theta_0\|_{L^2}
  \leq \|u_0\|_{\dot H^1} + \|\theta_0\|_{L^{\frac65}} < \infty.
\end{align}
Using \eqref{eq:Rbet-L2es} and \eqref{es.Gamma_0} again, we also infer that
\begin{align}\label{es:nab-uL2-2}
  \|\nabla u\|_{L^\infty_T(L^2)} \leq \|\Omega\|_{L^\infty_T(L^2)}
  \leq \|\Gamma\|_{L^\infty_T(L^2)} + \|\mathcal{R}_{-1}\theta\|_{L^\infty_T(L^2)}
  \leq C e^{CT}.
\end{align}

Next, we consider the estimation of $\|\Gamma\|_{L^1_T(B^{\frac{3}{2}}_{2,1})}$ 
and $\|\Gamma\|_{\widetilde{L}^r_T(B^{\frac{2}{r}}_{2,\infty})}$.
For every $q\in \mathbb{N}$, applying the frequency localization operator
$\Delta_q$ to the equation \eqref{eq.Gamm} yields
\begin{equation}\label{eq:Delta-q-Gamm}
\begin{aligned}
  \partial_t \Delta_q\Gamma + u \cdot \nabla \Delta_q\Gamma - \mathrm{Pr} \,\Delta \Delta_q\Gamma
  = - \left[\Delta_q, u \cdot \nabla\right] \Gamma +
  \Delta_q\big(\left[\mathcal{R}_{-1}, u \cdot \nabla\right] \theta+ \Omega\cdot\nabla u\big)  := f_q .
\end{aligned}
\end{equation}
Taking the scalar product of the above equation with $\Delta_q \Gamma$, we get
\begin{align*}
  \frac{1}{2}\frac{\dd}{\dd t} \|\Delta_q\Gamma(t)\|_{L^2}^2 
  + c_0 \mathrm{Pr} \,2^{2 q} \| \Delta_q \Gamma(t)\|_{L^2}^2
  \leq \|\Delta_q \Gamma\|_{L^2} \|f_q\|_{L^2},
\end{align*}
with $c_0>0$ some absolute constant. 
Integrating in the time variable leads to
\begin{align*}
  \|\Delta_q \Gamma(t)\|_{L^2} \leq e^{- c_0 \mathrm{Pr}\, 2^{2q} t} 
  \|\Delta_q \Gamma_0\|_{L^2}
  + \int_0^t e^{- c_0 \mathrm{Pr} (t-\tau) 2^{2 q}} \|f_q(\tau)\|_{L^2} \dd \tau,
\end{align*}
and
\begin{align}\label{eq:Del-Gam-L2-2}
  \|\Delta_q\Gamma\|_{L^r_t(L^2)} \leq (\tfrac{1}{c_0\mathrm{Pr}})^{\frac{1}{r}} 2^{-\frac{2 q}{r}} \|\Delta_q \Gamma_0\|_{L^2}
  +(\tfrac{1}{c_0\mathrm{Pr}})^{\frac{1}{r}} 2^{-\frac{2q}{r}} \|f_q\|_{L^1_t(L^2)}.
\end{align}
Noticing that
\begin{align*}%\label{eq:the-bdd1}
  \|[\mathcal{R}_{-1},u\cdot\nabla]\theta\|_{L^1_T(L^2)} 
  & \leq C \|\mathcal{R}_{-1} \mathrm{div}\, (u\,\theta)\|_{L^1_T(L^2)}
  + C \|u\cdot \nabla \mathcal{R}_{-1}\theta\|_{L^1_T(L^2)} \\
  & \leq C\|u\|_{L_T^\infty (L^6)}\|\theta\|_{L_T^1(L^3)}
  \leq C\|\nabla u\|_{L_T^\infty (L^2)}\|\theta\|_{L_T^1(L^3)},
\end{align*}
and
\begin{align*}
  \|\mathcal{R}_{-1}\theta\cdot\nabla u\|_{L_T^1 (L^2)} 
  & \leq C\|\mathcal{R}_{-1}\theta\|_{L_T^1(L^\infty)}\|\nabla u\|_{L_T^\infty (L^2)} \\
%  & \leq C\|\mathcal{R}_{-1}\theta\|_{L_T^1(B^0_{\infty,1})}\|\nabla u\|_{L_T^\infty (L^2)} \\
  &\leq C \Big(\|\Delta_{-1}\mathcal{R}_{-1}\theta\|_{L_T^1(L^2)} 
  + \sum_{j\in \mathbb{N}} \|\Delta_j \mathcal{R}_{-1}\theta\|_{L_T^1(L^\infty)} \Big)\|\nabla u\|_{L_T^\infty (L^2)}\\
  &\leq C \Big(\|\Delta_{-1} \theta\|_{L_T^1(L^{\frac{6}{5}})} 
  + \sum_{j\in \mathbb{N}} 2^{-\frac{1}{4}j} 
  \|\Delta_j \theta\|_{L_T^1(L^4)} \Big)\|\nabla u\|_{L_T^\infty (L^2)}\\
  &\leq C\|\theta\|_{L_T^1(L^1\cap L^6)}\|\nabla u\|_{L_T^\infty (L^2)},
\end{align*}
together with  Lemma \ref{lem:HKR}, we find 
\begin{equation}\label{es:fq-L1L2}
\begin{split}
  \|f_q\|_{L^1_T(L^2)} & \leq \|[\Delta_q,u\cdot \nabla]\Gamma\|_{L^1_T(L^2)}
  + C\|[\mathcal{R}_{-1},u\cdot\nabla]\theta\|_{L^1_T(L^2)} 
  + C \|(\mathcal{R}_{-1}\theta + \Gamma)\cdot \nabla u\|_{L^1_T(L^2)} \\
  & \leq C \big(\|\theta\|_{L^1_T(L^1\cap L^6)}  + \|\Gamma\|_{L^1_T(L^\infty)} 
  \big)\|\nabla u\|_{L^\infty_T(L^2)} \\
  & \leq C \big(\|\theta_0\|_{L^1\cap L^6} T  + \|\Gamma\|_{L^1_T(L^\infty)} 
  \big)\|\nabla u\|_{L^\infty_T(L^2)} .
\end{split}
\end{equation}
Let $N\in \mathbb{N}$ be a constant chosen later. 
Inserting \eqref{es:fq-L1L2} into \eqref{eq:Del-Gam-L2-2}, and using \eqref{es:GamL2-2}, \eqref{es:nab-uL2-2}, we have
\begin{align}\label{eq:Gam-B121-2}
  \|\Gamma\|_{L^1_T(B^{3/2}_{2,1})} &= \displaystyle\sum_{-1\leq q < N} 2^{\frac32q} \|\Delta_q \Gamma\|_{L^1_T(L^2)}
  + \sum_{q \geq N} 2^{\frac32q} \|\Delta_q \Gamma\|_{L^1_T(L^2)} \notag \\
  &\leq C 2^{\frac32N} \|\Gamma\|_{L^1_T(L^2)} + C \sum_{q\geq N} 2^{-\frac12q} \big( \|\Gamma_0\|_{L^2} 
  +\| f_q\|_{L^1_T(L^2)} \big)  \nonumber \\
  &\leq C 2^{\frac32N}  \|\Gamma\|_{L^1_T(L^2)}+ \,C  2^{-\frac12 N} \Big( \|\Gamma_0\|_{L^2}  + \|\nabla u\|_{L^\infty_T(L^2)}
  \big(\|\Gamma\|_{L^1_T(L^\infty)}  + T \|\theta\|_{L^1\cap L^6}\big) \Big) \nonumber \\
  &\leq C e^{CT} 2^{\frac32N} + C 2^{-\frac12N} e^{CT} \|\Gamma\|_{L^1_T(B^{3/2}_{2,1})},
\end{align}
where in the last line we have used the continuous embedding $B^{3/2}_{2,1}(\DD)\hookrightarrow L^\infty(\DD)$.
By choosing $N\in \mathbb{N}$ so that $C 2^{-\frac12N} e^{CT} \approx \frac{1}{2}$, one can infer that
\begin{align*}
  \|\Gamma\|_{L^1_T(B^{\frac32}_{2,1})} \leq C e^{CT}.
\end{align*}
Repeating the above process, we find that for every $r\in [1,\frac{4}{3})$,
\begin{align}\label{eq:Gam-B2a2inf-2}
  \|\Gamma \|_{\widetilde{L}^r_T(B^{\frac{2}{r}}_{2,\infty})}
  & =  \sup_{q\geq -1}  2^{\frac{2}{r} q} \|\Delta_q \Gamma\|_{L^r_T(L^2)} \nonumber \\
  &\leq C \|\Delta_{-1}\Gamma\|_{L^r_T(L^2)} + C\,\sup_{q\in \mathbb{N}} \Big(\|\Delta_q \Gamma_0\|_{L^2}
  + \|f_q\|_{L^1_T(L^2)}  \Big) \nonumber \\
  &\leq  C  T^{\frac{1}{r}}\|\Gamma\|_{L^\infty_T(L^2)}
   +  C  \|\Gamma_0\|_{L^2}
  + C \|\nabla u\|_{L^\infty_T(L^2)}
  \Big( \|\Gamma\|_{L^1_T(L^\infty)} + T\|\theta_0\|_{L^1\cap L^6}\Big) \nonumber \\
  &\leq Ce^{CT}.
\end{align}

Next, we want to prove that $\nabla u\in \widetilde{L}^1_T(C^{\frac{1}{2}})
= \widetilde{L}^1_T(B^{\frac{1}{2}}_{\infty,\infty})$ 
and $\nabla u \in L^r_T(L^\infty)$. 
Taking advantage of the identity \eqref{eq.v.GamThe0}, 
Bernstein's inequality, and estimates \eqref{es:nab-uL2-2}, \eqref{eq:Gam-B2a2inf-2}, we have
\begin{equation}\label{es:nab-uHold1-2}
\begin{split}
  \|\nabla u\|_{\widetilde{L}^1_T(C^{\frac{1}{2}})} & \leq \|\Delta_{-1} \nabla u\|_{L^1_T (L^\infty)}
  + \sup_{q\in \mathbb{N}} 2^{\frac{1}{2}q}\|\Delta_q \nabla u\|_{L^1_T(L^\infty)}  \\
  & \leq C \|\nabla u\|_{L^1_T(L^2)} + C \sup_{q\in \mathbb{N}} 2^{\frac{1}{2}q} \Big(\|\Delta_q \Gamma\|_{L^1_T(L^\infty)}
  + 2^{-q}\|\Delta_q \theta\|_{L^1_T(L^\infty)}\Big)  \\
  & \leq C \|\nabla u\|_{L^1_T(L^2)} + C \|\Gamma\|_{\widetilde{L}^1_T(B^{2}_{2,\infty})}
  + C \sup_{q\in \mathbb{N}} \|\Delta_q \theta\|_{L^1_T(L^6)} \leq C e^{CT}.
\end{split}
\end{equation}
Similarly, using the continuous embedding $B^{\frac{2}{r}}_{2,\infty}(\mathbb{R}^3)\hookrightarrow B^0_{\infty,1}(\mathbb{R}^3)$ for every $r\in[1,\frac{4}{3})$, we infer that
\begin{align*}
  \|\nabla u\|_{L_T^r (L^\infty)} %& \leq C\|\nabla u\|_{L_T^r(B^0_{\infty,1})} \\
  &\leq C\|\Delta_{-1}\nabla u\|_{L_T^r(L^2)} + \sum\limits_{q\in\mathbb{N}}\|\Delta_q\nabla u\|_{L_T^r(L^\infty)} \\
  &\leq CT^{\frac{1}{r}}\|\nabla u\|_{L_T^\infty (L^2)} + C \sum\limits_{q\in\mathbb{N}}\|\Delta_q\Gamma\|_{L_T^r(L^\infty)} 
  + C \sum\limits_{q\in\mathbb{N}}2^{-q}\|\Delta_q\theta\|_{L_T^r(L^\infty)} \\
  &\leq C T^{\frac{1}{r}}\|\nabla u\|_{L_T^\infty (L^2)} + C \|\Gamma\|_{\widetilde{L}_T^r(B^0_{\infty,1})} 
  + C \sum_{q\in \mathbb{N}} 2^{-\frac{1}{2}q} \|\Delta_q \theta\|_{L^r_T(L^6)} \\
  &\leq CT^{\frac{1}{r}}\|\nabla u\|_{L_T^\infty (L^2)} + C\|\Gamma\|_{\widetilde{L}_T^r(B^{\frac{2}{r}}_{2,\infty})}
  + C T^{\frac{1}{r}} \|\theta\|_{L_T^\infty(L^6)} \leq Ce^{CT}.
\end{align*}
Therefore, gathering the above estimates completes the proof of \eqref{eq:u-Gam-es2}.
\end{proof}

\begin{proposition}\label{prop:ap-es2}
Suppose that $(u,\theta)$ is the smooth solution for the 3D Boussinesq system \eqref{eq:BousEq} satisfying
\begin{itemize}
\item $u_0\in H^1\cap\dot{W}^{1,p}(\DD)$, $2<p<\infty$,
\item $\nabla\cdot u_0=0$,
\item $\theta_0\in L^1\cap L^{\infty}(\DD)$.
\end{itemize}
Let $\PPr \in [\Pr_*,\infty)$, with $\mathrm{Pr}_*>0$ given by \eqref{eq:data-cond}.
Then, for every $T>0$, there exists a constant $C>0$ 
depending on the norms of $(u_0,\theta_0)$ but independent of $\PPr$ 
such that the following statements hold.
\begin{enumerate}[(1)]
\item For every $1\leq r<\infty$, we have
\begin{align}\label{es:nabu-Gam2}
  \|(\nabla u,\Gamma)\|_{L^\infty_T(L^p)} 
  + \|\Gamma\|_{\widetilde{L}^r_T(B^{\frac{2}{r}}_{p,\infty})} \leq C e^{\exp\{CT\}}.
\end{align}
\item 
If $\frac{3}{2-\gamma}<p<\infty$, $\gamma\in[\frac12,1)$, 
then for every $r\in [1,\frac{2p}{3+\gamma p})$, $s\in[1,\frac{2p}{3})$, we have 
\begin{align}\label{es:nabu-Gam4}
  \|\Gamma\|_{L^r_T (C^{\gamma})}  +  \|\nabla u\|_{L^r_T(C^{\gamma})} 
  + \| u\|_{L^s_T(W^{1,\infty})}  \leq C e^{\exp\{CT\}}.
\end{align}
\item If $3<p<\infty$, then for every $s\in[1,\frac{2p}{p+3})$, we have
\begin{align}\label{es:nabu-Gam3}
  \|\Gamma\|_{L^s_T (B^{1}_{\infty,1})}
  +  \|\nabla u\|_{L^s_T(B^{1}_{\infty,\infty})} + \| u\|_{L^2_T(W^{1,\infty})} 
  \leq C e^{\exp\{CT\}}.
\end{align}
\end{enumerate}
\end{proposition}

\begin{comment}
\begin{remark}\label{rem:W2infty}
Under condition (2) of Theorem~\ref{thm:exi-reg}, we can prove that the $W^{2,\infty}$ 
regularity of the temperature patch is preserved throughout the evolution and is independent of~$\mathrm{Pr}$. 
Indeed, thanks to Lemma~\ref{lem:flow}, we only need to prove $u\in L_T^1(W^{2,\infty})$, 
one can get such result by combining the conclusion of Proposition~\ref{prop:ap-es2} 
and \cite[Section 4.2]{LLX24}. Moreover, one can get $u\in L_T^r(W^{2,\infty})$ for $r\in[1,\frac{2p}{p+3})$.
\end{remark}
\end{comment}

\begin{proof}[Proof of Proposition \ref{prop:ap-es2}]
%First, we localise in frequencies the equation for $\Gamma$ to get
Taking the inner product of both sides of the $\Gamma$-equation \eqref{eq.Gamm} 
with $|\Gamma|^{p-2}\Gamma(x,t)$, one can get
\begin{align}\label{eq:Gam-Lp-es1}
  \frac{1}{p} \frac{\dd}{\dd t} \|\Gamma(t)\|_{L^p}^p - \PPr \int_{\DD} \Delta \Gamma\cdot |\Gamma|^{p-2}
  \Gamma(x,t) \dd x \leq \Big(\|[\mathcal{R}_{-1}, u\cdot\nabla]\theta \|_{L^p} 
  + \|\Omega\cdot\nabla u\|_{L^p} \Big) \|\Gamma(t)\|_{L^p}^{p-1}.
\end{align}
Integration by parts ensures that the term in \eqref{eq:Gam-Lp-es1} 
coming from dissipation $-\Delta\Gamma$ is nonnegative, thus we have
\begin{align*}
  \frac{\dd }{\dd t}\|\Gamma(t)\|_{L^p}  \leq \|[\mathcal{R}_{-1},u\cdot\nabla]\theta(t)\|_{L^p} 
  + \|\Omega\cdot\nabla u(t)\|_{L^p}.
\end{align*}
Applying Lemma \ref{lem:Rbeta-cm}, the Calder\'on-Zygmund theorem and the continuous embeddings 
$L^\infty(\mathbb{R}^3)\hookrightarrow B^{-1/2}_{\infty,1}(\mathbb{R}^3)$ and $\dot{H}^1(\R^3)\hookrightarrow L^6(\R^3)$, we find
\begin{align}\label{eq:comm-Lp}
  \|[\mathcal{R}_{-1},u\cdot\nabla]\theta(t)\|_{L^p} 
  & \leq C \|[\mathcal{R}_{-1},u\cdot\nabla]\theta(t)\|_{B^{1/2}_{p,1}} \nonumber \\
  & \leq C \|\nabla u(t)\|_{L^p} \big(\|\theta(t)\|_{B^{-1/2}_{\infty,1}}+\|\theta(t)\|_{L^{2}} \big)
  + C\|u(t)\|_{L^{6}}\|\theta(t)\|_{L^{6/5}} \nonumber \\
  & \leq C \|\Omega(t)\|_{L^p} \|\theta_0\|_{L^2\cap L^\infty} 
  + C \|\nabla u(t)\|_{L^2}\|\theta_0\|_{L^{6/5}}.
\end{align}
Combining the above two inequalities yields
\begin{align*}
  \frac{\dd }{\dd t}\|\Gamma(t)\|_{L^p}
  & \leq C \|\Omega(t)\|_{L^p} \big(\|\theta_0\|_{L^2\cap L^\infty} + \|\nabla u(t)\|_{L^\infty} \big)
  + C \|\nabla u(t)\|_{L^2}\|\theta_0\|_{L^{6/5}} \\
  & \leq C \big(\|\Gamma(t)\|_{L^p} + \|\mathcal{R}_{-1}\theta(t)\|_{L^p} \big) \big(\|\theta_0\|_{L^2\cap L^\infty} 
  + \|\nabla u(t)\|_{L^\infty}\big)  + C\|\nabla u(t)\|_{L^2}\|\theta_0\|_{L^{6/5}} \\
  & \leq C \big(\|\Gamma(t)\|_{L^p} + \|\theta_0\|_{L^1\cap L^\infty} \big) \big(\|\theta_0\|_{L^1\cap L^\infty} 
  + \|\nabla u(t)\|_{L^\infty}\big) 
  + C\|\nabla u(t)\|_{L^2}\|\theta_0\|_{L^1\cap L^\infty},
\end{align*}
where in the last line we have used the Hardy-Littlewood-Paley inequality
\begin{align}\label{eq:R-1-Lp}
  \|\mathcal{R}_{-1} \theta \|_{L^p(\mathbb{R}^3)} 
  \leq C\|\Lambda^{-1}\theta\|_{L^p(\mathbb{R}^3)} 
  \leq C \|\theta\|_{L^{\frac{3p}{p+3}}(\mathbb{R}^3)}.
\end{align}
Gr\"onwall's inequality, estimate \eqref{eq:u-Gam-es2} and the fact that 
\begin{align*}
  \|\Gamma_0\|_{L^p}\leq \|\Omega_0\|_{L^p} + \|\mathcal{R}_{-1}\theta_0\|_{L^p} 
  \leq C\|\nabla u_0\|_{L^p} + C \|\theta_0\|_{L^{\frac{3p}{p+3}}}\leq C
\end{align*}
imply that
\begin{align*}
  \|\Gamma\|_{L_T^\infty(L^p)} & \leq C \Big(\|\Gamma_0\|_{L^p}+\|\theta_0\|_{L^1\cap L^\infty}
  + \|\theta_0\|_{L^1\cap L^\infty}\|\nabla u\|_{L_T^1(L^2)}\Big) e^{\int_0^T \big(\|\nabla u\|_{L^\infty} 
  + \|\theta_0\|_{L^1\cap L^\infty}\big)\dd t} \\
  & \leq Ce^{\exp\{C T\}}.
\end{align*}
Consequently, using the identity $\Omega =\Gamma + \mathcal{R}_{-1}\theta$ 
and estimate \eqref{eq:R-1-Lp}, we obtain
\begin{align}\label{es.Gamma}
  \|(\Gamma,\nabla u)\|_{L^\infty_T(L^p)} \leq C \|(\Gamma,\Omega)\|_{L^\infty_T(L^p)}
  \leq C \|\Gamma\|_{L^\infty_T(L^p)} + \|\mathcal{R}_{-1}\theta\|_{L^\infty_T(L^p)} 
  \leq Ce^{\exp\{C T\}}.
\end{align}

For every $q\in \mathbb{N}$, by taking the $L^2$-inner product of equation \eqref{eq:Delta-q-Gamm} with 
$\left|\Delta_q\Gamma\right|^{p-2} \Delta_q\Gamma(x,t)$
and using the following positivity estimate (see \cite{CMZ07})
%integrating in the space variable we find
%$$
%\frac{1}{p} \frac{\mathrm{d}}{\mathrm{d} t}\left\|\Delta_q\Gamma(t)\right\|_{L^p}^p+\int_{\mathbb{R}^2}\left(\Lambda^{2 \alpha} \Delta_q\Gamma\right)\left|\Delta_q\Gamma\right|^{p-2} \Delta_q\Gamma \mathrm{~d} x \leq\left\|\Delta_q\Gamma(t)\right\|_{L^p}^{p-1}\left\|f_q(t)\right\|_{L^p}.
%$$
%From Lemma~\ref{lem:CMZ},
%we have the following inequality,
\begin{align*}
  \int_{\DD} \left(-\Delta \Delta_q\Gamma\right)\left|\Delta_q\Gamma\right|^{p-2} \Delta_q\Gamma(x)
  \mathrm{~d} x \geq c\, 2^{2 q} \left\|\Delta_q\Gamma\right\|_{L^p}^p,\quad \forall q\in \mathbb{N},
\end{align*}
with some universal constant $c>0$ independent of $q$, we have
\begin{align*}
  \frac{1}{p} \frac{\mathrm{d}}{\mathrm{d} t}\left\|\Delta_q\Gamma(t)\right\|_{L^p}^p +
  c\,\PPr\,  2^{2 q}\left\|\Delta_q\Gamma(t)\right\|_{L^p}^p
  \leq\left\|\Delta_q\Gamma(t)\right\|_{L^p}^{p-1}\left\|f_q(t)\right\|_{L^p},
\end{align*}
which immediately gives
\begin{align*}
  \left\|\Delta_q\Gamma(t)\right\|_{L^p} \leq e^{-c\,\PPr\, t 2^{2  q}} \left\|\Delta_q\Gamma_0\right\|_{L^p}
  + \int_0^t e^{- c \,\PPr\, (t-\tau) 2^{2q}}\left\|f_q(\tau)\right\|_{L^p} \mathrm{d} \tau,
\end{align*}
where 
$f_q = - \left[\Delta_q, u \cdot \nabla\right] \Gamma +
  \Delta_q\big(\left[\mathcal{R}_{-1}, u \cdot \nabla\right] \theta+ \Omega\cdot\nabla u\big)$.
Taking the $L^r([0,T])$ $(r\in[1,\infty))$-norm in the above inequality, and using \eqref{eq:u-Gam-es2}, \eqref{eq:comm-Lp}, \eqref{es.Gamma}, 
Lemma~\ref{lem:HKR}, we find that for all $q\in \mathbb{N}$,
\begin{align*}
  & 2^{q\frac{2}{r} }\left\|\Delta_q\Gamma\right\|_{L_T^r (L^p)}\\
  & \leq  \frac{C}{\PPr^{1/r}}  \left\|\Delta_q\Gamma_0\right\|_{L^p} 
  + \frac{C}{\PPr^{1/r}} \left( \left\|\left[\Delta_q, u \cdot \nabla\right]
  \Gamma \right\|_{L_T^1(L^p)}+\|\Delta_q(\Omega\cdot\nabla u)\|_{L_T^1(L^p)}
  + \left\|\Delta_q \left[\mathcal{R}_{-1}, u \cdot \nabla\right] \theta
  \right\|_{L_T^1(L^p)}\right) \\
%  & \leq C \left\|\Gamma_0\right\|_{L^p}+ C \int_0^T \| \nabla u(\tau)\|_{L^p} \|\Gamma(\tau)\|_{B^0_{\infty,\infty}}
%  \mathrm{d} \tau + C \int_0^T \left\|\left[\mathcal{R}_{2 \alpha}, u \cdot \nabla\right]
%  \theta(\tau)\right\|_{B^0_{p,\infty}} \mathrm{d} \tau \\
  &\leq C \left\|\Gamma_0\right\|_{L^p}  + C \|\nabla u\|_{L^\infty_T(L^p)} \left(\|\Gamma\|_{L^1_T(B^0_{\infty,\infty})}
  +\|\theta\|_{L^1_T(L^2\cap L^\infty)} + \|\nabla u\|_{L_t^1L^\infty} \right)\\
  & \quad + C\|\nabla u\|_{L_T^\infty (L^2)} \|\theta\|_{L_T^1 (L^{6/5})}
  \leq  Ce^{\exp\{C T\}},
\end{align*}
where in the last line we have used the embedding $B^{\frac{3}{2}}_{2,1}(\mathbb{R}^3)\hookrightarrow B^0_{\infty,\infty}(\mathbb{R}^3)$. 
%(owing to the Hardy-Littlewood-Sobolev inequality)
%\begin{align*}
%  \left\|\mathcal{R}_{1-2 \alpha} \theta_0\right\|_{L^p} \leq C \|\Lambda^{1-2\alpha} \theta_0\|_{L^p}
%  \leq C \| \theta_0\|_{L^{\frac{2p}{(2\alpha-1)p+2}}} \leq C \|\theta_0\|_{L^1\cap L^\infty}.
%\end{align*}
Hence, we have
\begin{align}\label{eq:Gam-LrBesov}
  \|\Gamma\|_{\widetilde{L}^r_T(B^{\frac{2}{r}}_{p,\infty})} \leq C \|\Delta_{-1}\Gamma\|_{L^r_T (L^p)}
  + \sup_{q\in \mathbb{N}} 2^{q\frac{2}{r}}  \|\Delta_q \Gamma\|_{L^r_T(L^p)}  \leq Ce^{\exp\{C T\}}.
\end{align}

Now we turn to the proof of \eqref{es:nabu-Gam4}. 
If $p>\frac{3}{2-\gamma}$, together with the continuous embedding 
$\widetilde{L}_T^{r} (B^{2/r}_{p,\infty})\hookrightarrow L_T^r (B^\gamma_{\infty,\infty})$ for $r\in [1,\frac{2p}{3+\gamma p})$, we see that
\begin{equation*}%\label{eq:nab-u-C2alp-2}
\begin{split}
  \|(\nabla u,\Gamma) \|_{L^r_T(C^\gamma)} 
  & \leq C \|\Delta_{-1}\nabla u\|_{L^r_T(L^\infty)} +
  C \Big\| \sup_{q\in \mathbb{N}}2^{q\gamma} \|\Delta_q \nabla u\|_{L^\infty} \Big\|_{L^r_T} + \|\Gamma\|_{L^r_T(C^\gamma)} \\
  & \leq C T^{1/r}\|\nabla u\|_{L^\infty_T(L^2)} + C \|\Gamma\|_{L^r_T( B^{\gamma}_{\infty,\infty})}
  + C \|\theta\|_{L^r_T(B^{\gamma-1}_{\infty,\infty})} \\
  & \leq C T^{1/r}\|\nabla u\|_{L^\infty_T(L^2)} 
  + C \|\Gamma\|_{\widetilde{L}_T^r(B^{2/r}_{p,\infty})}
  + C \|\theta\|_{L^r_T(L^\infty)} 
  \leq Ce^{\exp\{C T\}}.
\end{split}
\end{equation*}
From the high-low frequency decomposition and \eqref{eq.v.GamThe0}, 
we also infer that for every $s\in[1,\frac{2p}{3})$,
\begin{equation*}
\begin{aligned}
  \|u\|_{L_T^s(W^{1,\infty})} & \leq C\| u\|_{L_T^s(B^1_{\infty,1})}
  \leq C\|\Delta_{-1} u\|_{L_T^s(L^6)}
  + C \sum\limits_{q\in\mathbb{N}}\|\Delta_q \nabla u\|_{L_T^s(L^\infty)} \\
  &\leq C\|\nabla u\|_{L_T^s(L^2)} 
  + C \sum\limits_{q\in\mathbb{N}}\|\Delta_{q}\Gamma\|_{L_T^s(L^\infty)} 
  + C \sum\limits_{q\in\mathbb{N}}2^{-q}\|\Delta_{q}\theta\|_{L_T^s(L^\infty)} \\
  &\leq C T^{\frac{1}{s}}\|\nabla u\|_{L_T^\infty (L^2)} 
  + C  \|\theta\|_{L^s_T(L^\infty)} 
  + C\|\Gamma\|_{\widetilde{L}_T^s(B^0_{\infty,1})} \\
  &\leq C T^{\frac1s}\|\nabla u\|_{L_T^\infty (L^2)} 
  + C T^{\frac{1}{s}}\|\theta_0\|_{L^\infty} 
  + C \|\Gamma\|_{\widetilde{L}_T^s(B^{2/s}_{p,\infty})} 
  \leq Ce^{\exp\{C T\}}.
\end{aligned}
\end{equation*}

Next we prove \eqref{es:nabu-Gam3}. If $p>3$, using the continuous embedding 
$B^{\frac{2}{s}}_{p,\infty}(\mathbb{R}^3) \hookrightarrow 
B^{\frac{2}{s}-\frac{3}{p}}_{\infty,\infty}(\mathbb{R}^3) 
\hookrightarrow B^1_{\infty,1}(\mathbb{R}^3)$ 
for $s\in[1,\frac{2p}{p+3})$, the above inequality \eqref{eq:Gam-LrBesov} yields 
\begin{align*}
  \|\Gamma\|_{L^s_T(B^{1}_{\infty,1})} \leq 
  \|\Gamma\|_{\widetilde{L}^s_T(B^1_{\infty,1})} \leq Ce^{\exp\{C T\}} .
\end{align*}
%which together with the continuous embedding $B^{2\alpha}_{p,\infty} \hookrightarrow B^{2\alpha-\frac{2}{p}}_{\infty,\infty}
%\hookrightarrow B^1_{\infty,1}$ (for all $p>\frac{2}{2\alpha-1}$ and $\frac{1}{2}<\alpha<1$)
%leads to the wanted result of $\Gamma$.
As for the estimates of $ u$ in \eqref{es:nabu-Gam3}, 
we also use the splitting \eqref{eq.v.GamThe0} and estimate \eqref{eq:Gam-LrBesov} 
to see that for every $s\in[1,\frac{2p}{p+3})$,
\begin{equation*}%\label{eq:nab-u-C2alp-2}
\begin{split}
  \|\nabla u\|_{L^s_T(B^{1}_{\infty,\infty})} 
  & \leq C \|\Delta_{-1}\nabla u\|_{L^s_T(L^\infty)} +
  C \Big\| \sup_{q\in \mathbb{N}}2^{q} \|\Delta_q \nabla u\|_{L^\infty} \Big\|_{L^s_T} \\
  & \leq C T^{\frac{1}{s}}\|\nabla u\|_{L^\infty_T(L^2)} + C \|\Gamma\|_{L^s_T( B^{1}_{\infty,\infty})}
  + C \|\theta\|_{L^s_T(B^0_{\infty,\infty})}
  \leq Ce^{\exp\{C T\}},
\end{split}
\end{equation*}
and
\begin{equation*}
\begin{aligned}
  \| u\|_{L_T^2(W^{1,\infty})} \leq C \|u\|_{L_T^2(B^1_{\infty,1})} 
  & \leq C\|\Delta_{-1} u\|_{L_T^2(L^6)} 
  + C \sum\limits_{q\in\mathbb{N}}\|\Delta_{q}\nabla u\|_{L_T^2(L^\infty)}\\
  &\leq C\|\nabla u\|_{L_T^2(L^2)} 
  + C \sum\limits_{q\in\mathbb{N}}\|\Delta_{q}\Gamma\|_{L_T^2(L^\infty)} 
  + C \sum\limits_{q\in\mathbb{N}}2^{-q}\|\Delta_{q}\theta\|_{L_T^2(L^\infty)}\\
%  &\leq C\|\nabla u\|_{L_T^2(L^2)} + CT^{\frac{1}{2}}\|\theta_0\|_{L^\infty}
%  + C \|\Gamma\|_{\widetilde{L}_T^2(B^0_{\infty,1})}\\
  &\leq C\|\nabla u\|_{L_T^2(L^2)} + C T^{\frac{1}{2}}\|\theta_0\|_{L^\infty} 
  + C \|\Gamma\|_{\widetilde{L}_T^2(B^1_{p,\infty})}\leq C e^{\exp\{C T\}},
\end{aligned}
\end{equation*}
where in the last line we have used the embedding 
$B^1_{p,\infty}(\mathbb{R}^3)\hookrightarrow B^0_{\infty,1}(\mathbb{R}^3)$ with $p>3$.

Hence, collecting the above estimates completes the proof of Proposition \ref{prop:ap-es2}.
\end{proof}
%%%%%%%%%%%%%%%%%%%%%%%%%%%%%%%%%%%

%%%%%%%%%%%%%%%%%%%%%%%%%%%%%%%%%%%%%%%%%%%
%\section{Propagation of higher regularity}\label{sec:C2gamma}

\subsection{Persistence of the curvature}\label{sec:W2}
Note that by applying Proposition \ref{prop:ap-es2} and some novel estimates in 
\cite{GSR95}, we can argue as \cite[Section 4.2]{LLX24} to show that 
$u\in L_T^r(W^{2,\infty})$ for $r\in[1,\frac{2p}{p+3})$ uniformly in 
$\PPr\in [\PPr_*,\infty)$, which combined with Lemma~\ref{lem:flow} 
implies the desired uniform-in-$\PPr$ $W^{2,\infty}$-regularity persistence result of the temperature patch boundary
$\partial D(t)$.

Here, we develop a new and elementary approach (partially inspired by \cite{Ra22}, 
and more adaptable to the general $W^{k,\infty}$-regularity case with $k\geq 2$ \cite{LXY26}) 
to directly show that the curvature of the temperature patch boundary 
$\partial D(t)$ globally persists along the evolution and the result is independent of $\PPr\in [\PPr_*,\infty)$. 
In fact, in view of Lemma \ref{lem:sr-cond}, we will prove the following result
\begin{align}\label{eq:targ7}
  \mathcal{W} \in L^\infty(0,T; W^{1,\infty}(\R^3)),\quad \textrm{uniformly in}\;\;\PPr\in [\PPr_*,\infty), 
\end{align}
where the admissible conormal vector class $\mathcal{W}=\{W^i\}_{1\leq i\leq 5}$ 
verifies \eqref{eq:Wi}.

The main result of this section is as follows.
\begin{proposition}\label{prop:W2}
Suppose that
\begin{itemize}
\item $u_0\in H^1\cap\dot{W}^{1,p}(\DD)$ for some $3<p<\infty$;
\item $\nabla\cdot u_0=0$;
\item $\theta_0(x) = \overline{\theta}_0(x) \mathbf{1}_{D_0}(x)$,
$\overline{\theta}_0 \in C^{\mu}(\overline{D_0})$ for some $\mu\in(0,1)$,
where $D_0\subset \DD$ is a bounded simply connected domain with boundary $\partial D_0\in W^{2,\infty}$.
\end{itemize}
Let $\PPr \in [\Pr_*,\,\infty)$, with $\mathrm{Pr}_*>0$ given by \eqref{eq:data-cond}. 
Then the 3D Boussinesq system \eqref{eq:BousEq} admits a unique global solution $(u,\theta)$ that satisfies \eqref{eq:the-patch-exp}
and
\begin{equation}\label{eq:reg-parD(t)-1}
  \partial D(t) \in L^{\infty}\left(0, T ; W^{2,\infty}(\DD)\right),
  \quad \textrm{uniformly in $\PPr\in [\PPr_*,\infty)$},
\end{equation}
with $D(t)=X_t(D_0)$ and $X_t:\mathbb{R}^3\mapsto \mathbb{R}^3$ the flow map given by \eqref{eq:flowmap}.
\end{proposition}
\begin{proof}[Proof of Proposition~\ref{prop:W2}]
The maximum principle of the equation \eqref{eq:Wi} gives
\begin{align}\label{eq:W-Linf-es}
  \|\WW(t)\|_{L^\infty} \leq \|\WW_0\|_{L^\infty} 
  + \int_0^t \|\nabla u(\tau)\|_{L^\infty} \|\WW(\tau)\|_{L^\infty} \dd \tau,
\end{align}
which together with \eqref{eq:u-Gam-es2} implies that
\begin{align}\label{eq:W-Linf-es-2}
  \|\WW(t)\|_{L^\infty} \leq \|\WW_0\|_{L^\infty} e^{\int_0^t \|\nabla u(\tau)\|_{L^\infty}\dd \tau}
  \leq C e^{\exp\{Ct\}}.
\end{align}

Applying the operator $\nabla$ to equation \eqref{eq:Wi} yields
\begin{align}\label{eq:nabW}
  \partial_t \nabla \WW+u \cdot \nabla(\nabla \WW)
  =\partial_\WW \nabla u+\nabla \WW \cdot \nabla u-\nabla u \cdot \nabla \WW,
  \quad \nabla \WW|_{t=0} = \nabla \WW_0,
\end{align}
where $\mathcal{W}=\{W^i\}_{1\leq i\leq 5}$ 
is the admissible class of divergence-free conormal vectors and 
$\partial_\WW = \{\partial_{W^i}\}_{1\leq i\leq 5}$,
$\WW\cdot \nabla = \{W^i\cdot \nabla\}_{1\leq i\leq 5}$.
Taking advantage of the maximum principle again, one can get
\begin{equation}\label{eq:nabW-Linf}
\begin{aligned}
  \|\nabla \WW(t)\|_{L^\infty} \leq & \left\|\nabla \WW_0\right\|_{L^\infty}
  + \int_0^t\left\|\partial_\WW \nabla u(\tau) \right\|_{L^\infty} \mathrm{d} \tau
  + 2 \int_0^t\|\nabla u(\tau)\|_{L^\infty}\|\nabla \WW(\tau)\|_{L^\infty} \mathrm{d} \tau .
\end{aligned}
\end{equation}
The main goal is to bound the $\partial_\WW \nabla u$ in $L^1_t(L^\infty)$. Using the identity \eqref{eq.v.GamThe0},
we find that
\begin{align}\label{eq:parWu-Linf}
  \|\partial_\WW \nabla u\|_{L^1_t(L^\infty)} 
  \leq \big\|\partial_\WW \nabla^2 \Lambda^{-2} \Gamma\big\|_{L^1_t(L^\infty)}
  + \big\|(\partial_\WW \nabla^3 \Lambda^{-4} \theta, \partial_\WW \nabla \Lambda^{-2} \theta) \big\|_{L^1_t(L^\infty)}.
\end{align}
For the first term, using \eqref{es:nabu-Gam2}, \eqref{es:nabu-Gam3} and \eqref{eq:W-Linf-es-2}, 
we obtain that for $3<p<\infty$,
\begin{align*}
  \|\partial_\WW \nabla^2 \Lambda^{-2} \Gamma\|_{L^1_t(L^\infty)} 
  & \leq C \|\WW\|_{L^\infty_t(L^\infty)} \|\nabla^3 \Lambda^{-2} \Gamma\|_{L^1_t(L^\infty)} \\
  & \leq C \|\WW\|_{L^\infty_t(L^\infty)} \big( \|\Delta_{-1}(\nabla^3 \Lambda^{-2} \Gamma)\|_{L^1_t(L^p)} + 
  \|\Gamma\|_{L^1_t(B^1_{\infty,1})} \big) \\
  & \leq C \|\WW\|_{L^\infty_t(L^\infty)} \|\Gamma\|_{L^1_t(L^p \cap B^1_{\infty,1})} \leq C e^{\exp\{C t\}}.
\end{align*}
By virtue of \eqref{es:prod-es6} in Lemma \ref{lem:str-es1} and the fact \eqref{eq:fact-commutator},
%$\partial_\WW \nabla^3\Lambda^{-4} \theta=\nabla^3\Lambda^{-4} \partial_\WW\theta - [\nabla^3\Lambda^{-4}, \WW\cdot\nabla]\theta$, 
one can infer that
\begin{align*}
  \big\|\partial_\WW \nabla^3\Lambda^{-4} \theta \big\|_{L^1_t(L^\infty)} 
  & \leq C \|\nabla^3\Lambda^{-4} \partial_\WW\theta \|_{L^1_t(B^\mu_{\infty, \infty})} 
  + \|[\nabla^3\Lambda^{-4}, \WW\cdot\nabla]\theta\|_{L^1_t(L^\infty)}\\
  & \leq C e^{\exp\{C t\}}
  + C \|\partial_\WW\theta\|_{L^1_t(C^{\mu-1})} 
  + C \int_0^t \|\nabla \WW(\tau)\|_{L^\infty} 
  \|\theta(\tau)\|_{L^1\cap L^\infty} \dd \tau,
\end{align*}
where in the last line we have used the following estimate
\begin{align*}
  \|\nabla^3\Lambda^{-4} \partial_\WW\theta\|_{L^1_t(B^\mu_{\infty, \infty})} 
  & \leq C \|\Delta_{-1}\nabla^3\Lambda^{-4} \text{div} (\WW\theta)\|_{L^1_t(L^\infty)} 
  + \sup\limits_{q\geq 0}2^{q\mu}\|\Delta_{q}\nabla^3\Lambda^{-4} \partial_\WW\theta\|_{L^1_t(L^\infty)} \\
  & \leq C\|\WW\|_{L^\infty_t(L^\infty)} \|\theta\|_{L^1_t(L^2)} 
  + C\sup\limits_{q\geq 0}2^{q(\mu-1)} 
  \|\Delta_{q} \partial_\WW\theta\|_{L^1_t(L^\infty)} \\
  & \leq C e^{\exp\{C t\}}
  + C\|\partial_\WW\theta\|_{L^1_t(C^{\mu-1})} .
\end{align*}
The estimation of $\big\|\partial_\WW \nabla \Lambda^{-2} \theta \big\|_{L^1_t(L^\infty)}$ 
is exactly the same as above.

Since $\theta$ solves the transport equation, we have
\begin{align*} 
  \|\theta(t)\|_{L^1\cap L^\infty}=\|\theta_0\|_{L^1\cap L^\infty}\leq C , 
  \quad \forall t\geq 0.
\end{align*}
Noting that $\partial_\WW\theta$ also solves the transport equation
\begin{align}\label{eq:par-W-the}
  \partial_t (\partial_\WW \theta) + u\cdot \nabla(\partial_\WW \theta) =0,\quad \partial_\WW \theta |_{t=0} = \partial_{\WW_0} \theta_0,
\end{align}
we apply \eqref{eq:T-sm3}, Lemma \ref{lem:str-reg} and \eqref{eq:u-Gam-es2} to deduce that for $\mu\in(0,1)$,
\begin{align}\label{eq:parWthe-Bes2}
  \|\partial_\WW \theta(t)\|_{C^{\mu-1}} \leq C  e^{C \int_0^t \|\nabla u(\tau)\|_{L^\infty}
  \dd \tau} \|\partial_{\WW_0}\theta_0\|_{C^{\mu-1}} \leq C  e^{\exp\{C t\}}.
\end{align}

Substituting these estimates into \eqref{eq:nabW-Linf}, we obtain
\begin{align*}
  \|\nabla \WW(t)\|_{L^\infty} 
  \leq C e^{\exp \{Ct\}} 
  + C \int_0^t \big(\|\nabla u(\tau)\|_{L^\infty} + 1\big)\|\nabla \WW(\tau)\|_{L^\infty} \mathrm{d} \tau .
\end{align*}
Gr\"onwall's inequality and \eqref{eq:u-Gam-es2}, \eqref{eq:W-Linf-es-2} imply that 
\begin{align}\label{eq:W-W1inf-bdd}
  \|\mathcal{W}\|_{L^\infty_T(W^{1,\infty})} \leq \|\mathcal{W}\|_{L^\infty_T(L^\infty)}  
  + \|\nabla \WW\|_{L^\infty_T( L^\infty)} \leq Ce^{\exp \{CT\}},
\end{align}
where $C>0$ is independent of $\PPr$.
Consequently, we conclude \eqref{eq:targ7} and the global persistence of 
$W^{2,\infty}$-boundary regularity of $\partial D(t)$ as in \eqref{eq:reg-parD(t)-1}.
\end{proof}

\subsection{Persistence of the \texorpdfstring{$C^{2,\gamma}$}{C(2,gamma)}-boundary regularity}\label{sec:C2gamma}
In this section, we focus on proving the global persistence of the $C^{2,\gamma}$-regularity of the temperature patch boundary
$\partial D(t)$ for every $\gamma\in(0,1)$. In particular, we show that this result holds uniformly with respect to $\PPr\in [\PPr_*,\infty)$.

By Lemma \ref{lem:sr-cond}, it is sufficient to show that \eqref{eq:targ-sr} holds for $k=2$.
In fact, we will prove a stronger result, namely,
\begin{align}\label{eq:targ5}
  \mathcal{W} \in L^\infty(0,T; C^{1,\gamma}(\R^3)),\quad \textrm{uniformly in}\;\;\PPr\in [\PPr_*,\infty),
\end{align}
where $\mathcal{W}=\{W^i\}_{1\leq i\leq 5}$ is the admissible conormal vector system that verifies \eqref{eq:Wi}.
The main result of this section is as follows.
\begin{proposition}\label{prop:C2gam}
Let $\gamma\in(0,1)$. Suppose that
\begin{itemize}
 \item $u_0\in H^1\cap W^{2,p}(\DD)$ for some $\frac{3}{2-\gamma}<p<\infty$;
\item $\nabla\cdot u_0=0$;
\item $\theta_0(x) = \overline{\theta}_0(x) \mathbf{1}_{D_0}(x)$, $\overline{\theta}_0 \in C^{\gamma}(\overline{D_0})$,
where $D_0\subset \DD$ is a bounded simply connected domain with boundary $\partial D_0\in C^{2,\gamma}$.
\end{itemize}
Let $\PPr \in [\Pr_*,\,\infty)$, with $\mathrm{Pr}_*>0$ given by \eqref{eq:data-cond}. 
Then there exists a unique global solution $(u,\theta)$
to the 3D Boussinesq system \eqref{eq:BousEq} that satisfies \eqref{eq:the-patch-exp}
and 
\begin{equation}\label{eq:reg-parD(t)-2}
  \partial D(t) \in L^{\infty}\left(0, T ; C^{2,\gamma}\right), \quad \textrm{uniformly in}\;\;\PPr\in [\PPr_*,\infty), 
\end{equation}
where $D(t)=X_t(D_0)$, $X_t:\mathbb{R}^3\mapsto \mathbb{R}^3$ is the flow map that solves \eqref{eq:flowmap}.
\end{proposition}

\begin{proof}[Proof of Proposition \ref{prop:C2gam}]
Recall that $\nabla \mathcal{W} = \{\nabla W^i\}_{1\leq i\leq 5}$ 
satisfies the equation \eqref{eq:nabW}.
Using \eqref{eq:T-sm2}, the embedding $L_t^1(C^\gamma)\hookrightarrow 
\widetilde{L}_t^1(C^\gamma)$ 
and the inequality $\|fg\|_{C^\gamma}\leq C\|f\|_{C^\gamma}\|g\|_{C^\gamma}$, one can get
\begin{equation}\label{eq:nabW-Cgam}
\begin{aligned}
  \|\nabla \WW(t)\|_{C^\gamma} \leq & C \left\|\nabla \WW_0\right\|_{C^\gamma}
  + C \int_0^t\left\|\partial_\WW \nabla u(\tau) \right\|_{C^\gamma} \mathrm{d} \tau
  + C \int_0^t\|\nabla u(\tau)\|_{C^\gamma}
  \|\nabla \WW(\tau)\|_{C^\gamma} \mathrm{d} \tau .
\end{aligned}
\end{equation}
The main goal is to bound $\partial_\WW \nabla u$ in $L^1_t(C^\gamma)$. 
Using the identity \eqref{eq.v.GamThe0},
we find
\begin{align}\label{eq:parWu-Cgam}
  \|\partial_\WW \nabla u\|_{L^1_t (C^\gamma)} 
  \leq \big\|\partial_\WW \nabla^2 \Lambda^{-2} \Gamma\big\|_{L^1_t(C^\gamma)}
  + \big\|(\partial_\WW \nabla^3 \Lambda^{-4} \theta ,
  \partial_\WW \nabla \Lambda^{-2} \theta) \big\|_{L^1_t(C^\gamma)}.
\end{align}
%Note that the Section 3.3 of \cite{CMX22} have provided the estimate for $\alpha=1$ and $\mathbf{D}=\RR^2$,
%but in order to clarify the dependence of parameter $\varepsilon$ and deal with the case for $\mathbf{D}=\TT^2$,
%we here sketch the proof by using a similar way.
We first consider the second term on the right-hand side of \eqref{eq:parWu-Cgam}.
Taking advantage of estimates \eqref{eq:prod-es}, \eqref{eq:prod-es4}, \eqref{eq:W-W1inf-bdd} and \eqref{eq:parWthe-Bes2} 
(with $\gamma$ in place of $\mu$), we infer that
\begin{align*}
  & \big\|\partial_\WW\big(\nabla^3 \Lambda^{-4} \theta\big)\big\|_{L_t^1\left(C^\gamma\right)} \nonumber \\
  &\leq C \big\|\Delta_{-1} \partial_\WW\big(\nabla^3 \Lambda^{-4}
  \theta\big)\big\|_{L_t^1\left(L^2\right)}
  + C \big\|\nabla \partial_\WW\big(\nabla^3 \Lambda^{-4} \theta\big)
  \big\|_{L_t^1\left(B_{\infty, \infty}^{\gamma-1}\right)} \nonumber \\
  &\leq C \|\WW\|_{L_t^{\infty}\left(L^{\infty}\right)}\|\theta\|_{L_t^1\left(L^2\right)}
  + C \|\nabla \WW\|_{L^\infty_t(L^{\infty})} \big\|\nabla^4 \Lambda^{-4}
  \theta\big\|_{L^1_t(B_{\infty, \infty}^{\gamma-1})} 
  + C \big\| \partial_\WW\big(\nabla^4 \Lambda^{-4} \theta\big)
  \big\|_{L_t^1\left(B_{\infty, \infty}^{\gamma-1}\right)} \nonumber \\
  &\leq C \|\mathcal{W}\|_{L^\infty_t(W^{1,\infty})} \|\theta\|_{L^1_t(L^2\cap L^\infty)} 
  + C \left\|\partial_\WW \theta\right\|_{L_t^1\left(B_{\infty, \infty}^{\gamma-1}\right)} 
  \leq C e^{\exp\{Ct\}} .
\end{align*}
The estimate for $\partial_\WW\big(\nabla\Lambda^{-2} \theta\big)$ is similar, which altogether yields
\begin{align}\label{es:ParWtheta0}
  \big\|\big(\partial_\WW\nabla^3 \Lambda^{-4} \theta,
  \partial_\WW \nabla\Lambda^{-2} \theta\big)\big\|_{L_t^1\left(C^\gamma\right)} 
  \leq C e^{\exp\{Ct\}} .
\end{align}
For the first term on the right-hand side of \eqref{eq:parWu-Cgam}, 
thanks to the striated estimate \eqref{eq:prod-es4} (with $\sigma=0$), Proposition \ref{prop:ap-es2} and \eqref{eq:W-W1inf-bdd}, 
we obtain 
\begin{equation}\label{es:parWGa}
\begin{split}
  \big\|\partial_\WW\big(\nabla^2 \Lambda^{-2} \Gamma\big)\big\|_{L_t^1\left(C^\gamma\right)}
  & \leq C\left\|\partial_\WW \Gamma\right\|_{L_t^1\left(C^\gamma\right)}
  + C \|\WW\|_{L^\infty_t(W^{1, \infty})} \|\Gamma\|_{L^1_t(C^\gamma)}  \\
  & \leq C\left\|\partial_\WW \Gamma\right\|_{L_t^1\left(C^\gamma\right)} + C e^{\exp\{Ct\}} .
\end{split}
\end{equation}
We intend to obtain the control of $\partial_\WW \Gamma=\WW \cdot \nabla \Gamma$. 
Observe that $\partial_\WW \Gamma$ solves the following equation
\begin{align*}
  \partial_t\left(\partial_\WW \Gamma\right) + u \cdot \nabla\left(\partial_\WW \Gamma\right)
  & -\PPr\,\Delta\left(\partial_\WW \Gamma\right)
  =-\PPr\,\left[\Delta, \partial_\WW\right] \Gamma
  +\partial_\WW\left(\left[\mathcal{R}_{-1}, u \cdot \nabla\right] \theta \right) +\partial_\WW\left(\Omega\cdot \nabla u \right) \\
  & = -\PPr\,\Delta \WW \cdot \nabla \Gamma-2\PPr\, \nabla \WW: \nabla^2 \Gamma
  + \partial_\WW\left(\left[\mathcal{R}_{-1}, u \cdot \nabla\right] \theta\right)+ \partial_\WW\left(\Omega\cdot \nabla u\right),
\end{align*}
with initial data $\partial_\WW \Gamma|_{t=0}=\partial_{\WW_0} \Gamma_0$.

%\textbf{(i). $\frac{3}{2-\gamma}<p\leq3$, let $r>3$ satisfies $W^{2,p}(\R^3)\subset W^{1,r}(\R^3)$.}

For convenience, we decompose $\partial_W\Gamma$ as
\begin{equation*}
 \partial_W\Gamma = F^{(1)} + F^{(2)},
\end{equation*}
where $F^{(1)}$ and $F^{(2)}$ respectively solve 
\begin{equation}\label{eq:F1}
\left\{
\begin{aligned}
\partial_t F^{(1)} + u\cdot\nabla F^{(1)} - \Pr \Delta F^{(1)}
 &= - \Pr \Delta W \cdot \nabla \Gamma
    - 2\Pr \nabla W : \nabla^2 \Gamma  \\
 &\quad + \partial_W \bigl([{\mathcal R}_{-1},u\cdot\nabla]\theta\bigr)
    + \partial_W (\Omega \cdot \nabla u), \\
F^{(1)}\big|_{t=0} &= 0 ,
\end{aligned}
\right.
\end{equation}
and
\begin{equation}\label{eq:F2}
\begin{aligned}
  \partial_t F^{(2)} + u\cdot\nabla F^{(2)} - \Pr \Delta F^{(2)} = 0, \quad
  F^{(2)}\big|_{t=0} = \partial_{W_0}\Gamma_0 .
\end{aligned}
\end{equation}
Thanks to the smoothing estimate \eqref{eq:TD-sm5} in Lemma~\ref{lem:TD-sm2}, 
we find that 
\begin{align*}
  & \PPr \, \Big(\sup_{j\in \mathbb{N}} 2^j  \|\Delta_j F^{(1)}\|_{L^1_t(L^\infty)} \Big)
  + \sqrt{\PPr}\, \Big(\sup_{j\in \mathbb{N}} \|\Delta_j F^{(1)}\|_{L^2_t(L^\infty)} \Big) \\
  & \leq Ce^{C\int_0^t\|u(\tau)\|_{B^1_{\infty,1}} \dd\tau}
  \bigg(\PPr \|\Delta \WW \cdot \nabla \Gamma\|_{L^1_t(B_{\infty, \infty}^{-1})}
  + \PPr \left\|\nabla \WW: \nabla^2 \Gamma\right\|_{L^1_t(B_{\infty,\infty}^{-1})}  \\
  &\quad + \big\|\partial_\WW\left[\mathcal{R}_{-1}, u \cdot \nabla\right] 
  \theta\big\|_{L^1_t(B^{-1}_{\infty,\infty})} 
  + \big\|\partial_{\mathcal{W}}(\Omega\cdot \nabla u) 
  \big\|_{L_t^1\left(B_{\infty, \infty}^{-1}\right)}\bigg) .
\end{align*}
By virtue of the product estimate \eqref{eq:prodBes-endp} 
and bounds \eqref{es:nabu-Gam3}, \eqref{eq:W-W1inf-bdd}, we have
\begin{align*}
  \|\Delta \mathcal{W}\cdot\nabla \Gamma\|_{L^1_t(B^{-1}_{\infty,\infty})}
  & \leq C \|\Delta \mathcal{W} \|_{L^\infty_t(B^{-1}_{\infty,\infty})} 
  \|\nabla \Gamma\|_{L^1_t(B^0_{\infty,1})} \\
  & \leq C \|\mathcal{W}\|_{L^\infty_t(W^{1,\infty})} 
  \|\Gamma\|_{L^1_t(B^1_{\infty,1})} \leq C e^{\exp\{Ct\}},
\end{align*}
and
\begin{align*}
  \|\nabla \mathcal{W}:\nabla^2\Gamma\|_{L^1_t(B^{-1}_{\infty,\infty})}
  & = \| \mathrm{div}\,(\nabla \mathcal{W} \, \nabla\Gamma)\|_{L^1_t(B^{-1}_{\infty,\infty})} 
  \leq \|\nabla \mathcal{W}\, \nabla\Gamma\|_{L^1_t(L^\infty)} \\
  & \leq \|\mathcal{W}\|_{L^\infty_t(W^{1,\infty})} 
  \|\Gamma\|_{L^1_t(W^{1,\infty})} \leq C e^{\exp\{Ct\}}.
\end{align*}
Using estimates \eqref{eq:R-1cm-es1} (with $q=6$), \eqref{eq:u-Gam-es2}
together with the embedding $L^\infty(\mathbb{R}^3)\hookrightarrow 
B^{-1/2}_{\infty,\infty}(\mathbb{R}^3)$ and $\dot{H}^1(\mathbb{R}^3) 
\hookrightarrow L^6(\mathbb{R}^3)$, we deduce that
\begin{align*}
\left\|\left[\mathcal{R}_{-1}, u \cdot \nabla\right]
  \theta\right\|_{L_t^1 \big(B_{\infty, 1}^0 \big)}
  &\leq  C\|[\mathcal{R}_{-1},u\cdot\nabla]\theta 
  \|_{L^1_t\big(B^{\frac{1}{2}}_{\infty,\infty}\big)} \\
  &\leq  C  \Big(\|\nabla u\|_{L^1_t(L^\infty)} 
  \|\theta\|_{L^\infty_t\big(B^{-\frac{1}{2}}_{\infty,\infty}\cap L^2\big)}
  + \|u\|_{L^1_t(L^6)} \|\theta\|_{L^\infty_t(L^{\frac{6}{5}})} \Big) \\
  &\leq  C  \Big(\|\nabla u\|_{L^1_t(L^\infty)} 
  \|\theta\|_{L^\infty_t (L^\infty\cap L^2 )}
  + \|\nabla u\|_{L^1_t(L^2)} \|\theta\|_{L^\infty_t(L^{\frac{6}{5}})} \Big) \\
  &\leq C  e^{\exp\{Ct\}},
\end{align*}
which, combined with estimate \eqref{eq:W-Linf-es-2}, leads to
\begin{align*}
  \left\|\partial_\WW\left(\left[\mathcal{R}_{-1}, u \cdot \nabla\right] \theta\right)
  \right\|_{L_t^1(B_{\infty, \infty}^{-1})}
  & = \|\mathrm{div}\,(\mathcal{W}\,[\mathcal{R}_{-1},u\cdot\nabla]\theta)\|_{L^1_t(B^{-1}_{\infty,\infty})} \\
  & \leq  C  \|\WW\|_{L_t^\infty(L^\infty)} \left\|\left[\mathcal{R}_{-1}, u \cdot \nabla\right]
  \theta\right\|_{L_t^1 (B_{\infty, 1}^0)} \leq  C e^{\exp\{Ct\}}.
\end{align*}
%In light of \eqref{eq.v.GamThe0}, \eqref{es:nabu-Gam2} together with the embedding
%$\widetilde{L}^2_t(B^1_{r,\infty}) \hookrightarrow L^2_t(B^{\gamma'-1}_{\infty,1})$ for every $\gamma'\in (1,2-\frac{3}{r})$,
%we find
%\begin{align}\label{es.F12nv}
%  \|\nabla u\|_{L^2_t(B^{\gamma'-1}_{\infty,1})} 
%  &\leq \|\Lambda^{-2}\nabla \nabla\wedge \Gamma\|_{L^2_t(B^{\gamma'-1}_{\infty,1})}
%  + \|\nabla^2 \partial_3 \Lambda^{-4}\theta \|_{L^2_t(B^{\gamma'-1}_{\infty,1})} 
%  + \|\nabla \Lambda^{-2}\theta\|_{L^2_t(B^{\gamma'-1}_{\infty,1})} \nonumber \\
%  &\lesssim  \| \Gamma\|_{L^2_t(L^r\cap B^{\gamma'-1}_{\infty,1})} + \| \Lambda^{-1} \theta\|_{L^2(L^p)} 
%  + \|\theta\|_{L^2_t(B^{\gamma'-2}_{\infty,1})} \nonumber \\
%  &\lesssim  \| \Gamma\|_{\widetilde{L}^2_t(B^1_{r,\infty})} + \| \theta\|_{L^2_t (L^{\frac{3r}{r+3}})} 
%  + \|\theta\|_{L^2_t(L^\infty)} \nonumber \\
%  & \lesssim e^{\exp\{Ct\}}.
%\end{align}
From \eqref{es:nabu-Gam3} (noting that even for $\frac{3}{2-\gamma}<p\leq3$, there exists
$r>3$ so that $W^{2,p}(\R^3) \hookrightarrow W^{1,r}(\R^3)$) and \eqref{eq:W-Linf-es-2}, we similarly get
\begin{align*}
  \left\|\partial_\WW\left(\Omega\cdot\nabla u\right) 
  \right\|_{L^1_t(B_{\infty, \infty}^{-1})}
  & = \|\mathrm{div}\,(\mathcal{W}\, \Omega\cdot\nabla u)
  \|_{L^1_t(B^{-1}_{\infty,\infty})} \\
  &\leq C\|\WW\|_{L^\infty_t(L^\infty)} 
  \left\|\Omega\cdot\nabla u\right\|_{L^1_t(L^\infty)} \\
  & \leq C\|\WW\|_{L^\infty_t(L^\infty)} 
  \left\|\nabla u\right\|^2_{L^2_t(L^\infty)} 
  \leq C e^{\exp\{Ct\}}.
\end{align*}
Collecting the above estimates and \eqref{eq:u-Gam-es2} yields
\begin{align*}
  \Big(\sup_{j\in \mathbb{N}} 2^j  \|\Delta_j F^{(1)}\|_{L^1_t(L^\infty)} \Big)
  +  \frac{1}{\sqrt{\PPr}}\, \Big(\sup_{j\in \mathbb{N}} \|\Delta_j F^{(1)}\|_{L^2_t(L^\infty)} \Big)
  \leq C e^{\exp\{Ct\}},
\end{align*}
where $C>0$ depends on $\PPr_*$ but is independent of $\PPr$, which immediately implies 
\begin{align}\label{eq:F1-es}
  \sum_{j\in \mathbb{N}} 2^{j\gamma}  \|\Delta_j F^{(1)}\|_{L^1_t(L^\infty)} 
  +  \frac{1}{\sqrt{\PPr}}\,\sum_{j\in \mathbb{N}} 2^{j(\gamma-1)}
  \|\Delta_j F^{(1)}\|_{L^2_t(L^\infty)} \leq C e^{\exp\{Ct\}}.
\end{align}
For the estimation of $F^{(2)}$ satisfying \eqref{eq:F2}, 
according to \eqref{TD-sm-es}-\eqref{eq:TD-sm5} in Lemma \ref{lem:TD-sm2}, 
we see that if $\frac{3}{2-\gamma}< p \leq 3$,
\begin{align*}
  \PPr \, \sum_{j\in\mathbb{N}} 2^{j(\gamma+1)} \|\Delta_j F^{(2)}\|_{L^1_t(L^3)}
  + \sqrt{\PPr}\, \sum_{j\in\mathbb{N}} 2^{j\gamma} \|\Delta_j F^{(2)}\|_{L^2_t(L^3)}
  \leq C\|\partial_{\WW_0}\Gamma_0\|_{B^{\gamma-1}_{3,1}}e^{C\int_0^t\|\nabla u\|_{L^\infty}\dd\tau},
\end{align*}
and if $p>3$,
\begin{align*}
  \PPr \, \Big(\sup_{j\in \mathbb{N}} 2^j  \|\Delta_j F^{(1)}\|_{L^1_t(L^\infty)} \Big)
  + \sqrt{\PPr}\, \Big(\sup_{j\in \mathbb{N}} \|\Delta_j F^{(1)}\|_{L^2_t(L^\infty)} \Big) 
  \leq C \|\partial_{\mathcal{W}}\Gamma_0\|_{B^{-1}_{\infty,\infty}} e^{C \int_0^t \|u\|_{B^1_{\infty,1}} \dd \tau}.
\end{align*}
In view of the relation $\Gamma_0=\Omega_0+\mathcal{R}_{-1} \theta_0$, the product estimate \eqref{eq:prod-es} and the embedding 
\begin{align*}
  \textrm{$L^p(\mathbb{R}^3) \hookrightarrow B_{3, 1}^{\gamma-1}(\mathbb{R}^3)$\quad for \,$\tfrac{3}{2-\gamma} <p \leq 3$},
  \quad\quad
  \textrm{$W^{1,p}(\mathbb{R}^3) \hookrightarrow B_{\infty,1}^0(\mathbb{R}^3)$\quad for \,$p > 3$},
\end{align*}
we infer that if $\frac{3}{2-\gamma}< p \leq 3$,
\begin{align*}
  \left\|\partial_{\WW_0} \Gamma_0\right\|_{B_{3,1}^{\gamma-1}}
  & \leq\left\|\partial_{\WW_0} \nabla u_0\right\|_{B_{3,1}^{\gamma-1}}
  + \left\|\partial_{\WW_0} \mathcal{R}_{-1} \theta_0\right\|_{B_{3,1}^{\gamma-1}} \\
   & \leq C \left\|\partial_{\WW_0} \nabla u_0\right\|_{L^p}
  + C \left\|\partial_{\WW_0} \mathcal{R}_{-1} \theta_0\right\|_{L^p} \\
  & \leq  C \left\|\WW_0\right\|_{L^{\infty}}\left\|\nabla^2 u_0\right\|_{L^p}
  + C\left\|\WW_0\right\|_{L^\infty}\left\|\nabla \mathcal{R}_{-1} \theta_0\right\|_{L^p} \\
  & \leq C \left\|\WW_0\right\|_{L^\infty}
  \left\|u_0\right\|_{W^{2, p}} + C\left\|\WW_0\right\|_{L^{ \infty}}\left\|\theta_0\right\|_{L^p}<\infty ,
\end{align*}
and if $p > 3$,
\begin{align*}
  \left\|\partial_{\WW_0} \Gamma_0\right\|_{B_{\infty, \infty}^{-1}}
  & \leq\left\|\partial_{\WW_0} \nabla u_0\right\|_{B_{\infty, \infty}^{-1}} 
  + \left\|\partial_{\WW_0} \mathcal{R}_{-1} \theta_0\right\|_{B_{\infty, \infty}^{-1}} \\
  & \leq C \left\|\WW_0 \, \nabla u_0\right\|_{L^{\infty}}
  + C \|\WW_0 \, \mathcal{R}_{-1}\theta_0 \|_{L^\infty} \\
  & \leq C\left\|\WW_0\right\|_{L^\infty}
  \left\|u_0\right\|_{W^{2, p}} + C\left\|\WW_0\right\|_{L^{ \infty}}\left\|\theta_0\right\|_{L^2 \cap L^{\infty}}<\infty .
\end{align*}
Consequently, based on \eqref{eq:u-Gam-es2} and the above estimates, we have
\begin{equation}\label{eq:F2-es}
\begin{split}
  & \sum_{j\in\mathbb{N}}2^{j\gamma} \|\Delta_j F^{(2)}\|_{L^1_t(L^\infty)}
  +  \frac{1}{\sqrt{\PPr}} \sum_{j\in\mathbb{N}} 2^{j(\gamma-1)} 
  \|\Delta_j F^{(2)}\|_{L^2_t(L^\infty)} \\
  & \leq 
  \begin{cases}
    C \sum\limits_{j\in\mathbb{N}} 2^{j(\gamma+1)} \|\Delta_j F^{(2)}\|_{L^1_t(L^3)}
    + \frac{C}{\sqrt{\PPr}} \sum\limits_{j\in\mathbb{N}} 2^{j\gamma} 
    \|\Delta F^{(2)}\|_{L^2_t(L^3)},\quad & \textrm{if}\;\; \frac{3}{2-\gamma}< p\leq 3, \\
    C \sup\limits_{j\in\mathbb{N}} 2^j \|\Delta_j F^{(2)}\|_{L^1_t(L^\infty)}
    + \frac{C}{\sqrt{\PPr}} \sup\limits_{j\in\mathbb{N}} 
    \|\Delta F^{(2)}\|_{L^2_t(L^\infty)},\quad & \textrm{if}\;\; p> 3,
  \end{cases} \\
  & \leq C e^{\exp\{Ct\}}.
\end{split}
\end{equation}
Hence, noting that 
\begin{align*}
  \|\Delta_{-1}\partial_{\mathcal{W}}\Gamma\|_{L^1_t\cap L^2_t(L^\infty)} 
  & \leq C (1+t)\|\Delta_{-1}\mathrm{div}\,(\mathcal{W}\,\Gamma)\|_{L^\infty_t (L^p)} \\
  & \leq C (1+t) \|\mathcal{W}\|_{L^\infty_t(L^\infty)} \|\Gamma\|_{L^\infty_t(L^p)} 
  \leq C e^{\exp\{Ct\}},
\end{align*}
and in combination with \eqref{eq:F1-es} and \eqref{eq:F2-es}, we have
\begin{equation}\label{es:parW-Gam2}
\begin{split}
  \left\|\partial_\WW \Gamma\right\|_{L_t^1\big(B_{\infty, 1}^{\gamma}\big)}
  & + \frac{1}{\sqrt{\PPr}}\left\|\partial_\WW \Gamma\right\|_{L_t^2\big(B_{\infty, 1}^{\gamma-1}\big)} 
  \leq C \left\|\partial_\WW \Gamma\right\|_{\widetilde{L}_t^1\big(B_{\infty, 1}^{\gamma}\big)} 
  + \frac{C}{\sqrt{\PPr}}\left\|\partial_\WW \Gamma\right\|_{\widetilde{L}_t^2\big(B_{\infty, 1}^{\gamma-1}\big)} \\
  &\leq   C \left\|\Delta_{-1}(\partial_\WW \Gamma)\right\|_{L_t^1 \cap L_t^2(L^\infty)} 
  + C \sum_{j\in \mathbb{N}} 2^{j\gamma} \|(\Delta_j F^{(1)},
  \Delta_j F^{(2)}\|_{L^1_t(L^\infty)} \\
  & \quad + \frac{C}{\sqrt{\PPr}} \sum_{j\in\mathbb{N}} 
  2^{j(\gamma-1)} \|(\Delta_j F^{(1)},\Delta_j F^{(2)})\|_{L^2_t(L^\infty)} \\
  & \leq C e^{\exp\{C t\}}.
\end{split}
\end{equation}

%\begin{remark}
%Here we cannot get
%\begin{align*}
%  \varepsilon\|W(t)\|_{L^\infty} \leq  C e^{\exp(Ct)},
%  \end{align*}
%which implies we cannot get
%\begin{align*}
%  & \varepsilon\|W\|_{L_T^{\infty}\left(\mathcal{C}_W^{ 1 +\gamma,0}\right)}
%  +\varepsilon\|\nabla u\|_{L_T^1\left(\mathcal{C}_W^{\gamma, 1}\right)}
%  +\varepsilon\|\Gamma\|_{L_T^1\left(\mathcal{B}_W^{\gamma^{\prime},1}\right)} \leq C e^{\exp\,\exp (CT)}.
%\end{align*}
%\end{remark}d
%Notice that by virtue of \eqref{es:nabu-Gam2-2},
%\begin{align}\label{es:Gamma1+}
%  \|\Gamma\|_{L_t^1\left(B_{\infty, 1}^{\gamma^{\prime}+1}\right)}
%  \leq C\|\Gamma\|_{\widetilde{L}_t^1\left(B_{p, \infty}^2\right)} \leq Ce^{CE_\varepsilon(t)}.
%\end{align}
Therefore, gathering \eqref{eq:nabW-Cgam}-\eqref{es:parWGa} and \eqref{es:parW-Gam2}
yields
\begin{align*}
  &\|\WW(t)\|_{C^{1,\gamma}}  + \|\partial_\WW \Gamma\|_{L_t^1(B_{\infty, 1}^{\gamma})}
  + \frac{1}{\sqrt{\PPr}}\left\|\partial_\WW \Gamma
  \right\|_{L_t^2\big(B_{\infty, 1}^{\gamma-1}\big)} 
  + \left\|\partial_\WW \nabla u\right\|_{L_t^1\left(C^\gamma\right)} \\
  & \leq C e^{\exp\{Ct\}} + C \int_0^t \|\WW(\tau)\|_{C^{1,\gamma}}
  \|\nabla u(\tau)\|_{C^\gamma} \dd\tau.
\end{align*}
Gr\"onwall's inequality and estimate \eqref{es:nabu-Gam4} guarantee that
\begin{equation}\label{eq:W-C2+gam-2}
\begin{split}
  & \|\WW\|_{L_T^{\infty}\left(C^{1,\gamma}\right)} 
  +\left\|\partial_\WW \Gamma\right\|_{L_t^1\big(B_{\infty, 1}^{\gamma}\big)}
  + \frac{1}{\sqrt{\PPr}}\|\partial_\WW \Gamma\|_{L_t^2(B_{\infty, 1}^{\gamma-1})} 
  +\left\|\partial_\WW \nabla u\right\|_{L_T^1\left(C^\gamma\right)} \\
  &\leq C e^{\exp\{CT\}} \exp\Big\{ C \|\nabla u\|_{L^1_T(C^\gamma)} \Big\} 
  \leq  C e^{\exp\{CT\}},
\end{split}
\end{equation}
where $C>0$ is independent of $\PPr$. This concludes the desired result \eqref{eq:targ5} 
and we complete the proof for the uniform-in-$\PPr$ persistence of the 
$C^{2,\gamma}$-boundary regularity of the temperature patch.
\end{proof}
%%%%%%%%%%%%%%%%%%%%%%%%%%%%%%%%%
\section{Infinite Prandtl number limit}\label{sec:limit}

This section is devoted to the proof of Theorem \ref{thm:limit}.
As a starting point, we present the following convergence result from the 3D Boussinesq system \eqref{eq:BousEq}
to the 3D Stokes-transport system \eqref{eq:ST} associated with the general well-prepared initial data.
\begin{proposition}\label{prop:weak-limit}
%Let $E_6$ be the subspace of probability measures in $L^6(\DD)$, which implies that $E_6$ is a subspace of $L^1\cap L^6(\DD)$. Assume that $(u_0, \theta_0) \in (H^1 \times E_6)(\DD)$ satisfies $\nabla\cdot u_0=0$.
Let $T>0$ be any given. %the constant given by \eqref{eq:data-cond}. 
For every $\mathrm{Pr} \in [\Pr_*,\infty)$ with $\PPr_*>0$ large enough 
($\PPr_*$ can be chosen to be universal due to \eqref{eq:data-cond} and the strong convergence of $(u_0^{\PPr},\theta_0^{\PPr})$), 
consider $(u^\mathrm{Pr} , \theta^\mathrm{Pr})$ the unique global regular solution of the 3D Boussinesq system
\eqref{eq:BousEq} with initial data $(u_0^{\PPr}, \theta_0^{\PPr})\in H^1(\mathbb{R}^3)
\times \big(L^1\cap L^6(\mathbb{R}^3)\big)$.
In addition, assume that by passing $\PPr\rightarrow \infty$, $(u_0^{\PPr},\theta_0^{\PPr})$ 
strongly converges to $(u_0,\theta_0) \in H^1(\mathbb{R}^3)
\times \big(L^1\cap L^6(\mathbb{R}^3)\big)$ which satisfies the compatibility condition
$-\Delta u_0 + \nabla p_0 = \theta_0 e_3$ with some function $p_0$.

Then, as $\mathrm{Pr}\rightarrow \infty$, up to the extraction of a subsequence, 
$(u^\mathrm{Pr}, \theta^\mathrm{Pr})$
converges (weakly) to the global unique solution
$(u, \theta)\in C([0,T]; H^2\cap W^{2,6}(\mathbb{R}^3)) \times 
C([0,T]; L^1\cap L^6 (\mathbb{R}^3))$
of the 3D Stokes-transport system  \eqref{eq:ST} with initial data $\theta_0$. 
\begin{comment}
More precisely, $(u^\mathrm{Pr}, \theta^\mathrm{Pr})$
converges weakly to 
$(u, \theta)$ in $L^{3/2}([0,T]; W^{1,6}(\DD))\cap L^r([0,T]; W^{1,\infty}(\DD)) \times L^\infty([0,T];L^1\cap L^6)$ with $r\in[1,\,\frac43)$.
\end{comment}
\end{proposition}

\begin{remark}
One can refer to recent work \cite[Theorem 2.1]{MS22} 
for the existence and uniqueness of the global regular solution for the 3D Stokes-transport system \eqref{eq:ST}
associated with initial data $\theta_0 \in L^p(\mathbb{R}^3)$, $p\geq 3$.
\end{remark}

\begin{proof}[Proof of Proposition~\ref{prop:weak-limit}]
Taking advantage of Proposition~\ref{prop:ap-es1}, the H\"older inequality and interpolation inequality, 
one can get that for every $\PPr\in [\PPr_*,\infty)$,
\begin{align*}
  \|\nabla u^\mathrm{Pr}\|_{L_T^{3/2}(L^6)}\leq C\|\nabla u^\mathrm{Pr} 
  \|_{L_T^{\infty}(L^2)}^{\frac13}\|\nabla u^\mathrm{Pr}\|_{L_T^{1}(L^\infty)}^{\frac23}\leq Ce^{CT},
\end{align*}
and 
\begin{align*}
  \| u^\mathrm{Pr}\|_{L_T^{3/2}(L^6)}\leq C\|\nabla u^\mathrm{Pr}\|_{L_T^{3/2}(L^2)}
  \leq CT^{\frac23}\|\nabla u^\mathrm{Pr}\|_{L_T^{\infty}(L^2)}\leq Ce^{CT},
\end{align*}
with $C>0$ independent of $\PPr$,
which implies $u^\mathrm{Pr}$ is uniformly bounded in $L_T^{3/2}(W^{1,6}(\mathbb{R}^3))$. 
We can also obtain that
$\theta^\mathrm{Pr}$ is uniformly bounded in $L_T^\infty(L^1\cap L^6(\DD))$.
Thus, up to extraction of subsequences, we have the weak (or weak-$*$) convergence of 
$(u^{\PPr}, \theta^{\PPr})$ to some function $(u,\theta)$ 
as $\mathrm{Pr}\rightarrow \infty$, that is,
\begin{align}
   u^\mathrm{Pr} \rightharpoonup u \quad\quad &\text{in}\;\; L^{3/2}_T\big(W^{1,6}(\DD)\big), \label{eq:u-weak-conv} \\
  \theta^\mathrm{Pr}\rightharpoonup^* \theta \quad\quad &\text{in}\;\; L_T^\infty \big(L^{2}(\DD)\big)~~~ 
  \text{and}~~~L_T^\infty \big(L^{6}(\DD)\big). \label{eq:the-weak-conv}
\end{align}
It follows from the lower semicontinuity of weak limit that
$u\in L^{3/2}_T(W^{1,6}(\DD))$ and $\theta\in L^{\infty}_T(L^{2}\cap L^6(\DD))$.
From the equation $\partial_t \theta^\mathrm{Pr} = -\mathrm{div}\,(u^\mathrm{Pr}\,\theta^\mathrm{Pr})$ and the uniform estimate
\begin{align*}
  \|\mathrm{div}\,(u^\mathrm{Pr}\,\theta^\mathrm{Pr})\|_{L^\infty_T(H^{-1}(B_R))} 
  \leq C \|u^\mathrm{Pr}\,\theta^\mathrm{Pr}\|_{L^\infty_T(L^2(B_R))}
  \leq C \|u^\mathrm{Pr}\|_{L^\infty_T(L^6(B_R))} \|\theta^\mathrm{Pr}\|_{L^\infty_T(L^3(B_R))} \leq C,
\end{align*}
we know that 
\begin{align*}
  \textrm{$\partial_t \theta^\mathrm{Pr}$\;\; is uniformly bounded in\;\; $L^\infty_T(H^{-1}(B_R))$\;\; for all $R>0$}.
\end{align*}
Since $L^{2}(B_R) \hookrightarrow H^{-1}(B_R)$ is compact for all $R>0$, 
we apply the Aubin-Lions lemma %or the Rellich compactness theorem 
(see Lemma \ref{lem:Aubin-Lions}) and a standard diagonal process to deduce the strong convergence of $\theta^{\PPr}$:
\begin{align}\label{eq:the-str-conv}
  \theta^\mathrm{Pr} \rightarrow \theta\quad \textrm{in}\;\; C\big([0,T]; H^{-1}(B_R)\big)~~~\text{for all}~R>0.
\end{align}

We only need to show the convergence of the nonlinear terms, %$u^\mathrm{Pr} \cdot \nabla \theta^\mathrm{Pr}$,
since the linear ones can be dealt with in a standard way.
For all $\zeta\in C_c^\infty(\DD\times[0,T])$, we have that as $\mathrm{Pr} \rightarrow \infty$,
\begin{align*}
  \frac{1}{\mathrm{Pr}} \Big|\int_0^T \int_{\DD} \big(u^\mathrm{Pr} \cdot\nabla u^\mathrm{Pr}\big) \zeta\dd x \dd t \Big|
  & \leq \frac{1}{\mathrm{Pr}} \|u^\mathrm{Pr}\|_{L^\infty_T(L^6)} 
  \|\nabla u^\mathrm{Pr}\|_{L^\infty_T(L^2)} \|\zeta\|_{L^1_T(L^3)} \\
  & \leq \frac{C }{\mathrm{Pr}} \|\nabla u^\mathrm{Pr}\|_{L^\infty_T(L^2)}^2 \|\zeta\|_{L^1_T(L^3)} \leq \frac{C}{\PPr} \rightarrow 0,
\end{align*}
and
\begin{align}\label{eq:converg-nonl}
  & \Big|\int_0^T \int_{\DD} \mathrm{div}\, (u^\mathrm{Pr}\,\theta^\mathrm{Pr})\, \zeta\, \dd x \dd t -
  \int_0^T \int_{\DD} \mathrm{div}\, (u\,\theta)\, \zeta\, \dd x \dd t  \Big| \nonumber \\
  & \leq \Big|\int_0^T \int_{\DD} (u^\mathrm{Pr} - u)\cdot \nabla \zeta \, \theta \, \dd x \dd t\Big|
  + \Big|\int_0^T \int_{\DD} u^\mathrm{Pr} \cdot \nabla \zeta\,(\theta^\mathrm{Pr} -\theta) \dd x \dd t \Big| \nonumber \\
  & \leq \Big|\int_0^T \int_{\DD} (u^\mathrm{Pr} - u)\cdot \nabla \zeta \, \theta \, \dd x \dd t\Big|
  + \|u^\mathrm{Pr}\cdot \nabla \zeta\|_{L^1_T(H_0^1(B_R))} \|\theta^\mathrm{Pr} -\theta\|_{L^\infty_T(H^{-1}(B_R))} \rightarrow 0,
\end{align}
where the convergence in \eqref{eq:converg-nonl} follows from \eqref{eq:u-weak-conv}, \eqref{eq:the-str-conv}
and the estimate that $$\|u^\mathrm{Pr}\cdot\nabla \zeta\|_{L^1_T(H_0^1(B_R))}
\leq C\|u^\mathrm{Pr}\|_{L^{3/2}_T(W^{1,6}(B_R))} \|\zeta\|_{L^3_T(W^{2,3}(B_R))} \leq C.$$
Hence, by letting $\mathrm{Pr}\rightarrow \infty$,
we can arrive at the limiting system and show that $(u,\theta)$ 
solves the 3D Stokes transport system \eqref{eq:ST} in the sense of distribution.

Referring to Proposition~\ref{prop:ap-es1}, $\nabla u^\mathrm{Pr}$ is uniformly bounded in $L_T^r(L^\infty)$ for $r\in(1,\frac43)$, 
one can get that $\nabla u^\mathrm{Pr} \rightharpoonup^* \nabla u~~~ \text{in}~~~ L^{r}_T\big(L^\infty(\DD)\big)$, 
which yields $\nabla u\in L_T^r(L^\infty)$. Let $X_t$ be the flow map generated by $u$, 
which means that $X_t$ solves
\begin{align}\label{eq:X_t}
  \frac{\dd}{\dd t} X_t(x) = u\big(X_t(x),t\big),\quad \mathrm{div}\,u =0,\quad
  X_t(x) |_{t=0} = x.
\end{align}
Taking advantage of Lemma~\ref{lem:flow}, the system \eqref{eq:X_t} has a unique solution
$X_t(\cdot):\DD \rightarrow \DD$ on $[0,T]$ which is a measure-preserving bi-Lipschitzian homeomorphism. 
Since $\theta$ solves the transport equation $\eqref{eq:ST}_1$ in the sense of distribution, 
we can infer that $\theta(t,x)=\theta_0(X_t^{-1}(x))$, which implies $\theta\in C([0,T]; L^1\cap L^6(\mathbb{R}^3))$. 
Since $u$ solves the second equation $\eqref{eq:ST}_2$ in the sense of distribution, 
we have $-\Delta u = \mathbb{P} (\theta e_3)$ with $\mathbb{P} = \mathrm{Id} - \nabla \Delta^{-1} \mathrm{div}$ 
the Leray projection operator, consequently, $u\in C([0,T]; H^2\cap W^{2,6}(\mathbb{R}^3))$ by the classical elliptic theory.
\begin{comment}
Finally, we show $\theta\in C(0,T; E_6)$, which directly implies that $u \in C(0,T; W^{2,6})$ since $u=(-\Delta)^{-1}\mathbb{P}\,\theta\, e_3$.
Moreover thanks to \cite{MS22}, we have for all
$0 \le t \le t' \le T$,
\begin{equation}\label{eq:W1-estimate}
\begin{aligned}
W_1\bigl(\theta(t),\theta(t')\bigr)
&\le \int_{\mathbb{R}^3} \bigl|X_{t'}(x)-X_t(x)\bigr|\,\theta_0(\dd x) \\
&\le (t'-t)\,\|u\|_{L^\infty(0,T;L^\infty(\mathbb{R}^3))}
   \,\|\theta_0\|_{L^1(\mathbb{R}^3)},
\end{aligned}\end{equation}
this shows that
$\theta \in C\bigl([0,T];\mathcal{P}(\mathbb{R}^3)\cap L^1(\mathbb{R}^3)\cap
L^6(\mathbb{R}^3)\bigr)$ for any $T>0$, which completes the proof of  $\theta\in C(0,T; E_6)$.
\end{comment}
\end{proof}

Now we give the proof of Theorem \ref{thm:limit}.
\begin{proof}[Proof of Theorem \ref{thm:limit}]
By applying Proposition \ref{prop:weak-limit}, we immediately get the convergence of $(u^{\PPr},\theta^{\PPr})$ solving system \eqref{eq:BousEq}
to the unique global solution $(u,\theta)$ of the 3D Stokes-transport system \eqref{eq:ST}, 
and also $\theta(x, t) = \overline{\theta}_0(X_t^{-1}(x)) \mathbf{1}_{D(t)}(x) $ is of the patch type. 
Below, we show the regularity persistence result of patch boundary $\partial D(t)$.

Thanks to Proposition \ref{prop:ap-es2}, we have
$u^\mathrm{Pr}\in L^r_T(C^{1,\gamma}(\DD))$ for $r\in [1,\frac{2p}{3+\gamma p})$
uniformly in $\mathrm{Pr}\in [\PPr_*,\infty)$, one can obtain that as $\PPr\rightarrow \infty$,
\begin{align*}
  \textrm{$u^\mathrm{Pr} \rightharpoonup^* u\quad \text{in}\;\; L^{r}_T\big(C^{1,\gamma}(\DD)\big)$}, 
\end{align*}
and the lower semicontinuity of weak limit also yields $u\in L_T^r(C^{1,\gamma})$.
Hence, using Lemma~\ref{lem:flow}, we find that the system \eqref{eq:X_t} has a unique solution
$X_t(\cdot):\DD \rightarrow \DD$ on $[0,T]$ that is a measure-preserving homeomorphism that satisfies both 
$X_t$ and its inverse $X_t^{-1}$ belong to $L^\infty_T(C^{1,\gamma}(\DD))$.
The regularity of $X_t^{\pm1}$ implies the global persistence of $C^{1,\gamma}$-regularity of $\partial D(t)$,
which yields the estimate of \eqref{eq:limit-targ-1}.

Next we turn to the proof of \eqref{eq:limit-targ-2}. Recalling \eqref{eq:Wi} and 
$\WW_0 = \{W^i_0\}_{1\leq i\leq 5} \in W^{1,\infty}(\DD)$ satisfying \eqref{W-i-0}, we have that $\WW^\mathrm{Pr}=\{W^{\PPr,i}\}_{1\leq i\leq 5}$ satisfies
\begin{equation*}
  \partial_t \WW^\mathrm{Pr} + u^\mathrm{Pr} \cdot\nabla \WW^\mathrm{Pr} = \WW^\mathrm{Pr}\cdot\nabla u^\mathrm{Pr} 
  = \partial_{\WW^\mathrm{Pr}}u^\mathrm{Pr}, \quad\quad \WW^\mathrm{Pr}|_{t=0}(x)=\WW_0(x),
\end{equation*}
Referring to \eqref{eq:W-W1inf-bdd} and \eqref{es:nabu-Gam3}, 
we find that $\WW^\mathrm{Pr}$ is uniformly bounded in 
$L_T^\infty(W^{1,\infty}(\DD))$ and $u^\mathrm{Pr}$ is uniformly bounded in 
$L_T^2(W^{1,\infty}(\DD))$ with respect to $\PPr\in [\PPr_*,\infty)$,
which implies that $\partial_t\WW^\mathrm{Pr}$ belongs to $L_T^2(L^\infty(\DD))$ uniformly in $\mathrm{Pr}$.
Since $C^\beta(B_R)\hookrightarrow L^{\infty}(B_R)$ with $0<\beta<1$ 
is compact for any $R>0$,
Aubin-Lions's lemma (Lemma~\ref{lem:Aubin-Lions}) yields that, up to the extraction of a subsequence,
\begin{align}\label{eq:W-conv}
  \WW^\mathrm{Pr}\rightarrow \WW\quad \textrm{in}\;\; 
  C\big([0,T];C^\beta(B_R)\big),\, 0<\beta<1 ~~~\text{ for all}~R>0.
\end{align}
By letting $\mathrm{Pr}\rightarrow \infty$ in the equation of $\WW^\mathrm{Pr}$, 
we infer that $\WW(x,t)$ solves the equation (in the sense of distribution)
%(in the sense of distribution)
\begin{equation}\label{eq:WW}
  \partial_t \WW + u \cdot\nabla \WW = \WW\cdot\nabla u = \partial_{\WW}u, \quad\quad \WW|_{t=0}(x)=\WW_0(x).
\end{equation}
Besides, it follows from the weak-$*$ limit of $\WW^\mathrm{Pr}$ that $\WW\in L_T^\infty(W^{1,\infty}(\DD))$.
In light of Lemma~\ref{lem:sr-cond}, we conclude the global persistence of the $W^{2,\infty}$-regularity of temperature patch boundary $\partial D(t)$.

Finally, we prove \eqref{eq:limit-targ-3}. According to \eqref{eq:W-C2+gam-2} and
\eqref{es:nabu-Gam3} and the embedding 
$W^{2,p}(\mathbb{R}^3)\hookrightarrow W^{1,r}(\mathbb{R}^3)$ for some $r>3$, 
we obtain that $\WW^\mathrm{Pr}$ is uniformly bounded in $L_T^\infty(C^{1,\gamma}(\DD))$ and 
$u^\mathrm{Pr}$ is uniformly bounded in $L_T^2(W^{1,\infty}(\DD))$ with respect to 
$\PPr\in [\PPr_*,\infty)$,
which implies that $\partial_t\WW^\mathrm{Pr}$ belongs to $L_T^2(L^\infty(\DD))$ uniformly in $\mathrm{Pr}$.
Since $C^{1,\gamma_1}(B_R)\hookrightarrow L^\infty(B_R)$ ($0<\gamma_1<\gamma$) 
is compact for any $R>0$,
Lemma~\ref{lem:Aubin-Lions} guarantees that, up to the extraction of a subsequence,
\begin{align}\label{eq:W-conv2}
  \WW^\mathrm{Pr}\rightarrow \WW\quad \textrm{in}\;\; 
  C\big([0,T];C^{1, \gamma_1}(B_R)\big),\, 
  0< \gamma_1<\gamma~~~~\text{ for all}~R>0.
\end{align}
It is clear that $\WW(x,t)$ solves \eqref{eq:WW} (in the sense of distribution).
Besides, it follows from the weak-$*$ limit of $\WW^\mathrm{Pr}$ that $\WW\in L_T^\infty(C^{1,\gamma}(\DD))$.
In view of Lemma~\ref{lem:sr-cond}, we conclude the global persistence of the 
$C^{2,\gamma}$-regularity of temperature patch boundary $\partial D(t)$.
Therefore, we finish the proof of Theorem \ref{thm:limit}.
\end{proof}

\noindent\textbf{Data availability}
Data sharing is not applicable to this article as no datasets were generated or analyzed during the current study.

\noindent\textbf{Conflicts of Interest} The authors declare that they have no conflicts of interest.

\end{document}